\documentclass{amsart}
\usepackage{amscd,amsthm}
\usepackage{latexsym}
\usepackage{tikz}
\usepackage{pdfpages}
\usetikzlibrary{cd}
\usetikzlibrary{matrix}
\usepackage{capt-of}
\usepackage{graphicx}
\PassOptionsToPackage{hyphens}{url}
\usepackage{hyperref}

\DeclareFontFamily{U}{wncy}{}
\DeclareFontShape{U}{wncy}{m}{n}{<->wncyr10}{}
\DeclareSymbolFont{mcy}{U}{wncy}{m}{n}
\DeclareMathSymbol{\Sha}{\mathord}{mcy}{"58}

\usepackage{amssymb,amsmath}
\usepackage{amsfonts}
\newcounter{ctfig}

\newcommand{\Z}{{\mathbb Z}}
\newcommand{\Q}{{\mathbb Q}}

\newcommand{\C}{\mathcal{C}}

\newcommand{\NS}[1]{#1_{NS}}

\newcommand{\E}{\mathcal{E}}
\newcommand{\A}{\mathbb{A}}

\newcommand{\F}{\mathcal{F}}
\renewcommand{\P}{\mathbb{P}}
\renewcommand{\S}{\mathcal{S}}

\newcommand{\Jac}{\mathop{\rm Jac}\nolimits}

\newcommand{\JacC}{{\hbox{Jac}_{\lower.5pt\hbox{$_\C$}}}}
\newcommand{\JacF}{{\hbox{Jac}_{\lower.5pt\hbox{$_\F$}}}}

\newcommand{\Aut}{\mathop{\rm Aut}}

\newcommand{\Sym}{\operatorname{\rm Sym}\!}

\newcommand{\tG}{\tilde G}
\newcommand{\tS}{\tilde S}
\newcommand{\tSigma}{\tilde \Sigma}
\newcommand{\tA}{\tilde A}
\newcommand{\tD}{\tilde D}
\newcommand{\tE}{\tilde E}
\newcommand{\tW}{\widetilde W}

\newcommand{\tsym}{\widetilde{\Sym^2}(E)}

\theoremstyle{plain}
\newtheorem{thm}{Theorem}[section]

\newtheorem{lemma}[thm]{Lemma}
\newtheorem{prop}[thm]{Proposition}
\newtheorem{cor}[thm]{Corollary}

\theoremstyle{definition}

\newtheorem{example}[thm]{Example}
\newtheorem{remark}[thm]{Remark}
\newtheorem{defn}[thm]{Definition}

\newtheorem{conv}[thm]{Convention}
\newtheorem{cons}[thm]{Construction}

\def\Q{{\mathbb Q}}

\def\F{{\mathbb F}}
\def\Z{{\mathbb Z}}
\def\C{{\mathbb C}}
\def\O{{\mathcal O}}
\def\M{{\mathcal M}}
\def\cN{{\mathcal N}}
\def\Mbar{{\overline{\mathcal M}}}

\def\mo{{\mathrm{mod}}}

\def\H{{\mathbb H}}
\newcommand{\Pic}{\mathop{\rm Pic}}
\newcommand{\disc}{\mathop{\rm disc}}

\newcommand{\Mz}{M_{(\Z/2\Z)^4}}
\newcommand\End{\mathop{\rm{End}}}

\def\blowuptwice#1{{\widetilde{\widetilde {#1}}}} 
\begin{document}
\bibliographystyle{plain}
\bibstyle{plain}

\title[Square-curve coverings of elliptic surfaces]{Explicit coverings of families of elliptic surfaces by squares of curves}

\author{Colin Ingalls \and Adam Logan \and Owen Patashnick}
\vspace{.2 in}

\email{cingalls@math.carleton.ca}
\address{School of Mathematics and Statistics, 4302 Herzberg Laboratories,
  1125 Colonel By Drive, Carleton University, Ottawa, ON K1S 5B6, Canada\vspace{.2 in}}

\email{adam.m.logan@gmail.com}
\address{The Tutte Institute for Mathematics and Computation,
  P.O. Box 9703, Terminal, Ottawa, ON K1G 3Z4, Canada; and
  School of Mathematics and Statistics, 4302 Herzberg Laboratories,
  1125 Colonel By Drive, Carleton University, Ottawa, ON K1S 5B6, Canada  \vspace{.2 in}}

\email{o.patashnick@bristol.ac.uk}
\address{School of Mathematics, University of Bristol, Bristol, BS8 1TW, UK; Department of Mathematics, King’s College London, Strand, London, WC2R 2LS, UK; and the Heilbronn Institute for Mathematical Research, Bristol, UK.}
\date{\today}

\begin{abstract}
  We show that, for each $n>0$, there is a family of elliptic surfaces
  which are covered by the square of a curve of genus $2n+1$, and
  whose Hodge structures have an action by $\Q(\sqrt{-n})$.
  By considering the case $n=3$,
  we show that one particular family of K3 surfaces are covered by the
  squares of curves of
  genus~$7$.  Using this, we construct a correspondence between the
  square of a
  curve of genus~$7$ and a general K3 surface in $\P^4$ with $15$
  ordinary double points up to a map of finite degree of K3 surfaces.
  This gives an explicit proof
  of the Kuga-Satake-Deligne correspondence for these K3 surfaces and any K3
  surfaces related to them by maps of finite degree,
  and further, a proof of the Hodge conjecture for the squares of these
  surfaces.  We conclude that the motives of these surfaces
  are Kimura-finite.  Our analysis gives a birational equivalence
  between a moduli space of curves with additional data and the moduli space
  of these K3 surfaces with a specific elliptic fibration.  
\end{abstract}

\maketitle

\section{Statement of results}\label{sec:results}
The work of Kuga and Satake \cite{kugasatake} and Deligne
\cite{deligne} demonstrates that, if one
assumes the Hodge conjecture, every K3 surface over $\C$ is  of abelian Hodge 
type, i.e., there exists an algebraic correspondence between any K3
surface and an associated abelian variety.  This would imply that the
motive of any such K3 surface is abelian, and hence that the variety, by
an appropriate Torelli-type argument, is completely determined by linear data associated to it.

More specifically (see, e.g.,   \cite{vangeemen}, section 10), given a polarized Hodge structure $V$ of weight 2 with $\dim V^{(2,0)}=1$, there exists an abelian variety $A$, the {\it Kuga-Satake variety} of $V$, such that $V$ is a sub-Hodge structure of $H^2(A\times A, \Q)$.  When there is another variety $X$ with $V\hookrightarrow H^2(X,\Q)$, the Hodge conjecture on $A^2\times X$ predicts the existence of an algebraic cycle $Z\subset A^2\times X$, the {\it Kuga-Satake-Deligne correspondence}, which realizes the morphism of Hodge structures 
\begin{align*}
H^2(A^2, \Q)\rightarrow V \rightarrow H^2(X,\Q).
\end{align*}
This morphism is particularly nice when $X$ is a K3 surface, or more generally when $\dim H^{(2,0)}(X, \Q)=1$.  In this case, $V$ can be identified with the orthogonal complement of the N\'eron-Severi group of $X$, i.e. $$H^2(X,\Q)=V\oplus NS(X)_{\Q}$$
and $Z$ induces an isomorphism of $V\subset H^2(A^2, \Q)$ with $V\subset H^2(X,\Q)$.

In \cite{paranjape}, Paranjape gives a method for computing an explicit cover, by the squares of curves of genus 5, of K3 surfaces $X$ of Picard rank 16 with an action of $\Q(\sqrt{-1})$ on their Hodge lattices, 
and shows that this product computes the abelian variety predicted by Kuga and Satake.
In turn this not only gives a constructive proof of the Kuga-Satake-Deligne correspondence in this special case, as one can take $Z$ to be the square of the aforementioned curve of genus 5, but also a proof of the Hodge conjecture for $X\times X$ (\cite{schlickewei}).
Note that to construct the curve and this cover of his given K3, Paranjape is first forced to construct intermediate curves and the surfaces their squares cover, so the construction is more delicate than it might at first appear.  The question of which other families of K3 surfaces his method can be generalized to is left open.

In \cite{g-s}, Garbagnati and Sarti characterize Picard lattices of K3 surfaces in $\P^{m}$ with $15$ nodes.  The simplest examples of such surfaces, namely, those K3s given by a double cover of the plane branched along six lines, were already known classically and satisfy the condition studied by Paranjape above.  The family of K3 surfaces in $\P^4$ with $15$ ordinary double points seems to have first appeared in their paper \cite{g-s}.

In this paper we show that a variant of Paranjape's method does
generalize: indeed, for each $n$, we show
(Theorem \ref{thm:motive-finite-ell-surfs})
that there are families of
elliptic surfaces which are covered by the square of a curve of
genus $2n+1$.
The properties of these surfaces are listed in Definition
\ref{def:higher-moduli}.
In particular, applying the work of Garbagnati
and Sarti, we show the following results:
\begin{enumerate}
\item 
  up to maps of K3 surfaces of finite degree, there is an explicitly computable
  correspondence between the square of a curve of genus 7 and 
  a K3 surface $Y$ 
  in $\P^4$ with $15$ ordinary double points (Proposition \ref{prop:explicit-corr} and Proposition \ref{prop:g-s-isogeny});
\item the
construction of the curve is unique in the sense of Theorem
\ref{thm:main-moduli-intro} below; 
\item the construction does in fact produce
  the family of abelian varieties predicted by Kuga and Satake and constructively compute the Kuga-Satake-Deligne correspondence
  (Propositions \ref{prop:hodge-theory} and \ref{prop:g-s-isogeny});
  and, moreover,
\item the construction gives a proof of
  the Hodge conjecture for $Y\times Y$ and the square of any K3 surface
  $X$ for which there is a sequence of K3 surfaces
  $X = X_0, X_1, \dots, X_n = Y$ such that for all $1 \le i \le n$
  there is either a rational map of finite degree from
  $X_{i-1}$ to $X_i$ or from $X_i$ to $X_{i-1}$
  (Theorem \ref{thm:hodge-conj-l2-surfaces}, Corollary \ref{thm:hodge-conj-l-surfaces}, and  Theorem \ref{thm:hodge-conj-gs-surfaces}).
\end{enumerate}
It would be interesting to use the methods of Schlickewei \cite{schlickewei}
to prove the Hodge
conjecture for $A^2\times Y'$ as well, where $Y'$ is any K3 surface related
to $Y$ by maps of finite degree.

Interestingly, our variant of Paranjape's construction only
yields an explicit covering of K3 surfaces by squares of curves in the
cases $n=2$, $3$, $4$, and $6$.  Numerical calculations prove that
the case $n=5$ does not produce a cover of K3 surfaces (Remark
\ref{rem:various-n}), and we suspect that the cases $n=2$, $3$, $4$, and $6$
are the only $n$ for which the elliptic surfaces so covered admit maps to K3
surfaces,
and hence for which we can constructively prove the
Kuga-Satake-Deligne correspondence.
Indeed, in the cases $n=2$, $4$ we recover special cases of Paranjape's
construction (see Remark \ref{rem:n-2} and Remark \ref{rem:n-4}), and
we expect (Remark \ref{rem:various-n}) that the case $n=6$ recovers a
special case of our construction (Theorem \ref{thm:main-cover} below).
 
We have focused on the special case of K3 surfaces in part because they are well known surfaces.  We are cognizant of the fact that there may be other surfaces that have interesting geometric or arithmetic properties that come up in (generalizations of) our construction.  In particular, 
it would be interesting to see if the elliptic surfaces we construct
for arbitrary values of $n$ encode any useful arithmetic data.

We now state our main results with some preliminary
definitions.  Throughout the paper we work over a field $k$ of
characteristic not $2$ or $3$.
\begin{defn}\label{def:our-pic-short}
  Let $\H$ be the lattice spanned by $x, y$ with $x^2 = 0$, $x \cdot y = 1$,
  $y^2 = -2$; we often refer to $\H$ as the {\em hyperbolic plane}. 
  Let $L_1$ be the Picard lattice of a K3 surface with an elliptic fibration
  with one $\tD_4$ and nine $\tA_1$ fibres and Mordell-Weil group
  $\Z \oplus (\Z/2\Z)^2$, such that there is a generator of the Mordell-Weil
  group modulo torsion that passes through the zero component of the
  $\tD_4$ fibre and the nonzero component of every $\tA_1$ fibre.
  Let $\Lambda_1$ be the essential lattice of $L_1$, namely $L_1$ modulo the
  sublattice (isometric to $\H$)
  spanned by the classes of the fibre and the zero section.
\end{defn}

\begin{defn}\label{def:isogeny}
  Let $X, X'$ be K3 surfaces.  We say that $X$ and $X'$ are {\em isogenous}
  if there is a finite sequence $X = X_0, X_1, \dots, X_n = X'$ of K3
  surfaces such that, for all $0 \le i < n$, there exists a rational map
  of finite degree from $X_i$ to $X_{i+1}$ or from $X_{i+1} \to X_i$.
  For a prime $p$ we say that $X$ and $X'$ are {\em $p$-isogenous} if
  there exist elliptic fibrations $\pi: X \to \P^1$ and $\pi': X \to \P^1$
  such that the generic fibres are $p$-isogenous elliptic curves; this
  implies that there exist maps of degree $p$ from $X$ to $X'$ and from
  $X'$ to $X$.  In the case $p = 2$, the existence of a rational $2$-torsion
  section on $X$ gives an involution~$\iota: X \to X$ by which the quotient
  is birationally equivalent to $X'$; this is sometimes known as a 
  {\em van Geemen-Sarti involution}.
\end{defn}

\begin{remark}\label{rem:what-is-isogeny}
Sometimes two K3 surfaces $X, X'$ over $\C$ are called
   isogenous under the strictly weaker condition that there is an isomorphism
   of integral or rational Hodge structures
   $\phi: H^2(X,\Q) \stackrel{\cong}\to H^2(X',\Q)$.
   By results of Mukai (see for example the introduction of
   \cite{buskin}) it is known that
   the class of any Hodge isometry (under this weaker definition) is
   algebraic.  This in particular implies the Hodge conjecture for
   self-products of K3 surfaces over $\C$ with complex multiplication (and
   elliptic surfaces over $\C$ with complex multiplication more generally by
   the work of Nikulin). However, the results described in this paper are
   stronger and more general.  In particular, they do not depend on the
   surface in question being defined over the complex numbers, are based
   on explicit constructions, and have applications (such as to
   motive-finiteness) not deducible from the above results.
   Thus in this paper we will work only with isogenies in the strong
   sense, even though some of our results extend to this more general
   context.
\end{remark}

\begin{thm}\label{thm:main-cover} A general K3 surface whose Picard lattice
  has a primitive sublattice isometric to $L_1$ is covered by the square of
  a curve of genus~$7$.
\end{thm}

We will see in Section \ref{sec:k3s} that this implies the following:

\begin{thm}\label{thm:chapter-6-family}
  A general K3 surface in $\P^4$ with $15$ ordinary double points
  is isogenous to a K3 surface $K$ which realizes the Kuga-Satake-Deligne correspondence between $K$ and the square of a curve of genus 7 constructed in Theorem  \ref{thm:main-cover}.
\end{thm}

Theorem \ref{thm:main-cover} can also be expressed in terms of a map of moduli spaces as follows:
\begin{defn}\label{def:mc}
  For a curve $C$ of genus~$1$, let $T_2(C)$ be the group of translations of
  $C$ of order $1$ or $2$.
  Let $\M_C$ be the moduli space that parametrizes curves
  $C_1$ of genus~$1$ with an unordered set of $4$ distinct points
  $\{p_1, \dots, p_4\}$, together with $O \in C_1/T_2(C_1)$ such that
  $4O \sim p_1 + p_2 + p_3 + p_4$ and an irreducible unramified cover of
  $C_1$ of degree $3$.
  Let $\M_{K,1}$ be the moduli space of marked
  $L_1$-polarized K3 surfaces as in \cite[Theorem 9]{ht}, together with
  a choice of elliptic fibration with one reducible fibre of type
  $\tD_4$ and $9$ of type $\tA_1$ and Mordell-Weil group $\Z \oplus (\Z/2\Z)^2$,
  such that a generator of infinite order has intersection~$1$ with the
  zero section and $0$ with the sections of order $2$.  The choice of
  the elliptic fibration is equivalent to the choice of an embedding of
  $\H$ into $\Pic L$ such that the orthogonal complement is isometric
  to $\Lambda_1$.
\end{defn}

\begin{remark}\label{rem:why-c1-mod-t2-c1}
  Although $C_1$ and $C_1/T_2(C_1)$ are isomorphic varieties, it is necessary
  to distinguish them here, because replacing $O$ by a point $t_2(O)$ for
  $t_2$ a nonidentity element of $T_2(C_1)$ would ultimately lead to the same
  K3 surface.
\end{remark}

\begin{thm}\label{thm:main-moduli-intro}
  Given a general point $P$ of $\M_C$, there is an explicit
  construction of a point $f(P)$ of $\M_{K,1}$, such that the data
  parametrizing $f(P)$ can be built from the data parametrizing $P$,
  and vice versa.  In particular, these constructions give a modular
  birational equivalence $f: \M_C \to \M_{K,1}$.
\end{thm}

Theorems \ref{thm:main-cover} and
\ref{thm:main-moduli-intro} follow directly from Remark
\ref{rem:m3-mc} and Theorem \ref{thm:main}, while Theorem
\ref{thm:chapter-6-family} follows from Propositions \ref{prop:hodge-theory}
and \ref{prop:g-s-isogeny}.

Our work shows that generic members of our families of elliptic surfaces are covered
by the square of a curve.
It follows that they, and generic members of any family of K3 surfaces that are related
(via isogenies and correspondences) to these families of elliptic surfaces,
also have Kimura-finite (and hence Schur-finite) motive. (See \cite{deligne2}, \cite{kimura} or \cite{mazza} for a definition of Kimura finiteness. See for example  \cite{laterveer} for a discussion of motives known to be Kimura-finite.)
Since the Kimura-finiteness of a motive is invariant under maps of K3 surfaces
of finite degree, as in the proof of \cite[Theorem 3.1]{laterveer}, this
implies Kimura-finiteness and Schur-finiteness for a large class of K3 surfaces.  In this paper, we will refer to this property as {\em motive-finiteness}.

The property of motive-finiteness has received some attention in previous
work, but for K3 surfaces it is not well understood.
Certainly every K3
surface isogenous to an algebraic Kummer surface is motive-finite, but such
Kummer surfaces all have Picard number at least $17$.  The work of
Paranjape \cite{paranjape} and Laterveer \cite{laterveer} give some examples
with Picard number $16$.  In addition, Garbagnati and Penegini \cite{gp}
studied the families of K3 surfaces
that are of the form $(C_1 \times C_2)/G$, where $C_1, C_2$ are curves and
$G$ acts diagonally.  We remark that our construction is different from theirs
in that our K3 surfaces are of the form $((C_1 \times C_2)/G_1)/G_2$, where
$(C_1 \times C_2)/G_1$ is not a K3 surface.

One new feature of our work is that we show that a K3 surface in our
family arises from an essentially unique curve.  This means that the
fields of moduli of the curve and the K3 surface are the same,
a fact that is useful in arithmetical applications.  We expect that it
would be possible to prove an analogous result for the construction of
\cite{paranjape}.

The plan for the paper is as follows.  In Section
\ref{sec:constr-moduli} we define the various moduli spaces of curves
and elliptic surfaces, including K3 surfaces, that we will use in the
paper, and indicate some birational equivalences between them. In
Section \ref{sec:constr-curves} we give an alternative, and explicit,
construction of the birational equivalence for $n=3$; namely, starting
with a rational curve together with a cover of degree $3$ and a point
of ramification of the cover, we construct a curve of genus~$7$ whose
square covers a K3 surface. In Section \ref{sec:hodge},
we reinterpret the construction of Section \ref{sec:constr-curves} in
terms of Hodge theory (Proposition \ref{prop:same-qf}), and in
particular show that in characteristic zero we have constructed the
Kuga-Satake variety for our K3 surfaces (Proposition
\ref{prop:hodge-theory}).  In Section \ref{sec:scope} we explore the
limits of our construction from Section \ref{sec:constr-curves}. In
particular, we show two results by explicit calculation.
First, not every K3 of rank 16 whose
Picard lattice has determinant $-12t^2$
is coverable by the square of a curve using our construction (Example
\ref{ex:no-isog}). Second, our construction, though it induces a
covering of some Picard rank 17 and 18 non-Kummer K3 surfaces not
previously known to be covered by squares of curves, does not
determine coverings for all Picard rank 17 and 18 K3s not isomorphic
to Kummers.  Finally, in Section \ref{sec:k3s}, we discuss Picard
lattices of various families of K3 surfaces, and in particular relate
our construction from Section \ref{sec:constr-curves} to K3 surfaces
of degree 6 with 15 singularities of type $A_1$ (Proposition \ref{prop:explicit-corr}, Proposition \ref{prop:g-s-isogeny}, and Theorem \ref{thm:hodge-conj-gs-surfaces}).

\vspace{.1in}
\noindent {\bf Acknowledgements:} C. Ingalls was partially supported by an NSERC Discovery Grant. A. Logan thanks the Tutte Institute for Mathematics and Computation for its support of his external research. A. Logan and O. Patashnick thank the Institut Henri Poincar\'e and the organizers of the Reinventing Rational Points program for their support and hospitality 
during their spring 2019 stay. We thank the referee for a detailed and helpful
report, Frank Calegari and Cec\'\i lia Salgado for interesting discussions,
and a preliminary reviewer
for a very helpful comment vis-\`a-vis the Hodge conjecture.

\section{The construction in terms of moduli spaces}\label{sec:constr-moduli}
In this section we introduce some moduli spaces of curves and of elliptic
surfaces, including K3 surfaces,
with additional structure and indicate some birational equivalences among them.
Our main goal is to pave the way for the construction in
Section \ref{sec:constr-curves}, in which we will show how to use a small
amount of starting data on a rational curve, namely a cover of degree $3$
and a point of ramification of the cover, to construct a curve of genus~$7$
whose square covers a K3 surface.  This will give an alternative birational
equivalence between two of our moduli spaces, and will show that the K3
surfaces in question have Kimura-finite motive and that the Kuga-Satake-Deligne 
correspondence in this case is realized by a correspondence between the
surface and the square of an abelian variety.

\subsection{Moduli spaces of elliptic K3 surfaces}
We begin by defining the moduli spaces of K3 surfaces that are of interest.

\begin{defn}\label{def:mk1}
  Let $\M_{K,1}$ be the moduli space of K3 surfaces together
  with an elliptic fibration $\phi_1$ with a reducible fibre of type $\tD_4$
  and nine of type $\tA_1$, full level-$2$ structure, and a section of infinite
  order that passes through the zero component of the $\tD_4$ fibre and
  the nonzero components of all $\tA_1$ fibres while meeting the $0$ section
  once.  Let $\M'_{K,1}$ be the $6$-to-$1$ cover of $\M_{K,1}$ that parametrizes
  the same data but with an additional choice of labelling of the sections of
  order $2$.
\end{defn}

\begin{remark}\label{rem:two-notations}
  When we discuss K3 surfaces in terms of their Picard lattices, as mostly
  in this section 
  and in Section~\ref{sec:k3s}, we use the ADE
  notation for reducible fibres, since the Picard lattice does not contain
  the more refined information present in the Kodaira classification.
  On the other hand, in studying a K3 surface as a variety, as we will in
  Section~\ref{sec:constr-curves}, it is better to
  keep track of this information.
  We recall the correspondence between the two notations, as in
  \cite[Table I.4.1]{miranda}:
  
\centering  \begin{tabular}{|l|l|}
    \hline
    $ADE$ type & Kodaira symbol \cr\hline
    --- & $I_0, I_1, II$ \cr\hline
    $\tA_1$ & $I_2, III$ \cr\hline
    $\tA_2$ & $I_3, IV$ \cr\hline
    $\tA_n$ $(n>2)$ & $I_{n+1}$ \cr\hline
    $\tD_n$ & $I_{n-4}^*$ \cr\hline
    $\tE_6$ & $IV^*$ \cr\hline
    $\tE_7$ & $III^*$ \cr\hline
    $\tE_8$ & $II^*$ \cr\hline
  \end{tabular}
\end{remark}

\begin{remark}\label{rem:partition-a1}

  Now let us perform the standard calculation to
  determine the intersection of the torsion sections with
  the reducible fibres and the reducible fibres on the quotients
  of the surface by the $2$-torsion translation automorphisms.
  The $\tD_4$ fibre goes to a $\tD_4$ fibre, and the $\tA_1$ fibres go to
  $\tA_3$ fibres if the section passes through the zero component and
  singular irreducible fibres otherwise: see \cite[Table~1]{dokchitsers}.

  As in \cite[Lemma~7.3]{ss} each $2$-torsion section passes through a
  different nonzero component of the $\tD_4$ fibre.
  The quotient is a K3 surface, so its Euler characteristic is $24$.
  It follows that each $2$-torsion section
  passes through the zero component of exactly three $\tA_1$ fibres.
  (Alternatively, this follows from \cite[Lemma~1.15]{cz}: the height
  of a torsion section must be $0$, and from \cite[Table~1.14]{cz}
  it is equal to $4 - 1 - n_1/2$, where $n_1$ is the number of $\tA_1$
  fibres such that the section passes through the nonzero component.)
  By considering the group structure on such a fibre, we see that 
  two of the $2$-torsion sections pass through the nonzero component
  and one through the zero component.  Hence we have partitioned
  the $\tA_1$ fibres into three sets of size $3$.
  For $\M'_{K,1}$ this becomes an ordered partition.
\end{remark}

We now define another moduli space $\M'_{K,2}$ of elliptic surfaces and
show that it is birationally equivalent to $\M'_{K,1}$.
In terms of lattices, we may describe the situation as follows.
The essential lattices $\Lambda_1, \Lambda_2$ of the two types of
fibration are in the same genus.  In addition, the images
of the automorphism groups of $\Lambda_1, \Lambda_2$ in the automorphism
group of the discriminant group are conjugate, so the two types of fibration
are determined by each other; cf.~\cite[Theorem~2.8]{fv}.  For our purposes,
we need to make the equivalence explicit; it is not enough to know that the
fibrations of the two types are in canonical bijection.

\begin{defn}\label{def:mk2}
  Let $\M_{K,2}$ be the moduli space of K3 surfaces with a specified elliptic
  fibration with three $\tD_4$ fibres and one $\tA_2$ fibre (and generically
  trivial Mordell-Weil group).
  Let $\M'_{K,2}$ be the cover of $\M_{K,2}$ that
  also keeps track of a labelling of the $\tD_4$ fibres.  We will denote
  a point of $\M'_{K,2}$ by $(S,\phi_2,F_1,F_2,F_3)$ where $\phi_2$
  is the fibration
  and the $F_i$ are the $\tD_4$ fibres in order.
\end{defn}

\begin{remark}\label{rem:moduli-dim}
  The dimensions of $\M_{K,1}$ and $\M_{K,2}$ are $4$,
  because they parametrize K3 surfaces of Picard number $16$
  together with a finite amount of additional data.  The same holds
  for the $\M'_{K,i}$, because these are finite covers of
  the $\M_{K,i}$.
\end{remark}

\begin{prop}\label{prop:mk1-bir-mk2}
  There is an $\S_3$-equivariant birational equivalence between
  $\M'_{K,1}$ and $\M'_{K,2}$.
\end{prop}

\begin{proof} First let $(S,\phi_1,T_1,T_2,T_3,G)$ be the data of a point of
  $\M'_{K,1}$, where the $T_i$ are the labelled torsion sections and $G$ is
  the generator of infinite order.
  Let $D_i$ be the set of curves on $S$ consisting of $T_i$,
  the zero components of $\tA_1$ fibres of $\phi_1$ that $T_i$ meets, and the
  component of the $\tD_4$ that $T_i$ meets.  Since the $\tD_4$ and
  $\tA_1$ are distinct fibres of $\phi_1$, the curves in them are disjoint,
  and so there is a $\tD_4$ fibre supported on $D_i$ whose nonreduced
  component is $T_i$.

  We now find an $\tA_2$ fibre.  Indeed, consider the zero component of
  the $\tD_4$ fibre of $\phi_1$ together with the curves $G, -G$.  From
  our description of the generator in Proposition \ref{prop:mw-ep},
  which pulls back to $S$, we see that $G$ and its inverse meet the zero
  section $0_{\phi_1}$ once each and pass through the nonzero components of all
  $\tA_1$ fibres and the zero component of the $\tD_4$ fibre.
  Let us denote the height pairing on sections by $(x,y)$ and
  the intersection pairing on the surface by $x \cdot y$.
  Applying the formulas of \cite{cz},
  in particular Lemma 1.18 and Table 1.19, we find that
  $(G,G) = -(0_{\phi_1}-G)^2 - \sum D(G) = 4+2-9(1/2) = 3/2$, where the
  $D(G)$ are the local correction terms of \cite{cz}, 
  so that $(G,-G) = -3/2$.  Then
  we have $(G,-G) = -(0_{\phi_1}-G)\cdot(0_{\phi_1}-(-G)) - \sum G \cdot D(-G)$ and so
  $-3/2 = 2 + 1 + 1 - G \cdot -G - 9/2$ and $G \cdot -G = 1$.
  It is easily checked that the Picard classes of the three $\tD_4$ fibres and
  the $\tA_2$ fibre just mentioned are equal, so there is a genus-$1$
  fibration with these reducible fibres.
  These fibres meet the central component of the $\tD_4$ fibre of $\phi_1$
  in a single point, so it is an elliptic fibration.  This gives a map
  $\M'_{K,1} \dashrightarrow \M'_{K,2}$.

  Conversely, suppose given $(S,\phi_2,F_1,F_2,F_3)$, a point of $\M'_{K,2}$.
  To find its image in $\M'_{K,1}$, first consider the zero section and the
  zero components of the reducible fibres of $\phi_2$: these constitute a
  $\tD_4$ configuration~$D$.  This gives an elliptic fibration $\phi_1$,
  since there are
  sections such as the nonzero components of the $\tA_2$ fibre.
  For a general point of $\M'_{K,2}$, the surface $S$ has
  Picard lattice isometric to $\H + D_4^3 + A_2$, and one computes that
  the orthogonal complement of the lattice spanned by $D$ and a section
  has root sublattice $D_4 + A_1^9$ and that this root lattice is
  embedded in its saturation with quotient $(\Z/2\Z)^2$.

  The reduced
  nonzero components of the $\tD_4$ fibres of $\phi_2$ are disjoint from
  $D$, so they are vertical for $\phi_1$.  Since any two of them are disjoint
  and no reducible fibre of $\phi_1$ other than the $\tD_4$ contains two
  disjoint curves, they must be components of the nine different $\tA_1$
  fibres of $\phi_1$.  Further, the nonreduced components of the
  $\tD_4$ fibres of $\phi_2$ meet $D$ once, so they are sections, and it is
  easily computed that the difference of any two is of order $2$.
  The labelling of these is given by the labelling of the $\tD_4$ fibres of
  $\phi_1$.  Finally, the section of infinite order and its inverse are
  given by the two nonzero components of the $\tA_2$ fibre.  This is not
  really a choice, because there is an automorphism of $S$
  that takes $G$ to $-G$ and preserves $\phi_1$ and the $T_i$.  Thus we have
  constructed a map $\M'_{K,2} \dashrightarrow \M'_{K,1}$, and it is clear that
  this is a birational inverse to the map $\M'_{K,1} \dashrightarrow \M'_{K,2}$
  constructed above.
\end{proof}

We illustrate the construction in figures.
In Figure~\ref{fibrationPhi}, we have drawn curves in an elliptic fibration $\phi_1: S \to \P^1$. The $0$-section $0_{\phi_1}$ is the black horizontal curve, and we have also indicated the $0$-section of the fibration $\phi_2:S \to \P^1$
as $0_{\phi_2}$.
We have indicated the three $\tD_4$ fibres of $\phi_2$ with their
central components $T_1,T_2,T_3$ in green, blue and red respectively.
The $\tA_2$ fibre has been drawn in orange.
  In addition, the nonzero components of the 9 $\tA_1$ fibres have been drawn in black.
  
\begin{figure}
  \caption{Elliptic fibration in $\M'_{K,1}$}\label{fibrationPhi}
  \centering
\begin{equation*}
\begin{tikzpicture}[xscale=0.7]

\draw[black] plot [smooth] coordinates {(-1,0) (9.5,0) (12,-3.5)};
\node[black] at (-1.5,0) {$0_{\phi_1}$};
\draw[orange] (0,-4) -- (0,5);

\draw[orange]  plot [smooth] coordinates {(-1,-3) (9.5,-3) (12,-0.5)};
\node[orange] at (-1.5,-3) {$G$};

\draw[orange]  (-1,-2.3) -- (12,-2.3);
\node[orange] at (-1.5,-2.3) {$-G$};

\draw[black] (-1,2) -- (5,7);
\draw[black] (-1,1.9) -- (5,6.9);
\node[black] at (-1.5,2) {$0_{\phi_2}$};

\draw[green] (1,6) -- (1,1.7);
\draw[green] (0.7,2) -- (4.2,2);
\node[green] at (1.5,2.2) {$T_1$};

\draw[green] (2,2.2) -- (2,-2);
\draw[green] (3,2.2) -- (3,-2);
\draw[green] (4,2.2) -- (4,-2);

\draw[blue] (2,7) -- (2,2.7);
\draw[blue] (1.7,3) -- (7.2,3);
\node[blue] at (3.5,3.2) {$T_2$};

\draw[blue] (5,3.2) -- (5,-2);
\draw[blue] (6,3.2) -- (6,-2);
\draw[blue] (7,3.2) -- (7,-2);

\draw[red] (3,8) -- (3,3.7);
\draw[red] (2.7,4) -- (10.2,4);
\node[red] at (5.5,4.2) {$T_3$};

\draw[red] (8,4.2) -- (8,-2);
\draw[red] (9,4.2) -- (9,-2);
\draw[red] (10,4.2) -- (10,-2);

\draw[black] plot [smooth] coordinates {(2,-3.5) (1.7,-2) (2.3,-1) (1.7,-0.5)};
\draw[black] plot [smooth] coordinates {(3,-3.5) (2.7,-2) (3.3,-1) (2.7,-0.5)};
\draw[black] plot [smooth] coordinates {(4,-3.5) (3.7,-2) (4.3,-1) (3.7,-0.5)};
\draw[black] plot [smooth] coordinates {(5,-3.5) (4.7,-2) (5.3,-1) (4.7,-0.5)};

\draw[black] plot [smooth] coordinates {(6,-3.5) (5.7,-2) (6.3,-1) (5.7,-0.5)};

\draw[black] plot [smooth] coordinates {(7,-3.5) (6.7,-2) (7.3,-1) (6.7,-0.5)};
\draw[black] plot [smooth] coordinates {(8,-3.5) (7.7,-2) (8.3,-1) (7.7,-0.5)};
\draw[black] plot [smooth] coordinates {(9,-3.5) (8.7,-2) (9.3,-1) (8.7,-0.5)};
\draw[black] plot [smooth] coordinates {(10,-3.5) (9.7,-2) (10.3,-1) (9.7,-0.5)};

\node at (6,-3.7) {$9\, \tA_1 \mbox{ fibres }$};

\node at (8,7) {$\downarrow\phi_1$};








\end{tikzpicture}
\end{equation*}
\end{figure}

In Figure~\ref{fibrationPi}, we have drawn curves in an elliptic fibration $\phi_2: S \to \P^1$.
We have indicated the three $\tD_4$ fibres $F_1,F_2,F_3$ with their
central components $T_1,T_2,T_3$ in green, blue and red respectively.
The $\tA_2$ fibre has been drawn in orange.
The $0$-section, $0_{\phi_2}$, is the black horizontal curve.
Note that the sets of curves of the same colour in each figure are equal;
this reflects the bijection between the elliptic fibrations of each type on $S$
illustrated in the two figures.

\begin{figure}
\caption{Elliptic fibration  in $\M'_{K,2}$}\label{fibrationPi}
  \centering
\begin{equation*}
\begin{tikzpicture}[xscale=0.6,yscale=0.8]
\draw[black] (-1,0) -- (13,0);
\node[black] at (-1.5,0) {$0_{\phi_2}$};

\draw[green] (0,-1) -- (0,5);

\draw[green] (-1,3) -- (4,8);
\draw[green] (-1,2.9) -- (4,7.9);
\node[green] at (-1.5,3) {$T_1$};
\draw[green] (0.5,5.5) -- (2,4);
\draw[green] (1.5,6.5) -- (3,5);
\draw[green] (2.5,7.5) -- (4,6);
\node[green] at (1.5,7.5) {$F_1$};

\draw[blue] (4,-1) -- (4,5);

\draw[blue] (3,3) -- (8,8);
\draw[blue] (3,2.9) -- (8,7.9);
\node[blue] at (2.5,3) {$T_2$};
\draw[blue] (4.5,5.5) -- (6,4);
\draw[blue] (5.5,6.5) -- (7,5);
\draw[blue] (6.5,7.5) -- (8,6);
\node[blue] at (5.5,7.5) {$F_2$};

\draw[red] (8,-1) -- (8,5);

\draw[red] (7,3) -- (12,8);
\draw[red] (7,2.9) -- (12,7.9);
\node[red] at (6.5,3) {$T_3$};
\draw[red] (8.5,5.5) -- (10,4);
\draw[red] (9.5,6.5) -- (11,5);
\draw[red] (10.5,7.5) -- (12,6);
\node[red] at (9.5,7.5) {$F_3$};

\draw[orange] (12,-1) -- (12,5);

\draw[orange] (11,4) -- (14,4);
\node[orange] at (10.7,3.7) {$G$};

\draw[orange] (11,2) -- (14,5);
\node[orange] at (10.5,2) {$-G$};

\node at (-1,7) {$\downarrow \phi_2$};









\end{tikzpicture}
\end{equation*}
\end{figure}
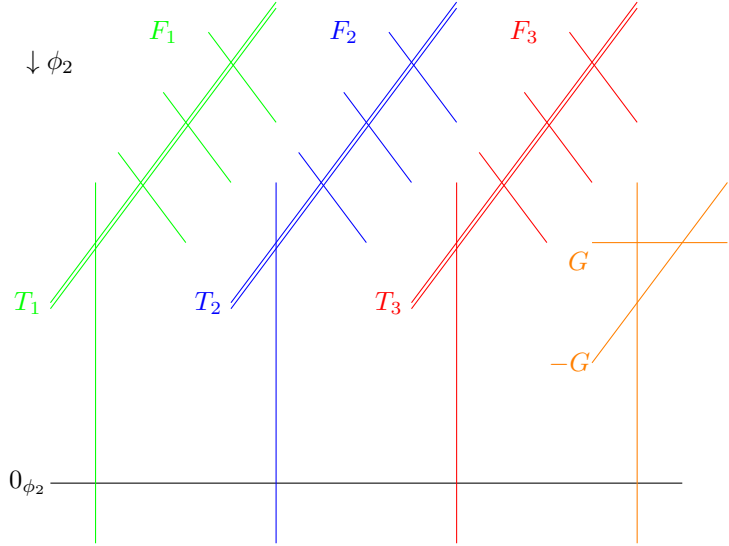

\begin{prop}\label{prop:deg-3-ramified}
  Let $\phi_1$ be the fibration associated to a point of $\M'_{K,1}$,
  let $\sigma_0$ be its zero section, and let $\phi_2$ be the fibration
  constructed in the proof of Proposition \ref{prop:mk1-bir-mk2}.  Then the
  intersection of $\sigma_0$ with a fibre of $\phi_2$ is $3$, and $\phi_2$
  restricted to $\sigma_0$ is ramified at the intersection of $\sigma_0$
  with the $\tA_2$ fibre of $\phi_2$.
\end{prop}

\begin{proof} The first statement is a routine calculation in light of the
  description of the reducible fibres of $\phi_2$.  As for the second,
  two of the components of the $\tA_2$ fibre are a generating section of
  $\phi_1$ and its inverse.  These meet the zero section at the same point:
  on the fibre of $\phi_1$ where the generating section meets the zero section,
  translation by the generating section acts as the identity, and applying
  its inverse to the zero section gives the intersection with the negative
  of the generating section.
\end{proof}

\subsection{Moduli spaces of covers of curves of genus $0$}
We now shift our attention to some moduli spaces of covers of curves of genus
$0$ that will turn out to be connected to $\M'_{K,1}$.
We use the standard notation $\Mbar_{g,n}$ \cite{knudsen}
for the usual compactification
of the moduli space of stable curves of genus~$g$ with $n$ marked
points.
We now introduce a subvariety of $\Mbar_{0,10}$ whose quotients by finite groups
we will relate to $\M_{K,1}$ and $\M'_{K,1}$, as well as a moduli space of
covers of curves of genus~$0$ that will be connected to $\M'_{K,1}$.

\begin{conv}\label{conv:m04-bar}
  We view $\Mbar_{0,4}$ not as an abstract $\P^1$ but rather as a $\P^1$ with
  three marked points $b_1, b_2, b_3$, namely the boundary divisors; as a
  result $\Aut \Mbar_{0,4}$ is trivial.  Thus when we speak of a moduli space
  of covers of $\Mbar_{0,4}$, we do not take the quotient by automorphisms of
  the target, but only of the source.  In addition, there is a natural
  isomorphism from $\Mbar_{0,4}$ to the modular curve $X(2)$ parametrizing
  elliptic curves with a basis for the $2$-torsion given
  on $\M_{0,4}$ by taking a set of $4$ points to the double cover of $\P^1$ branched there,
  with the first marked point going to the origin and the second and third
  to the basis.  In turn, there is a $6$-to-$1$ map from $X(2)$ to 
  the $j$-line $\Mbar_{1,1}$ given by forgetting the basis.
  It is a Galois cover of curves with Galois group $\S_3$.
  The support of the fibre of the composed map $\Mbar_{0,4} \to \Mbar_{1,1}$ 
  above $\infty$ in the $j$-line is the set of boundary points
  of $\Mbar_{0,4}$; in particular these are points of ramification of the cover.
\end{conv}

\begin{defn}\label{def:v-md}
  Let $U \subset \Mbar_{0,10}$ be the subvariety of $\Mbar_{0,10}$ parametrizing
  stable curves $C_0$ of genus~$0$ with marked points
  $p_1$ and $q_{ij}$ for $1 \le i,j \le 3$ such that there is a $3$-to-$1$ cover
  $\phi: C_0 \to \Mbar_{0,4}$ for which the $q_{ij}$ constitute the fibre above
  the boundary points $b_i$, while $(p_1,q,q)$ is a scheme-theoretic fibre
  for some $q \in C_0$.  Note that $\S_3 \wr \S_3$ acts on $U$ by permuting
  the $q_{ij}$.  Let $V = U/\S_3 \wr \S_3$ and let $V' = U/\S_3^3$.

  Fix a positive integer $d$.  Let $\M_d$ be the moduli space parametrizing
  $d$-to-$1$ covers $\phi: C_0 \to \Mbar_{0,4}$ together with a point
  $p \in \Mbar_{0,4}$.
\end{defn}

\begin{remark}\label{rem:moduli-dim-uv} The dimension of
  $U$ is $4$, because a point of $U$ is specified by a $3$-to-$1$ cover
  $C_0 \to \Mbar_{0,4}$, which is determined by a choice of
  two sections of $\O_{C_0}(3)$ up to scaling
  and the action of $PGL_2 = \Aut C_0$,
  together with a finite amount of additional data.
  The same follows for $V, V'$.  In general we have $\dim \M_d = 2d-1$.
\end{remark}

First we will describe a relation between points of $\Mbar_{0,4}$ and
certain rational elliptic surfaces.  By taking covers, we will relate this
to $\M_d$.  We start by recalling the standard description of
$\Mbar_{0,5}$ \cite[Exercise 1.3.10, Lemma 2.1]{cavalieri}.

\begin{defn}\label{def:m-0-5-bar}
  We identify $\Mbar_{0,5}$ with $\P^2$ blown up in four general
  points $p_0, \dots, p_3$.  To do so, let $E_i$ be the exceptional divisor
  above $p_i$ and $E_{ij}$ the strict transform of the line joining $p_i, p_j$.
  Then the $E_i, E_{ij}$ are the ten $-1$-curves on $\Mbar_{0,5}$.
  Let $\pi_{5,4}$ be the projection away from $p_0$; it
  gives a map to $\P^1$ whose general fibre is a
  smooth rational curve and that has three singular reducible fibres
  $E_i \cup E_{0i}$ with $1 \le i \le 3$.  There are four sections
  $E_0, E_{12}, E_{13}, E_{23}$, and this family is isomorphic to the 
  universal curve over $\Mbar_{0,4}$.
\end{defn}

\begin{cons}\label{cons:rat-surf} Fix $p \in \Mbar_{0,4}$, and let
  $F = \pi_{5,4}^{-1}(p)$ be the
  corresponding fibre.  Consider the double cover $\E_p$ of $\Mbar_{0,5}$
  branched along $F$ and the four sections.  (It is easy to see that the
  sum of the classes of these curves is
  divisible by $2$ in $\Pic(\Mbar_{0,5})$; since 
  $\Pic(\Mbar_{0,5})$ is torsion-free,
  the double cover is unique up to isomorphism and
  quadratic twist.  The properties that we will discuss do not depend on
  the choice of twist.)
  This is an elliptic surface,
  because the general fibre is a double cover of $\P^1$ branched at $4$ points.
  The surface $\E$ has fibres of type $\tA_1$ above the three reducible fibres
  of $\Mbar_{0,5} \to \Mbar_{0,4}$ and a $\tD_4$ above $F$.  Also the four
  sections of $\Mbar_{0,5} \to \Mbar_{0,4}$ pull back to sections; taking $E_0$
  as the zero section, the other three are the $2$-torsion sections.
\end{cons}

\begin{prop}\label{prop:mod-d4-3a1}
  The association $p \to \E_p$ gives a birational equivalence
  between $\Mbar_{0,4}$ and the moduli space of rational surfaces with an
  elliptic fibration with a $\tD_4$ and three labelled $\tA_1$ fibres.
\end{prop}

\begin{proof} First note that there is an obvious one-sided inverse to the map
  just constructed, namely the map that takes a fibration to the point $p$
  such that $p, b_1, b_2, b_3$ are the locations of the $\tD_4$ and $\tA_1$
  fibres.  Given such an elliptic surface $\E$, we obtain a map
  $\P^1 \to \Mbar_{0,4}$ from the family of curves $\E/\pm 1 \to \P^1$, where
  $\P^1$ is the base of the fibration.  This map determines the $j$-invariant
  of the fibration.  It is clear that the inverse images of the $b_i$ have
  degree $1$, so this map is an isomorphism and is therefore uniquely
  determined (recall Convention \ref{conv:m04-bar} on automorphisms of
  $\Mbar_{0,4}$).  Thus all such surfaces are twists of a fixed one; in
  particular, given one such surface with a $\tD_4$ fibre above $p_0$, we
  obtain all of them by twisting by a function with divisor $(p_0) - (p)$.
  Clearly these are parametrized by points of $\Mbar_{0,4}$.
\end{proof}

\begin{prop}\label{prop:mw-ep}
  The Mordell-Weil group of $\E_p$ is isomorphic to $\Z \oplus (\Z/2\Z)^2$, and
  the components of the pullback of the conic through $p_0, \dots, p_3$
  tangent to $F$ at $p_0$ are a generator modulo torsion and its inverse.
  These curves pass through the zero component of the $\tD_4$ fibre and the
  nonzero components of the $\tA_1$.
\end{prop}

\begin{proof} Since $\E_p$ is a rational elliptic surface, its Picard lattice
  has rank $10$ and discriminant $-1$.
  It follows by the Shioda-Tate formula that the rank of the
  Mordell-Weil group is $1$.  We know that the
  torsion subgroup contains $(\Z/2\Z)^2$; by \cite[Proposition 2.1]{cox},
  it is no larger.  Thus the canonical height of a generator is
  $4^2/(4 \cdot 2^3) = 1/2$.  Let $C$ be the conic in the statement: then
  the intersection of $C$ with the ramification locus of $\E_p \to \Mbar_{0,5}$
  is twice a divisor, so $C$ pulls back to the union of two curves.  The
  sum of these meets a fibre twice, so each one must be a section.  Taking
  the curve above $E_0$ as the origin, we see that the two sections intersect
  every good fibre in a point and its negative; it follows that they are
  inverses of each other.  It is
  clear that these sections meet the zero component of the $\tD_4$
  (coming from the intersection of $F$ with the exceptional divisor above $p_0$)
  and the nonzero components of the $\tA_1$ (the $E_{ij}$ with
  $1 \le i < j \le 3$).  In addition, they meet the zero section once.  It
  now follows easily from the results of \cite{cz} that they have height
  $1/2$.
\end{proof}
  
Given a point of $\M_d$, we thus obtain an elliptic surface
$\E_p \times_{\Mbar_{0,4}} C_0$, which generically has $3d$ fibres of type
$\tA_1$ and $d$ of type $\tD_4$.  In particular, for $d = 2$ these are
special double covers of $\P^2$ branched along six lines, which are
K3 surfaces.  The elliptic fibrations on a double cover of
$\P^2$ ramified along six lines in general position are classified in
\cite[Theorem~1.1, Corollary~1.3]{kloosterman} (we thank the referee
for calling this to our attention).

\subsection{Relating the two types of moduli space}

We now indicate how the two types of moduli spaces of elliptic surfaces
are related.  
\begin{defn}\label{def:m-34-prime}
  Let $\M'_3 \subset \M_3$ be the subvariety of pairs where $p$ lies under
  a ramification point of $\phi$, and let $\cN'_4 \subset \M_4$ be the
  subvariety of pairs where $p$ lies under a ramification point of type
  $(2,2)$.  (Later, in Definition~\ref{def:m-n-prime},
  we will define $\M'_n$ in general; our definition will
  coincide with $\M'_3$ for $n = 3$ but not with $\cN'_4$ for $n = 4$.)
\end{defn}

\begin{prop}\label{prop:m-34-k3}
  The elliptic surface corresponding to a point of $\M'_3$ or $\cN'_4$ is
  a K3 surface.
\end{prop}

\begin{proof} The ramification means that one or two pairs of $\tD_4$ fibres
  coalesce and therefore become a smooth fibre, so we have one $\tD_4$ and
  nine $\tA_1$ fibres, resp.~12 $\tA_1$ fibres, in the two cases.  The
  result follows from \cite[Lemma III.4.6 (a)]{miranda}.
\end{proof}

\begin{remark}\label{rem:n-3-and-beyond}
  Our primary concern in this paper is with the case $n = 3$.  Nevertheless
  we mention that in the case $n = 4$ we may obtain Mordell-Weil
  rank $2$, and thus Picard
  number $16$, by requiring an additional ramification point of type $(2,2)$.
  It turns out that these surfaces admit $4$-isogenies to double covers of
  $\P^2$ branched along six lines which are the composition of two
  quotients by van Geemen-Sarti involutions: see
  Proposition~\ref{prop:n-4-already-known}.
  When $n > 4$ we cannot obtain a
  K3 surface from this construction, but we do obtain some
  interesting elliptic surfaces.  These will be described in Definition
  \ref{def:higher-moduli} and a basic property given in Proposition
  \ref{prop:birat-moduli}.
  Unfortunately they do not appear to admit
  correspondences to K3 surfaces.
\end{remark}

Note in particular that we have a partition of the $\tA_1$ fibres into
three sets of $3$ according to the fibres of $\E_p$ above which they lie
(cf.~Remark \ref{rem:partition-a1}).
Thus we have constructed a map $\kappa:\M'_3 \to \M'_{K,1}$
(Definition~\ref{def:mk1}):

\begin{defn}\label{def:kappa}
  Let $\kappa$ be the map $\M'_3 \to \M'_{K,1}$ that takes a pair
  $(\phi,p)$ to $\E_p \times_{\Mbar_{0,4}} C_0$, where the induced fibration
  and section of infinite order are those pulled back from
  $\E_p \to \Mbar_{0,4}$ and the order on $2$-torsion
  sections is that of $\E_p$.
\end{defn}

In addition, there is an obvious map $\nu:\M'_{K,1} \to V'$ given by mapping to the
locations of the singular fibres:

\begin{defn}\label{def:nu} Suppose we are given the data of a point of $\M'_{K,1}$.
  We obtain a point of $V'$ as $(p,\{q_{ij}\})$, where $p$ lies under the
  $\tD_4$ fibre and the $q_{ij}$ lie under the $\tA_1$ fibres, in such a way
  that the torsion section $T_i$ passes through the zero components of the
  fibres above the $q_{ij}$.
\end{defn}

It is not difficult to see that $\M'_3$ and $V'$ are birationally equivalent.
Indeed, define $\rho: \M'_3 \to V'$ to send a point $(\phi,p)$ to the third
point above $p$ and the fibres above $b_1, b_2, b_3$; this is a birational
equivalence because a map $\P^1 \to \P^1$ is determined up to scaling on the
target by its fibres at $0, \infty$, while the scaling is fixed by the fibre
at $1$.  Also, it is clear from the construction that
$\nu \circ \kappa = \rho$.

\begin{prop}\label{prop:nu-birational}
  The map $\nu$ is a birational equivalence.  Hence $\kappa$ is also.
\end{prop}

\begin{proof}
Since the dimensions are equal, it suffices to show that the degree of $\nu$
is~$1$.  This will be done by an
argument much like that of Proposition \ref{prop:mod-d4-3a1}.  Indeed, given a
K3 surface
parametrized by a point of $\Mbar_{K,1}$, we obtain a family of $4$-pointed
stable curves of genus~$0$ by taking the quotient by the negation.  The map
to $\Mbar_{0,4}$ is of degree $3$, because there are $3$ of each type of
reducible fibre, and such a map is determined by the fibres at $0, 1, \infty$.
In addition, the location of the single $\tD_4$ fibre is determined once
we choose the image of a ramification point in $\Mbar_{0,4}$.  Hence the
$j$-invariant of the fibration and the location of fibres with starred Kodaira
type are determined by the image in $V$; but this is enough to recover the
fibration.
\end{proof}

The introduction of $\M'_3$ allows us to give a further relation between
$\M_{K,1}$ and $\M_{K,2}$.
\begin{prop}\label{prop:same-deg-3}
  The degree-$3$ map from $\sigma_0$ to the base of $\phi_2$
  coincides with the map $\phi$ in the corresponding point of $\M'_3$.
\end{prop}

\begin{proof} For $i = 1, 2, 3$, we consider the intersection
  $T_i \cap 0_{\phi_1}$.
  By definition of $\phi$, each of these is the fibre of the map
  $C_0 \to \Mbar_{0,4}$ above a boundary point.  On the other hand, since the
  $T_i$ are fibres of $\phi_2$, the map $\phi_2|_{\sigma_0}$ is constant on each
  of the $T_i \cap 0_{\phi_1}$, and in fact these are fibres because their degree
  is equal to that of the map.

  We have seen that the two maps from $\P^1$ to a curve of genus~$0$ have
  three fibres in common.  Therefore they are equal.  (We did not identify
  the base of $\phi_2$ with $\Mbar_{0,4}$, but it would be natural to do so
  by choosing the points under the $\tD_4$ fibres to be the boundary points.)
\end{proof}

\subsection{Some elliptic surfaces of Kodaira dimension $1$}
In most cases this construction does not give K3 surfaces, but with
larger $n$ and more general ramification type we still obtain specific
families of elliptic surfaces.  These satisfy conditions given 
in Definition \ref{def:higher-moduli} below.
To do so, we fix $n > 2$ and let $R$
be a subset of the ramification data for a map of degree $n$ of rational
curves.  In other words, $R$ is a sequence $(r_i)_{i=1}^k$ of partitions of
$n = \sum_{j=1}^{m_i} a_{ij}$ such
that $\sum_i \sum_j a_{ij}-1 \le 2n-2$, corresponding to maps of rational
curves with fibres $F_1, \dots, F_k$ such that the multiplicities of the points
in $F_i$ are given by $(r_i)$.  Let $t$ be the sum of the
number of parts of the $r_i$, and for $1 \le i \le k$ let
$c_i = \sum_{j=1}^{m_i} a_{ij}-1$ be the number of conditions imposed by the
ramification data at $F_i$.  We assume that $2n-2-\sum_{i=1}^n c_i \ge 0$.

\begin{defn}\label{def:higher-moduli}
  Given an elliptic curve $E/k$ with labelled $2$-torsion points
  $t_1, t_2, t_3$, let $\lambda(E) \in k$ be the unique element such that
  there is an isomorphism $E \to E_\lambda$ taking $t_1, t_2, t_3$ to
  $(0,0), (1,0), (\lambda,0)$ respectively, where $E_\lambda$ is the
  elliptic curve $y^2 = x(x-1)(x-\lambda)$.  Given a family of elliptic
  curves over a smooth curve $C$ with labelled $2$-torsion sections,
  we will refer to the map
  $C \to \Mbar_{0,4}$ taking a fibre $F$ to
  $\lambda(F)$ as the {\em $\lambda$-invariant} map.  A priori it is only
  defined over an open subset of $C$, but since $C$ is a smooth curve
  it extends to all of $C$.
  
  Let $\M'_{n,R}$ be the
  moduli space of pairs consisting of a map of degree $n$ from $C_0$ to
  $\Mbar_{0,4}$ and $k$ points $p_1, \dots, p_k \in \Mbar_{0,4}$
  where the ramification above $p_i$ is as specified by the partition~$r_i$.
  Let $V'_{n,R}$ be the quotient of the
  subset $U \subset \Mbar_{0,t+3n}/\S \times \S_n^3$ that parametrizes
  collections of $t+3n$ points
  $p_{11}, \dots, p_{1m_1}, \dots, p_{k1},\dots,p_{km_k}$,
  $q_{11}, \dots, q_{1n}$, $q_{21},\dots,q_{2n}$,
  $q_{31},\dots,q_{3n}$
  on a rational curve $C_0$ and maps $\phi$ of degree $n$
  from $C_0$ to $\Mbar_{0,4}$ and points $p_1, \dots, p_k \in \Mbar_{0,4}$
  for which
  $\phi^{-1}$ of the divisor $(p_i)$ is $\sum_{j=1}^{m_i} a_{ij} (p_{ij})$,
  and such that the $\{q_{ij}: 1 \le j \le n\}$ are
  the fibres of $\phi$ above the boundary points of $\Mbar_{0,4}$.  Here $\S$
  is the subgroup of $\prod_{i=1}^k \S_{m_i}$ that preserves the partitions,
  while $\S_n^3$ acts on the second subscripts of the $q_{ij}$.

  Let $\M'_{\E,n,R}$ be the coarse moduli space of elliptic surfaces $S \to C_0$
  with full level-$2$ structure and a labelling $(t_1,t_2,t_3)$ of the sections
  of order $2$, together with a collection of $t+3n$ points up to
  permutation as above, such that:
  \begin{itemize}
  \item the pair $(C_0,p_{11},\dots,p_{km_k},q_{11},\dots,q_{3n})$ corresponds to a
    point of $U$ as in the last paragraph;
  \item the surface has $\tD_4$ fibres above the points $p_{1j}$ for the $j$
    such that $r_{1j}$ is odd and has $\tA_1$ fibres above the $q_{ij}$;
  \item for all $i$, the $\lambda$-invariant of the fibre above $p_{ij}$ does
    not depend on $j$;
  \item for $1 \le i \le 3$, the torsion section $t_i$ passes through the
    zero component of the fibres above the $q_{ij}$ and the nonzero
    component of the other $\tA_1$ fibres (in other words, the
    $\lambda$-invariant map takes the $q_{ij}$ to $0, 1, \infty$ for
    $i = 1, 2, 3$ respectively).
\end{itemize}
\end{defn}

\begin{remark}\label{rem:euler-char-nm} Let $m$ be the number of odd parts
  of the partition~$r_1$.  Then the general point of $\M'_{\E,n,R}$
  describes an elliptic surface of Euler characteristic $6(m+n)$ and hence
  $h^{2,0} = (m+n-2)/2$ by \cite[(III.4.2), (III.4.3)]{miranda}.  So,
  as in Remark \ref{rem:n-3-and-beyond}, we do not obtain K3 surfaces for
  $n > 4$.
\end{remark}

\begin{prop}\label{prop:birat-moduli}
  The moduli spaces $\M'_{n,R}, \M'_{\E,n,R}, V'_{n,R}$ are birationally
  equivalent.
\end{prop}

\begin{proof} 
  First we consider the map $\mu_1: V'_{n,R} \to \M'_{n,R}$ taking the pair
  consisting of the given map $\phi$ and the collection of points to the
  same map $\phi$ and the points $\phi(p_{i1})$ for $1 \le i \le n$.  This
  is a birational equivalence.  The sets $\{q_{ij}: 1 \le j \le n\}$ are
  determined by $\phi$ and the set $\{p_{ij}: 1 \le j \le m_k\}$ modulo
  the subgroup of $\S_{m_k}$ preserving the partition uniquely determines,
  and is uniquely determined by, its image $\phi(p_{i1})$.

  We now define a map $\mu_2: V'_{n,R} \to \M'_{\E,n,R}$.  Given a point 
  $\alpha = (\phi,\{p_{ij}\},\{q_{ij}\}) \in V'_{n,R}$, we construct a point
  $\mu_2(\alpha)$ whose underlying surface (unique up to twist) has
  $\lambda$-invariant $\phi$.  Identify the $2$-torsion sections by means
  of the ordering of the $\{q_{ij}\}$ and choose the quadratic twist
  such that starred fibres occur only at the $p_{1j}$ where $r_{1j}$ is odd
  (this is unique up to an element of $k$; it exists because
  the number of such points is congruent mod $2$ to the contribution to the
  Euler characteristic from fibres with multiplicative reduction divided by
  $6$).  It is clear that this surface has the desired properties, and that
  the forgetful map taking a point of $\M'_{\E,n,R}$ to
  $(\lambda,\{p_{ij}\},\{q_{ij}\})$ is a birational equivalence.
\end{proof}
\begin{remark}\label{rem:motive-finite-ell-surfs}
  We will show in Theorem \ref{thm:motive-finite-ell-surfs} that, for
  certain special choices of $n$ and $R$, the surfaces parametrized by
  $\M'_{\E,n,R}$ are motive-finite.  We do not know whether this is to
  be expected in other cases.  The case
  $n = 4, R = ((2,2))$, corresponding to a $5$-dimensional family of K3
  surfaces of generic Picard rank $15$, is perhaps the most interesting.
\end{remark}

\begin{remark}\label{rem:moduli-spaces}
  In contrast to the case of K3 surfaces, we see no reason why the
  elliptic surfaces parametrized by the $\M'_{\E,n,R}$ should be determined
  by their Picard groups.  Consider, for example, the case
  $n = 5, R = ((2,1,1,1))$.  Then $\M'_{n,R}$ parametrizes covers of
  degree $5$ of a curve with a point of ramification, so it
  is of dimension~$8$.  On the other hand, the points of $\M'_{\E,n,R}$
  correspond to elliptic surfaces with $15$ fibres of type $\tA_1$ and $3$
  of type $\tD_4$, leading to a generic Picard number of $29$.  In the
  $38$-dimensional moduli space of elliptic surfaces over $\P^1$ with
  $h^{2,0} = 3$, surfaces with these reducible fibres should be a
  family of dimension~$11$.  We see no way to construct $3$ independent
  sections on such an elliptic surface and specialization does not often
  produce elliptic curves over $\Q$ of rank $3$ or greater,
  so we believe that the points of $\M'_{\E,n,R}$ do
  not exhaust the locus of elliptic surfaces of a given generic Picard
  group.
\end{remark}

\section{The construction in terms of curves}\label{sec:constr-curves}
In this section we will give an entirely different construction of a map
$\M'_3 \to \M'_{K,1}$, based on the ideas of Paranjape \cite[Section 3]{paranjape}.
We begin by using a point $p_0 \in \M'_3$ to construct an elliptic
curve together with some auxiliary data similar but not identical to that
used in \cite[Section 3]{paranjape}.
We will then use this to construct another curve
$C_3$ such that $S_{p_0}$ is a quotient of $C_3 \times C_3$,
where $S_{p_0}$ is the K3 surface associated to the point of $\M'_{K,1}$ which
is the image of $p_0$ by the construction of the last section.

Although the construction of a K3 surface quotient of the square of a curve
is special to $n = 2, 3$, most of the geometry can be studied without this
assumption.  Accordingly we will work with general $n$ for as long as possible
and specialize to $n = 3$ only at the end.  With $n = 2$ our construction is
less interesting from this point of view, because the K3 surfaces that we
obtain are isogenous (Definition \ref{def:isogeny}) to Kummer surfaces.
Hence motive-finiteness and the
strong form of the Kuga-Satake construction are already known for them.
In the case $n = 4$ our construction also gives a correspondence to a certain
moduli space of K3 surfaces.  We do not obtain 
results on the Kuga-Satake conjecture
or motive-finiteness directly,
because we do not realize these surfaces as quotients.  However, we
will show in Proposition \ref{prop:n-4-already-known} that these surfaces are 
isogenous to K3 surfaces of the type introduced
at the beginning of \cite[Section 1]{paranjape} and shown to be quotients
of the square of a curve at the end of \cite[Section 3]{paranjape}.

In this section we are concerned with moduli spaces and so for ease of
exposition we will assume that the ground field $k$ is algebraically
closed when constructing any map of moduli spaces.  The reason for
this is that in order to find the image of a point corresponding to a
K3 surface defined over $k$, we may need to make a choice of one of a
set of data, where the set is defined over $k$ but the individual
elements are not.  It will turn out that the image represents a point
in the moduli space defined over $k$, so that the map of moduli spaces
takes $k$-points to $k$-points for all $k$ and is therefore defined
over the prime subfield of $k$.  Further, we will see that our map is a
birational equivalence over an algebraically closed field; this
property descends to $k$.

\subsection{Moduli spaces related to curves of genus $0$ and $1$}
We start by defining a moduli space $\M'_n$ that generalizes the construction
of $\M'_3$ in Definition~\ref{def:m-34-prime}.
Our definition will capture the ramification properties of
the quotient of a cyclic $n$-isogeny of elliptic curves by $\pm 1$.  That is
to say, a cyclic $n$-isogeny $E_n \to E$ of elliptic curves descends to a map
of quotients $E_n/\pm 1 \to E/\pm 1$, which gives a point of $\M'_n$.

\begin{defn}\label{def:m-n-prime}
  Fix an integer $n > 2$.  Let $\M'_n$ be the moduli space parametrizing
  degree-$n$ covers $\phi_n: C_0 \to \Mbar_{0,4}$ such that there are $4$ points
  where the ramification is of type $(2,\dots,2,1)$ if $n$ is odd, or
  $2$ each where it is of type $(2,\dots,2)$ and $(2,\dots,2,1,1)$ if
  $n$ is even, together with a point $q_0 \in C_0$ in a fibre with
  $\lfloor \frac{n+2}{2} \rfloor$ distinct points.

  For $n = 2$ we make a slightly different definition.  Namely, we fix a
  degree-$2$ cover $C_0 \to \Mbar_{0,4}$, where $C_0$ is a rational curve,
  and two additional points of
  $\Mbar_{0,4}$, and single out a point of $C_0$ above one of these.
\end{defn}

\begin{remark}\label{rem:dim-mnprime}
  Note that in every case we have $\dim \M'_n = 4$: 
  the dimension of the space of degree-$n$ covers is $2n-2$, while the
  ramification imposes $4(n-3)/2$ conditions for odd $n$ or $2(n-2)/2+2(n-4)/2$
  for even $n > 2$, and the choice of $q_0$ is from a finite set.  For $n = 2$
  this is clear.  We also point out that the genus of $C_0$ is $0$ for $n > 2$,
  not only for $n = 2$: this follows from Riemann-Hurwitz.
\end{remark}

\begin{defn}\label{def:e-en}
Let $E$ be the double cover of $\Mbar_{0,4}$ branched at the ramification
points of $\phi_n$ (if $n = 2$, the ramification points and the two additional
points) and define $E_n$ to be the normalization of
$E \times_{\Mbar_{0,4}} C_0$.  Choose the origin
$O$ on $E$ to be the point of $E$ lying above
$\phi_n(q_0) \in \Mbar_{0,4}$ (if $n = 2$, the point that was specified in
defining the cover).
\end{defn}

\begin{remark}\label{rem:m3-mc}
  It is clear that $\M'_n$ is birational to the moduli space whose points
correspond to a cyclic $n$-isogeny $E_n \to E$, an unordered set of $4$ points
of $E$, and a double cover branched at those $4$ points.  (Recall that the
double cover is not uniquely determined by its branch locus; rather, the
double covers with a given branch locus constitute a torsor for $\Pic[2](E)$.
Concretely, if one double cover corresponds to the extension of function
fields $K(E)(\sqrt f)/K(E)$, then the others are given by the extensions
$K(E)(\sqrt {fg})$, where $g$ belongs to the group of functions with divisor
divisible by $2$ modulo squares.)
In particular,
taking $n = 3$, we have shown that $\M'_3$ is birationally equivalent to $\M_C$
(Definition \ref{def:mc}).
\end{remark}

Our first task is to define an auxiliary double cover $D_3 \to E$ branched
at $4$ points and hence of genus~$3$.  Let $\pm B_i$ be the points of $E$
lying above the boundary points $b_i \in \Mbar_{0,4}$.

\begin{prop}\label{prop:construct-pi}
  Up to translation and negation, there is a unique set $\{p_0,\dots,p_3\}$
  of $4$ points on $E$ such that
  $\{p_i - p_j: 0 \le i \ne j \le 3\} = \{\pm B_i \pm B_j: 1 \le i < j \le 3\}$, counted with multiplicity.
\end{prop}

\begin{proof} For existence, let $p_0 = O$ and $p_i = B_1 + B_2 + B_3 - B_i$
  for $1 \le i \le 3$.  For uniqueness, we may again let $p_0 = O$, since
  translations are permitted.  We then have $p_i = \pm B_{j_i} \pm B_{k_i}$ for
  suitable choices; it is not possible to have a plus and a minus sign in
  each point, so, negating and reordering if necessary, we take
  $p_1 = B_2 + B_3$.  Then, since $p_2, p_1 - p_2 \in \{\pm B_i \pm B_j\}$,
  we must have $p_2 \in \{B_1+B_2,-B_1+B_2,B_1+B_3,-B_1+B_3\}$; by symmetry
  we take $p_2 \in \{B_1+B_3,-B_1+B_3\}$.  It is easily seen that the two
  choices force $p_3 = B_1+B_2$ or $p_3 = -B_1+B_2$.  The second case is in
  fact the same as the first: translating by $B_2+B_3$ and changing the sign
  gives the same four points.
\end{proof}

\begin{defn}\label{def:d3}
  Let $f$ be a function on $E$ with divisor
  $(p_0) + (p_1) + (p_2) + (p_3) - 2(O) - 2(B_1 + B_2 + B_3)$, and let
  $D_3$ be the smooth curve with function field $k(E)(\sqrt f)$ (where
  $k$ is the ground field and $k(E)$ is the function field of $E$
  as usual).  Let the $P_i$ be the unique points of $D_3$ lying above the $p_i$.
\end{defn}

\begin{remark}\label{rem:negate-bi-same}
  Note that the choice of $O$ determines the double
  cover of $E$ ramified above $p_1, \dots, p_4$.  Indeed, replacing
  $O$ by $O+T$, where $T$ is a $2$-torsion point, would add $2(O+T)-2(O)$
  to the divisor of $f$.  This is a principal divisor and twice a divisor
  but not twice a principal divisor, so it gives a different double cover.
  
  Since we cannot distinguish $B_i$ from $-B_i$,
  we must explain why our construction does not depend on the choice.
  Indeed, let us replace $B_1$ by $-B_1$, so that the points become
  $p_0' = O, p_1' = B_2 + B_3, p_2' = -B_1 + B_3, p_3' = -B_1 + B_2$ and the
  divisor of the function is
  $(p_0') + (p_1') + (p_2') + (p_3') - 2(O) - 2(-B_1 + B_2 + B_3)$.  Translating
  by $-B_2 - B_3$ and changing the sign, we restore the previous points
  and obtain a function~$f'$ with divisor
  $(p_0) + (p_1) + (p_2) + (p_3) - 2(B_1) - 2(B_2 + B_3)$, which is equal to the
  previous $f$ up to a square.  Thus $D_3$ is unaltered.

  For $D_3$ to be unique up to isomorphism we must 
  assume that $k$ has no nontrivial quadratic
  extensions.  But even without this assumption the point of $\Mbar_3$
  corresponding to $D_3$ is well-defined.
\end{remark}

\subsection{A curve of genus $2n+1$ whose square covers an elliptic surface}
We now define the curve whose square will cover an elliptic surface over
the elliptic curve $E_n$.  We will then show that this surface has a quotient
which is an elliptic surface $W$ over $\P^1$; in the case $n = 3$, this will be
a K3 surface.
Following this, we compare our construction to Paranjape's.

\begin{defn}\label{def:cn}
  Let $C_n = D_3 \times_E E_n$.  This is a curve of genus~$2n+1$ with an
  action of $\Z/n\Z \oplus \Z/2\Z$.  Let $\iota$ be the involution of $D_3$
  induced by the double cover, let $\beta$ be its lift to $C_n$ (the quotient
  being $E_n$), and let $\gamma$ be a generator of the group of automorphisms
  of $C_n$ over $D_3$.  Let $G_n = \langle \gamma, \beta \rangle$, so that
  $G_n \cong \Z/n\Z \oplus \Z/2\Z$. 
\end{defn}

We summarize our definitions in Figure \ref{fig:quotients-cn}.

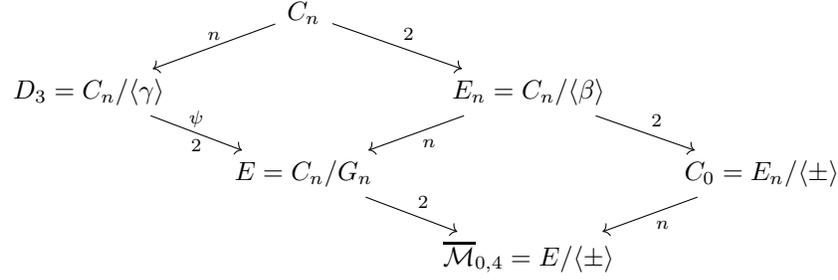
\begin{figure}\centering
\begin{equation*}
  \begin{tikzcd}[row sep=scriptsize, column sep=scriptsize]
    & C_n \arrow[dl,"n" above] \arrow[dr,"2"] & & \\
    D_3 = C_n/\langle \gamma \rangle \arrow [dr,"2" below,"\psi" above] & & E_n =C_n/\langle \beta \rangle \arrow[dl,"n"] \arrow[dr,"2"] & \\
    & E = C_n/G_n \arrow[dr,"2"] & & C_0 = E_n/\langle \pm \rangle  \arrow[dl,"n"]   \\
      & &  \Mbar_{0,4} = E/\langle \pm \rangle &
  \end{tikzcd}
\end{equation*}
\caption{Quotients of $C_n$}\label{fig:quotients-cn}
\end{figure}

\begin{remark}\label{rem:paranjape}
  Paranjape fixes $p_0, \dots, p_3 \in E$ (in his notation, $p_1, \dots, p_4$)
  and considers a double cover of $E_n$ (in fact he
  only works with $n = 2$) branched along the pullback of $\sum_{i=1}^4 (p_i)$
  from $E$ to $E_n$.  Thus in his construction the Galois group is $\Z/2n\Z$
  rather than $\Z/n\Z \oplus \Z/2\Z$.  The Prym variety of the cover
  $C_n \to E$ is an abelian variety of dimension~$2n$ with an automorphism
  of order $2n$, and we may view an appropriate component of its Hodge structure
  as a module of rank $2$ over $\Z[\zeta_{2n}]$.  In our construction the
  relevant module associated to a summand of the Hodge structure of the
  Prym variety of $C_n \to E$ is over $\Z[\zeta_n]$ instead, although this
  is a distinction without a difference for odd $n$.
\end{remark}

\begin{defn}\label{def:psi}
  Let $\psi: D_3 \to E$ be the natural quotient map
  $C_n/\langle \gamma \rangle \to C_n/G_n$,
  as in Figure \ref{fig:quotients-cn}.
\end{defn}

\begin{lemma}\label{lem:tangent-d3}
  Suppose that $D_3$ is not hyperelliptic, and consider it with its
  canonical embedding in $\P^2$.
  Let $P_i$ be the point above $p_i$ in $D_3$, and let $T_{i,1}, T_{i,2}$ be
  the residual intersection of the tangent line to $D_3$ at $P_i$.
  Then $\iota(T_{i,1}) = T_{i,2}$ and $p_i + \psi(T_{i,1}) \sim 2O$ on $E$.
  In other words, if $O$ is chosen as the origin of $E$, then
  $\psi(T_{i,1}) = -p_i$.
\end{lemma}

\begin{proof} Since $P_i$ is $\iota$-invariant, the same is true of the
  tangent line and of its intersection with $D_3$, so the first claim follows.
  For the second one, note that the divisor $(p_i) + (-p_i) + 2O$ is principal
  on $E$.  Let $S_{i,1}, S_{i,2}$ be the points above $-p_i$ in $D_3$: then 
  $2(P_i) + (S_{i,1}) + (S_{i,2}) - 2(O_1) - 2 (O_2)$ is the pullback of
  this principal divisor, so it is principal on $D_3$.  In addition,
  $(P_0) + (P_1) + (P_2) + (P_3) - 2(O_1) - 2(O_2)$ is
  principal, since it is the divisor on $D_3$ of the element
  $\sqrt{f} \in k(D_3)$ as defined in Definition~\ref{def:d3}.

  By the Riemann-Hurwitz formula, the divisor $(P_0) + (P_1) + (P_2) + (P_3)$
  is in the canonical class; the same follows for the linearly equivalent
  divisor $2(P_i) + (S_{i,1}) + (S_{i,2})$.  But $2(P_i) + (T_{i,1}) + (T_{i,2})$
  is also in the canonical class, being a hyperplane section of a canonically
  embedded curve, and $D_3$ is not hyperelliptic, so
  these divisors must be equal.
\end{proof}

No power of $\gamma$ not equal to the identity has fixed points.  The
quotient of $C_n$ by $\beta$ has genus~$1$, so $\beta$ has $4n$ fixed points.
To determine the number of fixed points of $\beta\gamma^i$, consider the group
generated by $\beta, \gamma^i$: the quotient is of genus~$1$ and $\beta$ has
$4n$ fixed points.  So by Riemann-Hurwitz no other nonidentity element of
$G_n$ has any fixed points.

We now introduce the surface $C_n \times C_n$.  It admits an action of
the group $G_n \wr \S_2$ in which $G_n$ acts on each copy of $C_n$ and
$\S_2$ interchanges the factors.

\begin{defn}\label{def:g-w}
  Following \cite[Section 3]{paranjape}, let $G$ be the subgroup
  of $G_n \wr \S_2$ generated by $(\gamma^{-1},\gamma), (\beta,\beta)$, and the
  nonidentity element of $\S_2$.  Denote the generators by $g, b, \sigma$
  respectively.  Let $W$ be the surface
  $(C_n \times C_n)/G$, and let $\omega$ be the map $C_n \times C_n \to W$.
  Let $\tW$ be the minimal desingularization of $W$ (we will describe
  the singularities of $W$ in Proposition \ref{prop:w-ell-surf}).
  For a surface of the form $C \times C$, where $C$
  is a curve, we use $\sigma$ for the involution
  $(x,y) \to (y,x)$.
  We record some subgroups of $G_n \wr \S_2$ whose quotients define surfaces of
  interest.  In particular, as in Figure \ref{diag:vars} we define
  $V := C_n \times C_n /\langle (\gamma,1),(\beta,\beta),\sigma)\rangle$ and
  $S:=C_n \times C_n/ \langle (\gamma,\gamma^{-1}),(\beta,1),\sigma)\rangle$.
\end{defn}

\begin{figure}\centering
\begin{equation*}
\begin{tikzcd}[row sep=scriptsize, column sep=scriptsize]
  & \langle e \rangle \arrow[dl,"n^2" above] \arrow[rr,"4"] \arrow[dd,"4n" near start] & & \langle (\beta,1),(1,\beta) \rangle \arrow[dl,"n^2"] \arrow[dd,"2n"] \\
\langle (\gamma,1),(1,\gamma)\rangle \arrow[rr,"4" near end, crossing over] \arrow[dd,"4" left] & & G_n\times G_n \arrow [dd,"2" near start]\\
& G=\langle (\gamma,\gamma^{-1}),(\beta,\beta),\sigma)\rangle \arrow[dl,"n" above] \arrow[rr,"2" near start] & & \langle (\gamma,\gamma^{-1}),(\beta,1),\sigma)\rangle \arrow[dl,"n"]  \\
\langle (\gamma,1),(\beta,\beta),\sigma)\rangle \arrow[rr,"2" near end, crossing over] & &  G_n \wr \S_2 \arrow[from=uu, crossing over] 
\end{tikzcd}
\end{equation*}

\begin{equation*}
\begin{tikzcd}
  & C_n\times C_n \arrow[dl,"n^2" above] \arrow[rr,"4"] \arrow[dd,"\omega" near start,"4n" near end] & & E_n\times E_n \arrow[dl,"n^2"] \arrow[dd,"2n"] \\
D_3 \times D_3 \arrow[rr,"4" near end, crossing over] \arrow[dd,"4" left] & & E\times E \arrow [dd,"2" near start]\\
& W \arrow[dl,"n" above] \arrow[rr,"2" near start] \arrow[dd] & & S \arrow[dl,"n"] \arrow[dd] \\
V \arrow[rr,"2" near end, crossing over] \arrow[dd] & & \Sym^2E \arrow[from=uu, crossing over] \\
& E_n \arrow[dl,"n" above] \arrow[equal,rr] & & E_n \arrow[dl,"n"] \\
E \arrow[equal,rr] & & E \arrow[from=uu, crossing over]\\
\end{tikzcd}
\end{equation*}
\caption{Relations of groups and varieties}\label{diag:vars}
\end{figure}

\begin{prop}\label{prop:w-ell-surf} The surface $W$ is an elliptic surface
  over $E_n$ with full level-$2$ structure.
  It has $6n$ ordinary double points; resolving these creates
  $6n$ fibres of type $\tA_1$, and these are the only singular fibres of
  $\tW$.
\end{prop}

\begin{proof}
  We first claim
  that $$W = V \times_{\Sym^2E} S =  V \times_{\Sym^2E} \Sym^2E \times_E E_n
  = V\times_E E_n.$$
  The first equality follows from the fact that $W, V, \Sym^2E, S$ are
  all quotients of $C_n \times C_n$ for which the corresponding diagram
  of groups is a pushout.
  (This is illustrated in Figure~\ref{diag:vars}.)
  The second results from the fact that $E_n \to E$ are unramified, so
  that the fibres of
  $W \to E_n$ are fibres $V \to E$.  The third is immediate.

  Since $\pi:V \to \Sym^2 E$ has degree
  $2$ and $\Sym^2E \to E$ is a ruled surface, we see that the generic
  fibre of $V \to E$ is a double cover of a smooth rational curve, so
  we need to compute the ramification locus of $\pi:V \to \Sym^2E$.
  Because $V = C_n\times C_n /\langle (\gamma,1),(\beta,\beta),\sigma\rangle$
  and $\Sym^2E = C_n\times C_n /\langle (\gamma,1),(\beta,1),\sigma\rangle$, the
  map $\pi:V \to \Sym^2E$ is given by the quotient by $(\beta,1)$ and
  so $\pi$ is branched along the curves $\{p_i,E\} \subset \Sym^2E$ for
  $i=0,\ldots,3$.
  The curves $\{p_i,E\}$ are sections of the addition map $\Sym^2 E \to E$.
  Since the general fibre of $V \to E$ is a double cover of the fibre of
  $\Sym^2 E \to E$ branched at $4$ points, and the fibre of $\Sym^2 E \to E$
  is a rational curve, we see that the general fibre of $V \to E$
  is an elliptic curve.  Furthermore, we see that
  $\{p_i,E\}\cap \{p_j,E\} = \{p_i,p_j\}$, so that the singular fibres of
  $\pi:V \to E$ are above the points $p_i+p_j \in E$ where the
  sections meet.  These give 6 nodes in $V$.  Resolving these nodes
  give 6 fibres of type $\tA_1$ in the minimal resolution $\tilde V$ of $V$.

  Since $W = V \times_E E_n$, we obtain its singular points and the singular
  fibres of $\tW \to E_n$ by pulling back those of $V, \tilde V$,
  obtaining $6n$ of each.  In addition, the differences of the given sections
  of $V \to E$ are of order $2$, so the same is true of their pullback to
  sections of $\tW \to E_n$.
\end{proof}

\begin{cor}\label{cor:chi-w}
  The topological Euler characteristic of $W$ is $6n$, while that of $\tW$
  is $12n$.
\end{cor}

\begin{proof}
  The Euler characteristic of an elliptic surface is the sum of those of the
  singular fibres.  On $W$ these are nodal rational curves, so each has Euler
  characteristic $1$, and there are $6n$ of them.  On $\tW$, each one
  is an $I_2$ fibre, whose Euler characteristic is $2$ (alternatively, we
  obtain $\tW$ from $W$ by blowing up the $6n$ singular points,
  replacing $6n$ points by $6n$ smooth rational curves).
\end{proof}

We now study some curves on the surfaces $D_3 \times D_3$ in order to construct
sections of the elliptic surface structure of $\tW$.  This will be
essential to match our construction in this section with that of
Section \ref{sec:constr-moduli}.

\begin{defn}\label{def:h-v-c3c3}
  Recall that the $p_i$ are the points of $E$ at which the map
  $D_3 \to E$ is ramified.  Let the $P_i$ be their inverse images in
  $D_3$ and the $(Q_{i,r})_{r=0}^{d-1}$ their inverse images in $C_n$, chosen such that
  $\gamma(Q_{i,r}) = Q_{i,r+1 \bmod n}$.
  
  Let $P \in D_3$.  Define $V_P, H_P$ to be the curves on $D_3 \times D_3$
  obtained by pulling back $P$ through the first and second projections
  respectively, and let $\NS{V}, \NS{H}$ be their divisor classes up to
  algebraic equivalence (which do not depend on $P$).
  Let $V_i = V_{P_i}$ and $H_i = H_{P_i}$ for $0 \le i \le 3$.
  In addition, let $\Delta, \Gamma_\iota$ be the classes of the
  divisors of the diagonal and the graph of $\iota$ respectively.
  For $i \in \{0,1,2,3\}$, pull back $V_{P_i}$ to a
  curve on $C_n \times C_n$ and let $S_i$ be its image under $\omega$
  (the curve on $C_n \times C_n$ has $n$ components but they all have the
  same image in $W$).
\end{defn}

We now use the definition of $W$ as a quotient to identify its singularities.
\begin{prop}\label{prop:w-singularities}
  The singularities of $W$ are the images of the points
  $(Q_{i,r},Q_{j,s}) \in C_n \times C_n$ for $i \ne j$ and
  $r,s \in \{0,\dots,n-1\}$.
\end{prop}

\begin{proof} A singularity of $W$ can only occur at the image of a point of
  $C_n \times C_n$ with nontrivial $G$-stabilizer.  Further, the image of
  a point $P$ for which the fixed locus of its $G$-stabilizer is a divisor
  that is smooth at $P$ is smooth.

  In particular, the fixed locus of an element of $G$ of the form
  $g^k\sigma$ is the graph of $\gamma^k$ or $\gamma^k \beta$ on
  $C_n \times C_n$, which is a smooth curve, so singularities are only
  possible at points of intersection of two of these, which are fixed points
  of some $g^k$ or $g^k b$.
  For $k \ne 0$ these have no fixed points, so it suffices to consider the
  $16n^2$ fixed points of $\beta$, which are exactly the
  pairs $(Q_{i,r},Q_{i',r'})$.  This map acts on the tangent space of such a
  point by $-1$.  If $i \ne i'$ it is the full stabilizer, so the quotient
  is \'etale-locally isomorphic to $\A^2/\pm 1$ and we obtain an $A_1$
  singularity.

Since there are $12n^2$ pairs $(Q_{i,r},Q_{i',r'})$, each with stabilizer of
order $2$, they fall into orbits of size $2n$ and we obtain $6n$ singular
points on the quotient.  We showed in Proposition \ref{prop:w-ell-surf}
that there are no more.
\end{proof}

\begin{remark}\label{rem:disjoint-fibres}
  Suppose that $P_i + P_j \ne P_k + P_l$ on $E$ for all permutations
  $(i,j,k,l)$ of $0,1,2,3$. 
  Then the $6n$ singular points of $W$ all lie on
  distinct fibres of the map $W \to E_n$, and conversely.
  Note also that any two of the $S_i$ intersect in $n$ points; since there
  are four $S_i$, this gives $n \binom{4}{2} = 6n$ points of intersection,
  which are precisely the singularities of $W$.
  Each of the three equalities of the form $P_i + P_j = P_k + P_\ell$, if it
  holds, causes two $\tA_1$ fibres to merge to form an $\tA_3$ fibre.
\end{remark}

The $H_i, V_i$ will give us the identity and $2$-torsion sections on $W$
from another point of view.  In \cite[Section 3]{paranjape} that is
sufficient, because the K3 surfaces considered there have Picard number $16$
while the elliptic fibration constructed has six $\tA_1$
fibres and two $\tD_4$, so the Mordell-Weil group of the fibration is torsion.
In our situation with $n = 3$, however, we will find that $\tW$ has a
quotient K3 surface that comes with an elliptic fibration with nine $\tA_1$
fibres and one $\tD_4$ fibre,
and so we need to construct a section of infinite order.

\begin{prop}\label{prop:ei-unique}
  For each $i$ with $0 \le i \le 3$,
  the linear system $|K_{D_3 \times D_3} - V_i - H_i - \Delta|$ has
  projective dimension~$0$.  In other words,
  there is a unique effective divisor in the canonical
  class of $D_3 \times D_3$ whose support includes $V_i$, $H_i$,
  and the diagonal.
\end{prop}

\begin{proof}
  The diagonal is defined in $D_3 \times D_3 \subset \P^2 \times \P^2$ by
  $x_0 y_1 - x_1 y_0 = x_0 y_2 - x_2 y_0 = x_1 y_2 - x_2 y_1 = 0$.
  In this representation, the canonical class of $D_3 \times D_3$ is $\O(1,1)$.
  Choose coordinates on $\P^2$ so that the first coordinate of every $P_i$ is
  $0$.  Let $P_i = (0:a:b)$, and let $r, s, t$ be such that
  $$r(x_0y_1 - x_1 y_0) + s(x_0y_2 - x_2y_0) + t(x_1y_2 - x_2y_1)$$ vanishes
  on $V_i$ and $H_i$.  In particular, setting $y_0 = 0, y_1 = a, y_2 = b$
  we find that $(ra + sb)x_0 + t(bx_1 - ax_2) = 0$ for all points
  $(x_0:x_1:x_2) \in D_3$.  This means that all of the coefficients of the
  $x_i$ in this linear form must be $0$: in particular, $t = 0$, while
  $ra = -sb$.  Up to scaling the only possibility is that
  $r = b, s = -a, t = 0$, so $E_i$ can only be the divisor cut out by
  $b(x_0 y_1 - x_1 y_0) - a(x_0 y_2 - x_2 y_0)$.
\end{proof}

\begin{remark} The same statement is true with the diagonal replaced
  by the graph, $\Gamma_\iota$, of the involution~$\iota$, up to the changes of sign between
  the equations defining the diagonal and those defining the graph.
\end{remark}

\begin{defn}\label{def:residual}
  For $0 \le i \le 3$, let $U_i$ be the corresponding divisor
  from Proposition \ref{prop:ei-unique}.
  Let the $R_i$ be the residual component of the $U_i$.  Let $G_i \subset W$ be
  the images of the pullbacks of the $R_i$ to $C_n \times C_n$ under
  $\omega$.  Similarly, let the $U_i'$ be the effective divisors, the
  $R_i'$ the residuals, and the $G_i'$ the curves on $W$ that are obtained by
  the analogous construction with the diagonal replaced by $\Gamma_\iota$.
  Let $\NS{H}, \NS{V}, \NS{D}$ be the classes of $D_3 \times p_0, p_0 \times D_3$, and the
  diagonal in the N\'eron-Severi group of $D_3 \times D_3$ (sometimes $D$ will
  also denote the diagonal itself).
\end{defn}

In order to obtain further information about the elliptic surface structure on
$W$, we make a closer study of $D_3 \times D_3$ and the $R_i$.  First we
observe that $R_i$ was obtained from the canonical by removing one divisor of
each of the classes $\NS{H}, \NS{V}, \NS{D}$.
Since $K_{D_3 \times D_3} \sim 4\NS{H} + 4\NS{V}$,
we see that $R_i \sim 3\NS{H}+3\NS{V}-\NS{D}$.  By the adjunction formula we
have $\NS{D}^2 = -4$.  It is now easily checked that $R_i \cdot \NS{D} = 10$,
while $R_i \cdot \NS{V} = R_i \cdot \NS{H} = R_i \cdot \Gamma_\iota = 2$,
where $\Gamma_\iota$ is the graph of the involution.
Likewise $R_i' \sim 3\NS{H}+3\NS{V}-\Gamma_\iota$, so that
$R_i \cdot R_i' = 9+9-6-6+4 = 10$.

\begin{prop}\label{prop:which-pjpk-in-ri} 
  The point $(P_j,P_k) \in D_3 \times D_3$ belongs to $R_i$
  if and only if: (1) $\#\{i,j,k\}=3$, or (2) $i = j = k$ and $P_i$ is a flex
  of $D_3$.
  The same holds with $R'_i$ in place of $R_i$.
\end{prop}

\begin{proof}
  We give the proof for $R_i$, that for $R_i'$ being identical.
  We use the formula for the section of $\O(1,1)$ defining $U_i$ from
  Proposition \ref{prop:ei-unique}.  Clearly all points of the form
  $(P_j,P_k)$ belong to $U_i$.  It is routine to prove ``only if'' in case (1):
  if $j \ne k$ then $(P_j,P_k) \notin \Delta$, while $i \ne j$ implies
  that $(P_j,P_k) \notin V_i$ and $i \ne k$ tells us that $(P_j,P_k) \notin H_i$.
  Thus if all three conditions hold, then $(P_j,P_k) \in R_i$.

  Conversely, let us start with a point $(P_j,P_k)$, where $k \ne i$ (and
  either $j = i$ or $j = k$).  Now, $U_i$ and $V_k$
  have no components in common and their intersection number is $4$.
  Since $U_i \cap V_k$ contains the $4$ points $(P_j,P_k)$ for $1 \le j \le 4$,
  the local intersection multiplicity at each of them is $1$.  Thus $U_i$
  must be smooth at each of these points.  In particular, it is not possible
  for more than one component of $U_i$ to pass through $(P_i,P_k)$.  But
  $H_i$ does, so $R_i$ cannot.  Similarly for $(P_j,P_i)$ where $j \ne i$.

  Finally, we consider the point $(P_i,P_i)$.  We have
  $U_i \cdot \Gamma_\iota = 8$.
  Three of the points of intersection are the $(P_j,P_j)$ for $i \ne j$.
  Each of $H_i, V_i, D$ passes through $(P_i,P_i)$.  The two
  remaining points of intersection are the $(t_m,\iota(t_m))$, where the $t_m$
  for $m = 1, 2$ are the residual intersection points of the tangent line to $C$
  at $P_i$ with $C$ (this is easily checked: letting $P_i = (0:a:b)$ as before,
  the tangent line to $C$ at $P_i$ is defined by $bx_1 - ax_2 = 0$).  So
  $U_i \cap \Delta$ consists of these two points, which are equal to $P_i$
  if and only if $P_i$ is a flex.  (If $P_i$ is a flex it must be a hyperflex,
  in view of Lemma \ref{lem:tangent-d3}.)
\end{proof}

\begin{cor}\label{cor:ei-meet-vj}
  $U_i \cap V_j = \{(P_j,P_k),(P_j,P_\ell)\}$, where
  $\{i,j,k,\ell\} = \{0,1,2,3\}$.  Similarly for $U_i \cap H_j$.
\end{cor}

Let $\tS_i, \tG_j, \tG_k'$ be the strict transforms of
$S_i, G_j, G_k'$ on $\tW$.  In fact we have already met the $\tS_i$.

\begin{prop}\label{prop:si-two-torsion}
  The $\tS_i$ coincide with the pullback of the ramification locus of the
  map $V \to \Sym^2(E)$ (Proposition \ref{prop:w-ell-surf}) to $\tW$.
  In particular, they are sections and
  the difference of any two of them is $2$-torsion.
\end{prop}

\begin{proof} Indeed, both $\tS_i$ and the component above $\{p_i,E\}$ map to
  $\{p_i,E\} \subset \Sym^2(E)$ and are irreducible.  The result follows.
\end{proof}

\begin{prop}\label{prop:si-meet-gj}
  If $i \ne j$ then $\tS_i \cap \tG_j = \tS_i \cap \tG_j' = \emptyset$.
\end{prop}

\begin{remark}\label{rem:skip-primes}
  Here and in the following proofs the arguments required for the
  $\tG_j'$ are identical to those for the $\tG_j$, so we will not state
  them separately.
\end{remark}

\begin{proof}
  The intersection $U_i \cap V_j$ is transverse, so the same is true after
  pulling back by the \' etale map $C_n \times C_n \to D_3 \times D_3$.
  The intersection of the inverse images consists of the points of the
  form $(Q_{k,r},Q_{\ell,s})$, where $\{i,j,k,\ell\} = \{1,2,3,4\}$ and
  $r, s \in \{0,\dots,n-1\}$.  So if we blow up the preimages of the $2n^2$
  singular points of $W$ on $C_n \times C_n$, the strict transforms of the
  inverse images do not meet.  The quotient of this blowup by $G$ is exactly
  $\tW$.  By push-pull and the fact that the inverse image of
  $V_j$ consists of a complete $G$-orbit, it follows that the images of
  these inverse images on $\tW$ do not meet either.
\end{proof}

\begin{prop}\label{prop:si-meet-gi}
  We have $\tS_i \cdot \tG_i = \tS_i \cdot \tG_i' = n$.  In fact, the image of
  $\tS_i \cap \tG_i$ is the inverse image of $O \in E$ under the unramified
  cover $E_n \to E$; in particular it is independent of $i$.
\end{prop}

\begin{proof}
  We have already seen that $R_i \cdot V_i = 2$, so when we pull back to
  $C_n \times C_n$ the intersection number is $2n^2$ (this being an \' etale
  map of degree $n^2$).  The intersection locus is preserved by the
  action of $b$ and $g$ on $C_n \times C_n$, while it is disjoint from its
  image under $\sigma$.  The only fixed points of a nonidentity element of
  the group generated by $b, g$ are the $Q_{i,j}$, which are generically not
  in the intersection, as we have seen.  Thus this group acts on the set
  of $2n^2$ points with orbits of size $2n$, and so there are $n$ orbits.
  It also follows that the intersection
  points are smooth, so that the number of points of intersection of $\tS_i$ and
  $\tG_i$ on $\tW$ is the same as that of $S_i$ and $G_i$ on $W$.

  To obtain the more precise information in the second statement, we appeal
  to Lemma \ref{lem:tangent-d3}.  The intersection $R_i \cdot H_i$ is
  linearly equivalent to $\O_{H_i}(1) - (H_i + V_i + D) \cdot H_i$.  Here
  the first is the canonical, while the second consists of two copies of
  the point $(P_i,P_i) \in H_i$.  Since it is an effective divisor of degree
  $2$, and plane quartics are canonically embedded, it can only be the
  residual intersection of the tangent line, so it consists of the
  two points $(P_i,P^-_{i,j})$ for $j = 1, 2$, where the $P^-_{i,j}$ are the two
  points of $D_3$ with image $-p_i$ on $E$.  When we pull back to $C_n$,
  therefore, the first coordinate is an inverse image of $P_i$ and the second
  is an inverse image of $P^-_{i,j}$, so their images on $E_n$ are inverse
  images there of $\pm p_i$.  The sums of these are the points of $E_n$
  in the kernel of the isogeny $E_n \to E$.
\end{proof}

\begin{prop}\label{prop:gi-meet-giprime}
  We have $\tG_i \cdot \tG_i' = n$, while
  $\tG_i \cdot \tG_j' = 2n$ for $i \ne j$.
\end{prop}

\begin{proof}
  First, for $\tG_i \cdot \tG_i'$, we computed just before
  Proposition \ref{prop:which-pjpk-in-ri} that $R_i \cdot R_i' = 10$, and in
  that proposition that
  $6$ of the special points $(P_j,P_k)$ lie on the intersection.  Thus we
  have $10n^2$ points of intersection when we pull back to $C_n \times C_n$,
  of which $6n^2$ map to singular points of $W$.  When we blow up to pass
  to $\tW$, these intersections are pulled apart and do not contribute
  to $\tG_i \cdot \tG_i'$, leaving $4n^2$ points of intersection on
  $C_n \times C_n$.  The degree-$4n$ map to $W$ combines these into
  $4n^2/4n = n$ points, since as in the argument for Proposition
  \ref{prop:si-meet-gi} just above they are not fixed by any element of $G$
  other than the identity.

  Likewise, for $\tG_i \cdot \tG_j'$, we again have $R_i \cdot R_j' = 10$,
  but this time only $2$ of the special points are on the intersection.
  Thus we are left with $10n^2-2n^2$ points of intersection on $C_n \times C_n$
  that give rise to intersection points on $\tW$, and they are combined in
  sets of $4n$, so the intersection number is $2n$.
\end{proof}

\begin{prop}\label{prop:w-sections}
  The $S_i$, the $G_i$, and the $G_i'$ are sections of the map
  $W \to E_n$.
\end{prop}

\begin{proof}
  We have already explained this for the $S_i$ in Proposition
  \ref{prop:si-two-torsion}, so we consider the $G_i$.
  Let $F$ be a fibre of the map
  $D_3 \times D_3 \to E$, i.e., a curve whose points are the $(x,y)$ such
  that the images of $x,y$ on $E$ have a given sum.  Since the map
  $D_3 \to E$ has degree $2$, we have $F \cdot \NS{H} = F \cdot \NS{D} = 2$, while
  $F \cdot \Delta = F \cdot \Gamma_\iota = 8$ (here using the fact that
  doubling has degree $4$ on $E$ and that $\iota$ preserves the map to $E$).
  Thus, if $NS(D_3 \times D_3)$ is generated by $\NS{H}, \NS{V}, \Delta, \Gamma_\iota$,
  we easily find that $F$ has class $4(\NS{H}+\NS{V}) - \Delta - \Gamma_\iota$, and this
  follows in general by writing down the equation defining
  $F \cup \Delta \cup \Gamma_\iota$.  We now compute that $F \cdot R_i = 4$.
  Thus the pullbacks to $C_n \times C_n$ meet in $4n^2$ points, and,
  since the pullback of $F$ consists of complete $G$-orbits, their images
  on $W$ meet in $4n^2/\#G = n$ points.  That is to say, the curve $G_i$ meets
  a fibre of the map $W \to E$ in $n$ points.  Such a fibre consists of
  $n$ fibres of the map $W \to E_n$, and the intersection with each of them
  must be the same since the fibres are algebraically equivalent: it is
  therefore $1$.
\end{proof}

Henceforth we assume the condition of Remark \ref{rem:disjoint-fibres},
which holds on an open subset of the moduli space.

\begin{remark}\label{rem:negative-sections}
  The sections $\tilde G_i'$ are the negatives of the sections $\tilde G_i$
  if one of the $\tilde S_j$ is taken as origin.  We will prove this in
  Proposition \ref{prop:char-lambda}.
\end{remark}

\begin{prop}\label{prop:bad-fibs-w}
  The curves $\tG_i, \tS_i$ pass through different components of every one
  of the $\tA_1$ fibres, and likewise $\tG_i',\tS_i$.
  Under the assumption of Remark \ref{rem:disjoint-fibres}, for
  $i \ne j$, there are $2n$ fibres of type $\tA_1$ for which $\tS_i$ and
  $\tS_j$ meet the same component.  On these fibres the other two $\tS$
  pass through the other component.
\end{prop}

\begin{proof}
  For the first statement, it suffices to use the fact (Proposition
  \ref{prop:which-pjpk-in-ri}) that $(P_j,P_k) \in R_i$ if and only if
  $i \notin \{j,k\}$, which implies that each singular point lies on
  exactly one of $\tG_i,\tS_i$ for each $i$.

  For the second statement, note that $\tS_i, \tS_j$ can be on the same
  component in two ways: either they both pass through the singular point
  or neither does.  As in Remark \ref{rem:disjoint-fibres}, there are $n$
  of each type, and if $\tS_i, \tS_j$ meet in a point then that point is not
  on $\tS_k, \tS_\ell$ (where $\{i,j,k,\ell\} = \{0,1,2,3\}$).
\end{proof}

\begin{remark}\label{rem:canonical-w}
  The map induced by the canonical divisor class on $\tW$ is precisely the
  fibration $\phi: \tW \to E_n$ for an appropriate embedding of $E_n$ into
  $\P^{n-1}$. Indeed, let $F$ be a fibre.  Then the adjunction formula shows that
  $(K_{\tW}+F) \cdot F = K_F$, and both $K_F$ and $F \cdot F$ are $0$ as
  divisor classes on $F$.  On the other hand, the fundamental line bundle
  ${\mathbb L}$ \cite[Definition II.4.1]{miranda} has degree $n$, since
  $\tW$ has Euler characteristic $12n$.  The canonical bundle of
  $\tW$ is $\phi^*(K_{E_n} \otimes {\mathbb L}) = \phi^*({\mathbb L})$
  \cite[III.1.1]{miranda}, and so its intersection with a section of $\phi$
  (such as one of the $\tG_i$ or $\tS_i$ defined earlier) is $n$.
  It therefore embeds the section as an elliptic normal curve in $\P^{n-1}$
  isomorphic to $E_n$,
  and, since the fibres of $\phi$ are contracted to points, the image
  of $W$ is the same curve.
\end{remark}

\subsection{A K3 quotient of $W$}
We have constructed a singular surface $W$ whose minimal desingularization
$\tW$ is an elliptic surface over $E_n$.  We would like to construct an
elliptic K3 surface that is a quotient of $\tW$.  In order to do this, we
need to construct a group of automorphisms of $\tW$ that includes an element
that acts as negation on the base $E_n$, since a K3 surface does not have
a nonconstant map to an elliptic curve.  We will do this by relating $W$
to the symmetric square of $E$, which is a $\P^1$-bundle over $E$, and to its
pullback to $E_n$.

\begin{defn}\label{def:tsym}
  Let $\tsym$ be the blowup of $\Sym^2(E)$ at the points
  $\{p_i,p_j\}$ for $i \ne j$.  Let the $\Sigma_i$ be the sections
  $x \to \{p_i,x\}$
  of $\Sym^2(E) \to E$ and the $\tSigma_i$ their proper
  transforms on $\tsym$ (which are still sections of $\tsym \to E$).
  Let $E_{ij} \subset \tsym$ be the
  exceptional divisor above the point of intersection of
  $\tSigma_i \cap \tSigma_j$ and let $F_{ij}$ be the other component of the
  fibre there.
\end{defn}

We will construct a birational automorphism on the double cover $\tilde V$
(Definition~\ref{def:g-w}, Proposition~\ref{prop:w-ell-surf}) of
$\tsym$ and pull it back to $W$.  The idea is that
$\tsym$ is itself a double cover of
$\Mbar_{0,5}$, and so $V$ covers $\Mbar_{0,5}$ with degree~$4$.
The construction is closely related to that of Construction \ref{cons:rat-surf}.
We will show that
the Galois group of the cover is the Klein four-group, so that there is an
additional involution that acts nontrivially on the base and by which the
quotient is not $\Sym^2(E)$.

\begin{lemma}\label{lem:self-int-sym2}
  The self-intersection of $\Sigma_i$ on $\Sym^2(E)$
  is the divisor class $\{p_i,p_i\}$ on $\Sigma_i$.  On $\tsym$, the
  self-intersection of $\tSigma_i$
  is $\{p_i,p_i\} - \sum_{j=0,i \ne j}^3 \{p_i,p_j\}'$, where $\{p_i,p_j\}'$
  is the point of the exceptional divisor above $\{p_i,p_j\} \in \Sym^2(E)$
  that lies on $\tSigma_i$.
\end{lemma}

\begin{proof}
  On $\Sym^2(E)$ one checks easily that the self-intersection number is $1$,
  so it suffices to identify the point.
  Let $p' \ne p_i \in E$ correspond to the section $x \to \{p',x\}$.  Clearly
  these two sections intersect at the point $\{p_i,p'\}$, so the same remains
  true in the limit $p' \to p_i$.

  We consider the map $\tsym \to \Sym^2(E)$ and apply the push-pull formula,
  finding that $\tSigma_i \cdot (\tSigma_i + \sum_j E_{ij}) = \Sigma_i^2$.  But
  $\tSigma_i \cdot E_{ij} = \{p_i,p_j\}'$.  The result follows.
\end{proof}

\begin{lemma}\label{lem:prin-div-v}
  Let $q$ be a point of $E$ such that $2q = p_0 + p_1 + p_2 + p_3$
  in the group law of $E$.  Let $F_q$ be the fibre of $\Sym^2(E) \to E$
  lying above $q$.  Then the divisor
  $$D_g = \sum_{i=0}^3\tSigma_i - 2(\tSigma_0 + \tSigma_1 + E_{01} + F_{23} - F_q)$$
  on $\tsym$ is principal.
\end{lemma}

\begin{proof}
  The Picard group of $\tsym$ is an extension of the N\'eron-Severi group
  by $\Pic^0$.  The N\'eron-Severi group
  is generated by the classes of a section, a fibre, and the exceptional
  curves.  On the other hand, $\Pic^0$ is a $1$-dimensional abelian variety
  and is
  easily seen to be isomorphic to $E$ by the map $\tsym \to \Sym^2(E) \to E$.

  The $\tSigma_i$ have self-intersection~$-2$ and 
  $\tSigma_i$ meets $E_{jk}$ in degree $1$ if $i \in \{j,k\}$ or $0$ otherwise.
  Also $(E_{ij},F_{k\ell}) = 1$ for $\{i,j\} = \{k,\ell\}$, else $0$.
  Since the $E_{ij}$ are $-1$-curves, it is now clear that $(D_g,E_{ij}) = 0$
  for all $i,j$; that $D_g$ has degree $0$ on fibres is also obvious.

  In addition, we easily compute that $(D_g,\tSigma_0) = 0$, but we need to
  check that the intersection is the principal divisor class on $\tSigma_0$.
  The self-intersection of $\Sigma_0$ on $\Sym^2(E)$
  is the divisor $\O(P)$, where $P$ is the point $\{p_0,p_0\}$ on $\Sigma_0$.
  When we blow up the points $(p_0,p_i)$ for $1 \le i \le 3$, the
  self-intersection becomes $\O(p_0-p_1-p_2-p_3)$.  On the other hand, the
  other $\tS_i$ are disjoint from $\tS_0$, and $E_{01}, F_{23}$ intersect it
  in $\{p_0,p_1\}$ and $\{p_0,p_2+p_3-p_0\}$.  Finally, $F_q, F_0$ intersect
  in $\{p_0,q-p_0\}, \{p_0,-p_0\}$.  Thus the intersection of $D_g$ with
  $\tSigma_0$ is given by the divisor
  $-(p_0)+(p_1)+(p_2)+(p_3)-2(p_1)-2(p_2+p_3-p_0)+(2q-p_0)$.
  This is a divisor of degree $0$ and the points add to $0$, so it is
  a principal divisor.
\end{proof}
  
We would like to say that $V$ is the double cover of $\tsym$ ramified
along the $\tSigma_i$ obtained by adjoining a
square root of the function with divisor $D_g$.
However, this is only well-defined when the point $q$ is specified.
In fact this is the same choice that we made in Definition
\ref{def:d3} where we defined a double cover of $E$ branched at
$p_0, \dots, p_3$: choosing such a double cover is equivalent to defining a
function with divisor $(p_0)+(p_1)+(p_2)+(p_3)-2D$, where $D \in \Pic^2(E)$.
This is equivalent to choosing the origin (as we did there) or $q$ (as we
do here).

We now observe that the $\tSigma_i$ are the pullbacks of the special sections
of $\Mbar_{0,5} \to \Mbar_{0,4}$, while the map $E \to \Mbar_{0,4}$ induced by
the family maps both $p_i+p_j$ and $p_k+p_\ell$ to the boundary point
$\delta_{ij} = \delta_{k\ell}$.  Accordingly it is the quotient by the
negation map $p \to \sum_{i=0}^3 p_i - p$.
The fibres fixed as sets by this negation map are those at
the four points $(\sum p_i)/2$.  Comparing the canonical divisors, we see that
all four of these are ramified in the cover.

\begin{prop}\label{prop:gal-dc-v}
  The extension of function fields corresponding to the cover
  $V \dashrightarrow \Mbar_{0,5}$ is Galois and its Galois group is the
  Klein four-group.
\end{prop}

\begin{proof}
  The $\tSigma_i$ are pullbacks of the special sections
  $\Sigma_i: \Mbar_{0,4} \to \Mbar_{0,5}$, while $E_{01}+F_{23}$ is the pullback of
  $E_{01} \subset \Mbar_{0,5}$; the discussion just above shows that
  $F_q$ is the pullback of a single fibre above the image $q'$ of $q$ in
  $\Mbar_{0,4}$.  Accordingly the divisor $D_g$ above is the pullback of the
  divisor $\sum_{i=0}^3 E_i - 2(E_0 + E_1 + E_{01}) + F_{q'}$ on $\Mbar_{0,5}$
  (recall the notation introduced in Definition~\ref{def:m-0-5-bar}),
  which has intersection~$0$ with all $E_{ij}$ and $E_i$ and is therefore
  principal.  Let $g \in k(\Mbar_{0,5})$ be a function with this
  divisor (in terms of the interpretation of $\M_{0,5}$ as a blowup of
  $\P^2$, the fibre $F_{q'}$ is defined by a quadratic form $Q$, while
  $E_{01}$ is defined by a linear form $L$; the pullback of $Q/L^2$ to
  $\M_{0,5}$ is the desired function).
  Then the cover $\tilde V \to \tsym$ is given by adjoining a square root of
  $g$.  Let the $F_{p_i}$ be the fibres above the $p_i$, let $D$ be a divisor
  which is the sum of two arbitrary fibres, and let $f$ be a function with
  divisor $\sum_{i=0}^3 F_{p_i} - 2D$.  Then we have
  $$k(V) = k(\Sym^2(E))(\sqrt g) = k(\Mbar_{0,5})(\sqrt f, \sqrt g).$$
  Since $f, g \in k(\Mbar_{0,5})$, and
  clearly their ratio is not a square, the result follows.
\end{proof}

Let $\nu_V$ be the negation automorphism of the elliptic surface $\tilde V$
(the quotient of this is $\tsym$, whose function field is
$k(\Mbar_{0,5})(\sqrt f)$, so at the level of fields it is given by
$\sqrt{f} \to \sqrt{f}, \sqrt{g} \to -\sqrt{g}$).
Let $\lambda_V$ be the involution of $\tilde V$ induced by the automorphism
$\sqrt{f} \to -\sqrt{f}, \sqrt{g} \to \sqrt{g}$.

\begin{remark}\label{rem:fixed-fibres}
  Since the divisor
of $g$ has odd multiplicity along $F_{q'}$, the fibre above $q'$ is unramified
in the quotient map $\tilde V \to \tilde V/\lambda_V$, while the other
three fibres above points $(\sum p_i)/2$ are ramified in that extension.
The opposite is true for $\lambda_V \nu_V$.
\end{remark}

We are now ready to consider $W$.  This involves changing the base from
$\Mbar_{0,5}$ to $C_0 \times_{\Mbar_{0,4}} \Mbar_{0,5}$; the covers are
$\Sym^2(E) \times_E E_n$ and $\tW = \tilde V \times_E E_n$
(cf.{} Figure \ref{diag:vars}).
The choice of $\lambda_V$ determines an origin on $E$, which will be denoted
$O_\lambda$.

\begin{defn}\label{def:lambda}
  We define an automorphism
  $\lambda = \lambda_{W,Q}$ of $\tW$
  as the map corresponding to the maps
  $\tW \to \tilde V: \lambda_V \circ \pi_1$ and
  $\tW \to E_n: (-1)_{E_n} \circ \pi_2$
  by the universal property of a product, where $(-1)_{E_n}$ is the negation
  on $E_n$ with origin $Q$.  (Note that $\lambda$ induces the
  negation on $E$, so we must compose with $(-1)_{E_n}$ to obtain maps that
  make the diagram commute.)
\end{defn}

We restate this, together with information on the fixed fibres, as a theorem.

\begin{thm}\label{thm:w-inv}
  For all lifts $Q$ of $O_\lambda$ to $E_n$, the negation map of $E_n$ with
    origin at $Q$ lifts to an automorphism $\lambda_Q$ of
    $\tW$ that preserves the elliptic fibration and the sections $\tS_i$.
    The map $\lambda$ acts as $-1$ on the fibres at the $2$-torsion points of
    $E_n$ relative to the origin $Q$ that are not in the kernel of the isogeny
    $E_n \to E$ and as $+1$ on those that are.
\end{thm}

\begin{proof} The automorphism $\lambda_Q$ is that of Definition
  \ref{def:lambda}.  The automorphism of $E_n$ induced by $\lambda_W$ lifts
  the negation on $E$ induced by $\lambda$.
  The only fibres of $\tW \to E_n$ that are fixed as sets by $\lambda$ are
  those above the $2$-torsion of $E_n$ relative to the chosen origin, and these
  are fixed pointwise or not according as the fibres over the images of these
  points in $E$ are fixed pointwise or not (clear from the description of the
  map as a fibre product).
\end{proof}

\begin{remark}\label{rem:lift-n-torsion-trans}
  In particular, the automorphism of $E_n$ given by translation by
  the difference of two points mapped to the same point by the isogeny to $E$
  also lifts to an automorphism of $\tW$.
\end{remark}

Let $\nu = \nu_W$ be the negation automorphism of $\tW$ with respect to
one of the $\tSigma_i$ as zero section (in light of Proposition
\ref{prop:si-two-torsion}
they all give the same negation automorphism).
Let $\alpha_P: \phi^{-1}(P) \to \phi^{-1}(-P)$ be an isomorphism 
that takes the intersection of the zero section with $\phi^{-1}(P)$ to
its intersection with $\phi^{-1}(-P)$.  Since $\lambda$ preserves the
zero section, it must identify $\phi^{-1}(P)$ with $\phi^{-1}(P)$ either
by $\alpha_P$ or by $-\alpha_P$.  In either case it commutes with $\nu$.
Because $\lambda$ and $\nu$ commute on an open subset of $\tW$, they commute
on all of $\tW$ and generate a Klein four-group.

\begin{prop}\label{prop:euler-char-quotient}
  Both $\tW/\langle\lambda\rangle$ and $\tW/\langle\lambda \nu\rangle$ have
  $3n$ fibres of
  type $I_2$.  If $n$ is even then both have two fibres of type $I_0^*$, and
  if $n$ is odd then $\tW/\langle\lambda\rangle$ has one and $\tW/\langle\lambda \nu\rangle$
  has three.  There are no other singular fibres.
\end{prop}

\begin{proof} Recall that $\tW$ has $6n$ fibres of type $I_2$, of
  which $n$ contain the strict transform of a point of
  $\tSigma_i \cap \tSigma_j$ for each subset $\{i,j\} \subset \{0,1,2,3\}$
  of order $2$.  Both $\lambda$ and $\lambda \nu$ identify these fibres for a
  given $\{i,j\}$ with those for the complement, so the quotient has $3n$
  such fibres.  These are the only singular fibres of $\tW$ and
  the image of a smooth fibre cannot be singular in the quotient
  unless it is fixed setwise and the involution acts as negation, in which
  case we obtain an $I_0^*$ fibre.  (The quotient of an elliptic curve by
  negation is a rational curve, which will be double in the quotient fibration,
  and we need to blow up the four fixed points.)

  As in Theorem \ref{thm:w-inv}, then, the number of points of order $2$ of
  $E_n$ that are (respectively, are not) in the kernel of the map $E_n \to E$
  is the number of $I_0^*$ fibres of $\tW/\langle\lambda\rangle$ (resp.~$\tW/\langle\lambda \nu\rangle$).
  Clearly this is as claimed.
\end{proof}

\begin{cor}\label{cor:h20-quotient}
  The desingularizations of $\tW/\langle\lambda\rangle$ and $\tW/\langle\lambda \nu\rangle$
  have Euler characteristics
  $12 \lceil \frac{n-1}{2} \rceil, 12 \lfloor \frac{n+1}{2} \rfloor$
  respectively, and they have
  $h^{2,0} = \lceil \frac{n-1}{2} \rceil, \lfloor \frac{n+1}{2} \rfloor$.
  In particular $\tW/\langle\lambda\rangle$ is a K3 surface for $n = 2, 3$.
\end{cor}

\begin{proof} For the first statement, we recall that the Euler characteristic
  of an elliptic surface is the sum of those of the singular fibres and
  use the fact that $I_0^*$ and $I_2$
  fibres have Euler characteristic $6, 2$ respectively.
  Alternatively, we could compute
  this by determining the Euler characteristics of $\tW$ and of the
  fixed loci of $\lambda, \lambda \nu$.
  The second follows from the first by applying \cite[(III.4.2)]{miranda}.
\end{proof}


\begin{remark}\label{rem:n-2}
  In the case $n = 2$, our construction gives two K3 surfaces.  In addition,
  we have two abelian surfaces, namely the Prym varieties of the double
  covers $C_2 \to D_3$ and $D_3 \to E$.  It is natural to expect a relation
  between these, and indeed we verified in some examples
  that the K3 surfaces are isogenous in the sense of Definition
  \ref{def:isogeny}
  to the Prym
  varieties of these two covers.  In any case, they come with elliptic
  fibrations that have two $\tD_4$ and six $\tA_1$ fibres, and so the results
  of \cite{paranjape} already show that they are covered by the square of a
  curve of genus~$5$ and construct an explicit correspondence between them
  and the square of the Kuga-Satake varieties.

  The $\tA_1$ fibres are above the same points of $\P^1$.  If we twist to
  construct a surface $E_2$ with $\tD_4$ fibres above the four points where
  one of the surfaces
  $\tW/\langle\lambda\rangle, \tW/\langle\lambda \nu\rangle$ has such
  a fibre, we obtain
  an elliptic surface with $h^{2,0} = 2$ and Kodaira dimension~$1$.
  It turns out that the double cover $E_2 \to E$ in this case induces an
  involution~$\mu$ of $W_2$ that acts nontrivially on the base.  Letting $\nu$
  be the negation map for the fibration, we find that $W_2/\langle\mu\rangle$,
  $W_2/\langle\mu \nu\rangle$
  are both K3 surfaces that come with an elliptic fibration with (generically)
  three fibres of type $I_0^*$ and three of type $I_2$, full level-$2$
  structure, and Picard rank $17$.  The quotient of such a fibration by a
  $2$-torsion translation has three $I_0^*$ fibres, an $I_4$, and a $2$-torsion
  point.  Under the genericity assumption that the rank is $17$, this
  determines the Picard lattice completely.  Since the same reducible fibres
  and torsion subgroup arise for fibration type $4$
  in the list in \cite{kumar}, the quotients are the Kummer surfaces of
  principally polarized abelian surfaces.
  These surfaces are not quotients of $\tW$ and there appears to be no
  obvious connection between the corresponding abelian surfaces and the
  curves of our construction; nevertheless, we do obtain a correspondence
  on a subvariety of the square of the moduli space of curves of genus~$2$.
\end{remark}

\begin{remark}\label{rem:n-4}
  In the case $n = 4$, we obtain, as quotients of $\tW$, two surfaces
  with $h^{2,0} = 2$ and Kodaira dimension~$1$.  These admit involutions
  by which the quotients are isogenous to Kummer surfaces as above.
  If we twist to remove the $\tD_4$ fibres, we obtain a K3 surface,
  with an elliptic fibration with $12$ fibres of type $I_2$ and
  generic Mordell-Weil rank $1$: compare Remarks \ref{rem:n-3-and-beyond}
  and \ref{rem:motive-finite-ell-surfs}.
  Such a K3 surface has Picard number $15$, and we would not expect
  it to arise from a construction such as this one.  In Section
  \ref{sec:constr-moduli} we constructed a special family of surfaces of this
  type with generic Mordell-Weil rank $2$; however, we will show in
  Proposition \ref{prop:n-4-already-known} that
  these are isogenous to double covers of $\P^2$
  branched along six lines, so the fact that they are quotients of the square
  of a curve already follows from the main result of \cite{paranjape}.
\end{remark}
  
\begin{remark}\label{rem:various-n} In \cite{paranjape}, in which $n = 2$,
  the choice of $\lambda$ is less relevant
  because $n_\lambda = 2$ for both possibilities
  and Paranjape is not concerned
  with the degree of the map of moduli spaces.  Thus it is not necessary
  for him to distinguish between $\lambda$ and $\lambda \nu$,
  as it is for us.

  For $n = 5$, we obtain an elliptic surface over $\P^1$ with $15$ fibres
  of type $I_2$ and one of type $I_0^*$.  We proved by checking a single
  example that there is no automorphism
  of the base that permutes the $I_2$ fibres, so there is no automorphism of
  the surface that could have a K3 surface as its quotient (recall that the
  elliptic fibration is the map associated to the canonical divisor, so it
  is preserved by all automorphisms of the surface).  Furthermore,
  we have found by numerical
  calculation for small primes that the characteristic polynomial of
  Frobenius acting on the transcendental lattice, whose dimension
  is $12$, is often absolutely irreducible.  That is to say, it remains
  irreducible when any power of the variable is substituted for the variable;
  it follows that the Galois representation on the transcendental lattice
  over any finite field of the appropriate characteristic is irreducible
  and hence that there can be no nontrivial correspondence to a K3 surface.

  In the case $n = 6$, we obtain a curve $C_n$ of genus~$13$ which is an
  unramified cover of degree $6$ of $D_3$.  The intermediate covers
  $C_{3,2}, C_{3,3}$ have genus $5, 7$, so the quotient
  $\Jac(C_n)/(\Jac(C_{3,2})+\Jac(C_{3,3}))$ is of dimension $13-5-7+3 = 4$
  and we expect to obtain a family of K3 surfaces of Picard number $16$.
  In terms of our construction here, we obtain surfaces with $h^{2,0} = 3$.
  These surfaces should map to those given by the $n = 2, n = 3$ constructions,
  which suggests that there should be a third quotient which has $h^{2,0} = 1$
  and is therefore a K3 surface.  We expect that it will be a surface already
  covered by our main result, Theorem \ref{thm:main-cover}, since the Hodge
  structure will admit an action of an order of $\Q(\sqrt{-3})$ rather
  than of some other quadratic ring (cf.~Section~\ref{sec:hodge}).
\end{remark} 

We will apply the following alternative characterization of $\lambda$ in
the case $n = 3$.  However, the proof is valid for all $n > 2$.
\begin{prop}\label{prop:char-lambda}
  For $n>2$, the $\tG_i$ and $\tG_i'$ (Definition \ref{def:residual}) are
  preserved by $\lambda$, whereas $\lambda \nu \tG_i = \tG_i'$
  and vice versa.
\end{prop}

\begin{proof}
  First we show that $\lambda \nu \tG_i \ne \tG_i$: this is the only
  part of the proof that requires $n > 2$.  If they were equal,
  then $\tG_i$ would have to meet the fibres fixed as sets
  but not pointwise by $\lambda \nu$ in a point on one of the $\tS_i$,
  these being the only points on such fibres fixed by $\lambda \nu$.
  Since $\tG_i$ and $\tS_j$ are disjoint for $i \ne j$, and
  $\tG_i \cdot \tS_i = n$, this would only be possible if all $n$
  points of intersection were on these fibres.  This cannot happen,
  because the fibres differ by $2$-torsion whereas the images of the
  points of intersection in $E_n$ differ by $n$-torsion.

  Next we show that $\nu \tG_i = \tG_i'$ and $\nu \tG_i' = \tG_i$.
  Note that $\nu$ lifts to the automorphism
  of $C_n \times C_n$ that acts as the identity on one copy of $C_n$ and
  $\beta$ on the other.  This in turn descends to the automorphism of
  $D_3 \times D_3$ that acts as the identity on one copy and $\iota$ on the
  other.  This automorphism exchanges $R_i$ and $R'_i$ and the claim follows.

  Finally we prove that the set $\{\tG_i,\tG_i'\}$ is preserved by $\lambda$
  and $\lambda \nu$.  Let $L_i = \lambda(\tG_i)$ and consider the intersection
  matrix $M_n$ of the divisors $\tS_i, \tG_i, \tG_i', L_i, F, A_1, \dots, A_{6n}$,
  as summarized in Table \ref{eqn:intersections},
  where $F$ is the class of a fibre and the $A_i$ are the nonzero components
  of the $\tA_1$ fibres of the fibration.  We know all of the intersections
  among these divisors, except for $\tG_i \cdot L_i$ and
  $\tG_i' \cdot L_i$, as follows:
  \begin{itemize}
  \item We saw in Proposition \ref{prop:si-meet-gi} that
    $\tS_i \cdot \tG_i = \tS_i \cdot \tG_i' = n$.  Since $\lambda(S_i) = S_i$,
    it follows that $\tS_i \cdot L_i = n$ as well.
  \item In Proposition \ref{prop:gi-meet-giprime} we found that
    $\tG_i \cdot \tG_i' = n$.
  \item In Remark \ref{rem:canonical-w} we showed that $K_{\tW} \sim nF$.
    By the adjunction formula, we conclude that the self-intersection of
    every section is $-n$ (recall that the base is an elliptic curve, so
    the canonical divisor of a section is $0$).  Likewise the self-intersection
    of a fibral rational curve is $-2$.
  \item As usual we have $F^2 = F \cdot A_i = A_i \cdot A_j = 0$ for $i \ne j$.
  \item By definition $\tS_i \cdot A_j = 0$ for all $j$.  On the other hand 
    $\tG_i \cdot A_j = \tG_i' \cdot A_j = 1$ for all $j$ by Proposition
    \ref{prop:bad-fibs-w}.
  \end{itemize}

  Let $\tG_i \cdot L_i = a, \tG_i' \cdot L_i = b$.
  The divisor $\tS_i + \tG_i + \tG_i'$
  has self-intersection~$3n$, so by the Hodge index theorem $M_{n}$ cannot have
  negative determinant.  We now show that
  $\det M_{n} = -n \cdot 2^{6n} \cdot (a+b)^2$.  Add $1/2$ the sum of 
  the last $6n$ rows to rows $2, 3, 4$, which
  does not change the determinant but makes the matrix block triangular.
  The bottom right block is $-2I_{6n}$ of determinant $2^{6n}$, and the
  top left block is a $5 \times 5$ matrix whose determinant is checked to
  be $-n(a+b)^2$.  Since the determinant must be nonnegative, the conclusion is that the determinant
  is $0$ and so $a \le 0$ or $b \le 0$.
  \begin{table}
    $\begin{array}{c|rrrrr|rrr}
     & \tS_i & \tG_i & \tG_i' &  L_i &  F &  A_1 &  \dots &  A_{6n} \\
    \hline
\tS_i &    -n    & n    & n      &   n   & 1 &0 & \cdots & 0 \\
\tG_i  &  n    & -n    & n      &    a    & 1 & 1& \cdots & 1 \\
\tG_i' &    n    & n    & -n      &   b     & 1 & 1 & \cdots & 1 \\
L_i   &  n    &   a  &   b      &  -n    & 1  & 1 & \cdots & 1 \\
F &  1        &  1   &  1       & 1     & 0  & 0  & \cdots & 0  \\
\hline
A_1 &  0 & 1  & 1      & 1      & 0  &  &  & \\
\vdots & \vdots  & \vdots & \vdots &  \vdots    &   &  & -2I_{6n} &  \\
A_{6n} &  0 & 1  & 1 & 1           & 0 & &  &  
    \end{array}$
    \caption{Intersections of divisors on $\tW$}\label{eqn:intersections}
  \end{table}

  It is not possible for $\tG_i \cdot L_i$ to be $0$:
  these are both irreducible
  curves with nonempty intersection (at the points of $\tG_i \cap \tS_i$),
  so that would require $\tG_i = L_i$ and $\tG_i^2 = 0$.  The second of these
  is false.  Therefore
  either $a < 0$ and $L_i = \tG_i$, or $b < 0$ and $L_i = \tG_i'$.  We ruled
  out the second choice above, so the first must hold: that
  is, $\lambda(\tG_i) = \tG_i$.
\end{proof}

\begin{defn}\label{def:quo-lambda}
  Let $K = \tW/\langle\lambda\rangle$.
  Let $\pi_K: K \to \P^1$ be the induced elliptic fibration on $K$ arising
  from the fibration $\tW \to E_n$.  Let $T_i, J_i, J_i'$
  be the images of the
  curves $\tS_i, \tG_i, \tG_i'$, the strict transforms of the
  $S_i, G_i, G_i'$ on
  $\tW$, on $K$ (recall that the $S_i, G_i, G_i'$ were
  defined in Definitions \ref{def:h-v-c3c3} and \ref{def:residual}).
\end{defn}

By construction, relative to the origin $O$ on $E$ we have
$p_i + p_j = -(p_k + p_l)$ when $\{i,j,k,l\} = \{1,2,3,4\}$, so for all lifts
$q_i, q_j, q_k$ of $p_i, p_j, p_k$ to $E_n$ there is a lift $q_l$ of $p_l$ with
$q_i + q_j = -(q_k + q_l)$.  As these are the points lying under the
singular fibres of $\tW \to E_n$, these fibres are identified in pairs by
the quotient map $\tW \to \tW/\langle\lambda\rangle$.  In addition, 
the quotient map takes every fibre that is negated to a rational curve
along which there are four singularities: we thus introduce fibres of type
$I_0^*$ there.  So we have proved:

\begin{thm}\label{thm:image-is-k3}
  The elliptic fibration $\pi_K$ has $3n$ reducible fibres of type $I_2$.
  It also has one of type $I_0^*$ when $n$ is odd and two when $n$ is even.
\end{thm}

\begin{prop}\label{ji-sections}
  The $J_i, J_i'$ are sections of the induced fibration $K \to \P^1$.
\end{prop}

\begin{proof}
  More generally, we consider the following situation.  Let $X_1 \to C_1$
  be an elliptic surface, and let $\alpha$ be an
  involution of $X_1$ preserving the set of fibres and acting nontrivially
  on $C_1$; let $X_2 = X_1/\alpha$,
  so that $X_2$ has an induced elliptic fibration to $C_1/\alpha$,
  and let $\pi: X_1 \to X_2$ be the
  quotient map.  Let $F_1, F_2$ be the classes of a fibre on $X_1, X_2$
  respectively.  Then $\pi_*(F_1) = F_2$ and $\pi^*(F_2) = 2F_1$ up to
  algebraic equivalence.  Thus if $S_1$ is a section on $X_1$, we have
  $2 = (S_1,\pi^* F_2) = (\pi_* S_1,F_2)$.  In other words, $S_1$ descends
  to a section of $X_2 \to C_1/\alpha$ if and only if $\alpha(S_1) = S_1$
  (the condition that $\alpha$ act nontrivially on $C_1$ prevents
  a section from being ramified in the quotient).

  The claim now follows from Proposition \ref{prop:char-lambda}.
\end{proof}

We now choose $T_1$ as the zero section of $\pi_K$.
Let $z_i, a_i$ be the curves in the $I_2$ fibres of $K$ through which $T_1$ does
(respectively does not) pass.

\begin{lemma}\label{lem:which-a-b}
  \begin{enumerate}
  \item For $j = 2, 3, 4$, the curves $T_j$ pass through
    $2n$ of the $a_i$ and $n$ of the $z_i$.
  \item The curves $J_j$ pass through exactly those of the $z_i, a_i$
    that are disjoint from $T_j$.
  \end{enumerate}
\end{lemma}

\begin{proof}
  This follows from Proposition \ref{prop:bad-fibs-w}.  The $6n$ singular
  fibres are identified by $\lambda$ in pairs in such a way that a fibre
  containing the exceptional divisor above a point in $\tS_i \cap \tS_j$
  goes to one coming from an intersection point of $\tS_k \cap \tS_\ell$
  (where again $\{i,j,k,\ell\} = \{1,2,3,4\}$).  Thus $T_1$ passes through
  the same component as $T_i$ for exactly $2n/2 = n$ of the singular fibres.
  For the second statement, it suffices to recall from Corollary
  \ref{cor:ei-meet-vj} that $E_i$ contains the points $(P_j,P_k)$ of
  $D_3 \times D_3$ when $\#\{i,j,k\} = 3$, since that
  implies that $G_i$ passes through the points of $S_j \cap S_k$ and hence
  that $J_i$ meets the exceptional divisors above these.
\end{proof}

\begin{remark} The symmetry of the situation with respect to translation by
  the $2$-torsion divisors $T_j - T_1$ implies that any two $T_i$ and any two
  $J_i$ pass through the same component of exactly $n$ of the $\tA_1$ fibres,
  while $T_i$ and $J_j$ pass through the same component of $2n$ if $i \ne j$.
  (Again this follows from Proposition \ref{prop:bad-fibs-w} as well.)
\end{remark}

\begin{prop}\label{prop:hi-ti-meet-1}
  We have $J_i \cdot T_j = 0$ for $i \ne j$ and
  $J_i \cdot T_i = \lfloor \frac {n-1}2 \rfloor$.
\end{prop}

\begin{proof}
  For $i \ne j$ the curves $\widetilde S_j$ and $\widetilde G_i$ are disjoint,
  while $\lambda$ fixes $S_j$ as a set: the proves the first equality.  For the
  second, we would like to use the push-pull formula for the map
  $q_\lambda: \widetilde W \to K$.  However, this requires that we blow up
  the isolated fixed points of $\lambda$ on $\widetilde W$: let the blowup be
  $\blowuptwice{W}$.
  These include the points on the fixed
  fibres above $O$ at which $S_i$ meets $G_i$
  (cf.{} Proposition \ref{prop:si-meet-gi}).  For $n$ odd there is one such
  fibre and for $n$ even there are $2$,
  so the intersection on $\blowuptwice{W}$
  is $1$ or $2$ less than on $\widetilde W$ and is therefore $n-1$ or $n-2$
  (Corollary \ref{prop:si-meet-gi}).

  Now applying the push-pull formula and observing that $S_j$ maps to
  $T_j$ with degree $2$, we see that
  $2T_j \cdot J_i = \blowuptwice{{S_j}} \cdot \blowuptwice{{G_i}}$, and
  the result follows.
\end{proof}

\begin{thm}\label{thm:mw-gen}
  The sections $J_i$ are of infinite order in the Mordell-Weil group of
  $\pi_K$, and further we have $[J_i] - [J_j] = [T_i] - [T_j]$.
\end{thm}

\begin{proof}
  In view of the second statement, it suffices to prove the first for $i = 0$.
  If it were not true, then the Picard class of $J_i$ would be in the
  subspace of $\Pic K \otimes \Q$ spanned by curves in reducible fibres and
  the $T_i$.  We will obtain a contradiction by solving
  for $[J_i]$ under the assumption that it is in this subspace of
  $\Pic K \otimes \Q$.  Note that a section~$S$ of $\pi_K$ satisfies
  $S(K_K + S) = -2$; since $K_K$ is $\lfloor \frac{n-3}{2} \rfloor$ times the
  fibre class, this means that $S^2 = -2 - \lfloor \frac{n-3}{2} \rfloor$.

  We first consider the case of $n$ odd.  Choose a basis consisting of the
  fibre class $F$, the $0$ section~$S_0$, the nonzero components
  $d_1, \dots, d_4$ of the $I_0^*$ fibre, and the nonzero components
  $a_1, \dots, a_{3n}$ of the $3n$ fibres of type $I_2$.  The class of
  $J_i$ has intersection~$1$ with all of these basis vectors except for the
  $d_i$, with which its intersection is $0$.  One easily computes that
  if $[J_i]$ is in the given subspace its class is
  $C = \frac{n+3}{2}F + S_0 - \sum_{i=1}^{3n} a_i/2$ and that
  $C^2 = -\frac{5}{2}-n \notin \Z$,  a contradiction.

  Now let $n$ be even.  In this case there is no $I_0^*$ fibre and our basis
  has $3n+2$ elements, each of which intersects the putative class $[J_i]$
  once.  Again one sees that $[J_i] = nF + S_0 - \sum_{i=1}^{3n} a_i/2$.
  One of the torsion sections has the class
  $(n-1)F + S_0 - \sum_{i=1}^{2n} a_i/2$;
  the intersection of these is $1-\frac{n}{2}$, which is negative for
  $n \ge 4$, a contradiction.  One also checks that $[J_i]^2 = 2-n$.
  This contradicts our calculation of the self-intersection of a section
  for $n \ne 4$ and in particular for the last remaining possibility $n = 2$.
  
  To show that $[J_i] - [J_j] = [T_i] - [T_j]$ in the Mordell-Weil group of
  the fibration, we return to the definitions.  Recall that $J_i$ was obtained
  by pulling back a divisor $R_i$ from $D_3 \times D_3$ to $C_n \times C_n$
  and mapping down to $W$ and then $K$.  Now, on $D_3 \times D_3$ we have
  the linear equivalence $[R_i] + [H_i] + [V_i] \sim [R_j] + [H_j] + [V_j]$,
  since both of these are the class $\O(1,1) - [D]$, where $\O(1)$ refers
  to the embedding of $D_3$ as a plane quartic (cf.~Proposition
  \ref{prop:ei-unique}) and $D$ is the diagonal
  on $D_3 \times D_3$.  We pull this back to $C_n \times C_n$ and then map
  down to $K$.  The inverse image of $R_i$ on $C_n \times C_n$ maps with
  degree $2n$ to $G_i \subset W$, while those of $H_i$ and $V_i$ map with
  degree $n$ to $S_i \subset W$.  Thus, by pushing forward to $W$ and back to
  $\tW$,
  we obtain the relation~$2n ([J_i] + [T_i]) = 2n ([J_j] + [T_j])$ up
  to a linear combination of the exceptional divisors of $\tW$ above
  the singular points of $W$.  These divisors are vertical for the fibration,
  so this relation holds in the Mordell-Weil group of $\tW$ (as elliptic
  surface over $E$) up to torsion.

  So $[J_i] + [T_i] = [J_j] + [T_j]$ up to torsion,
  and the two sides agree in the
  component groups of all $\tA_1$ fibres.  Thus, in the Mordell-Weil group,
  the class of $[J_i] - [J_j] - ([T_i] - [T_j])$ is a torsion section~$T$
  passing through the zero component of every reducible fibre.
  The height of a torsion section is $0$, but if we
  apply \cite[Lemma 1.18]{cz} to compute the height pairing of such a section
  with itself, there are no correction terms and we obtain
  $0 = -(T_0 - T)^2 = -T_0^2 - T^2 + 2 T_0 \cdot T$.  For $T_0 \ne T$ we
  have $T_0 \cdot T = 0$; but this is not possible since all sections
  have the same self-intersection and this cannot be $0$ for any fibration
  that is not a product.
\end{proof}

We now specialize to $n = 3$ so as to study the situation of a K3 surface
with Picard lattice $L_1$.  In this case we have a complete description of
the Mordell-Weil group, at least for a generic choice of initial data.
\begin{thm}\label{thm:enough-gen}
  Suppose that the Picard number of $K$ is $16$.  Then the Mordell-Weil group
  of $K$ is isomorphic to $\Z \oplus (\Z/2\Z)^2$ and any one class
  $J_i - T_j$ generates the group modulo torsion.
\end{thm}

\begin{proof}
  The given description of the singular fibres together with the assumption on the
  Picard number shows (by means of the Shioda-Tate formula) that the
  Mordell-Weil rank is
  $1$.  Further, the discriminant of the lattice spanned by the vertical
  curves and the given sections is $192$, so if its index in the full Picard
  group is $d$ then $192/d^2 \in \Z$, so that $d$ is a power of $2$.  Thus
  the only question is whether there is a nonempty subset of
  $\{T_1, T_2, J_0\}$ whose sum
  can be divided by $2$.  To see that this is not the case, note simply that
  the component groups of the singular fibres are all of exponent $2$, so that
  no section can be divided by $2$ if it passes through a nonzero component on
  any singular fibre.  But, as already noted, the $T_j$ pass through six of the
  $b_i$ for $i \ne j$ and $J_0$ through all nine; it follows that $T_1 + T_2$
  also passes through six, while  $J_0 + T_1 + T_2$ and the $J_0 + T_j$
  pass through three.
\end{proof}

\subsection{Proof of the main theorem and two variants}\label{subsec:proof-main}
At this point we have constructed a new map
$\kappa_\mo: \M'_3 \to \M'_{K,1}$ (recall that
these are respectively the moduli space of degree-$3$ covers
$C_0 \to \Mbar_{0,4}$ with a choice of ramification point and of K3 surfaces
with a suitable elliptic fibration and a labelling of the $2$-torsion sections;
we take the torsion sections in the order $T_1, T_2, T_3$).  To reiterate,
starting from a point of $\M'_3$ we successively construct
$E, E_3, D_3, C_3, W, \tW, K$, of which the last is a K3 surface with the
desired additional data (the generator $J_0$ matches the section of
infinite order denoted by $G$ in Figure \ref{fibrationPhi}).
As pointed out at the beginning of this section,
the choice of $D_3$ is not determined by the initial data and the field of
moduli of $D_3$ may not be the same as that of $E$ and the four points on it,
but since $K$ is determined up to isomorphism by the initial data its field
of moduli is the same.

We are now in a position to prove the main theorem (Theorem \ref{thm:main-cover}) of the paper.  In view
of Remark \ref{rem:m3-mc}, this is equivalent to proving
Theorem \ref{thm:main-moduli-intro}.

\begin{thm}\label{thm:main}
The map $\kappa_\mo$ coincides with the map $\kappa$
defined in Definition \ref{def:kappa}.
\end{thm}

\begin{proof}
  We showed in Proposition
\ref{prop:nu-birational} that a fibration of this type is determined by the
locations of its singular fibres, so it suffices to show that these are the same
for the two constructions.
Recall that the definition of $\kappa$
started from a triple cover
$\phi: C_0 \to \Mbar_{0,4}$ and associated to it an elliptic surface with
a $\tD_4$ fibre at the third point in a chosen ramified fibre and with $\tA_1$
fibres above the boundary points of $\Mbar_{0,4}$.  To see that this
construction does the same, we start by locating the $\tD_4$ fibre.  It is
above the chosen origin $q''$ of $E_3$ above $q'$,
which in turn lies above the point of $\Mbar_{0,4}$ under the chosen
ramified fibre.  In the associated map $E_3/\pm \to E/\pm$, the fibre above
the image of the origin of $E$ consists of $q'', q''+T, q''-T$, where $T$ is
a point of order $3$.  Thus the map is ramified there and the origin of $E_3$
is the third point of the fibre, which is $q''$.

The $\tA_1$ fibres of $V$ lie above the points $p_i + p_j$ of $E$,
as we saw in Proposition \ref{prop:w-ell-surf}.
Recall that $p_0 = O$ and $p_i = B_j + B_k$
where $\{i,j,k\} = \{1,2,3\}$.  So the $\tA_1$ fibres lie above
the points $B_1 + B_2 + B_3 \pm B_i$ in $E$, and therefore above their
inverse images in $E_3$.  These are paired by the negation map to which
$\lambda$ descends on $E_3$, which must therefore fix a point above
$B_1 + B_2 + B_3$.  Composing with the translation by $-(B_1+B_2+B_3)$ on
$E$, we find that the $\tA_1$ fibres are above the points $\pm B_i$ and
that the $\tD_4$ fibre is above $O$.  Then the $\tD_4$ fibre in this
structure is at the origin of $E_3$,
which maps to the unramified point of $E_3/\pm$ in the fibre above the image
of $q'$ in $E/\pm$.  So the locations of the singular fibres match in the two
constructions.
\end{proof}

Now we briefly discuss the case $n = 4$ to indicate why it does not
receive more detailed attention in this paper.  As in the case
$n = 3$, we express a $4$-parameter family of K3 surfaces as quotients of
squares of curves, but this essentially
follows from the results of \cite{paranjape}
and does not require our construction.
\begin{prop}\label{prop:n-4-already-known}
  The K3 surfaces of Picard number $16$
  obtained from a point of $\M'_4$ admit
  maps of degree $4$ to and from double covers of $\P^2$ branched along
  six lines.
\end{prop}

\begin{proof}
  Generically, the Picard lattice is generated by the following curves:
  the components of the $12$ fibres of
  type $\tA_1$; the zero section and the $2$-torsion sections;
  two disjoint sections $C_1, C_2$ whose classes together with the torsion
  generate a subgroup of the Mordell-Weil group of index $2$,
  each passing through the
  nonzero components of all reducible fibres and meeting the zero
  section twice; and a section whose class in the Mordell-Weil group is
  $(C_1+C_2)/2$.  (The sections $C_1, C_2$ are obtained by pulling back conics
  as in Proposition \ref{prop:mw-ep},
  like the section~$G$ in the case $n = 3$.
  We know that the section of class $C_1+C_2$ is divisible
  by $2$: some linear combination of those already mentioned must be, since
  otherwise the discriminant group would need $8$ generators, which is not
  possible for the Picard lattice of a K3 surface of rank $16$.  On the
  other hand, no other class in the Mordell-Weil group mod $2$ can be
  divisible by $2$, because this is the only one other than $0$
  that passes through the zero components of all the reducible fibres.)
  
  By enumerating the genus of the Picard lattice, we find that
  this surface also admits a fibration with eight fibres of type
  $\tA_1$ and two of type $\tA_3$ with full level-$2$ structure, such
  that there is a $2$-torsion section passing through the identity
  component of both $\tA_3$ fibres and no $\tA_1$ fibre.  The quotient
  by translation by this section is therefore a fibration with two
  $\tA_7$ fibres and a $2$-torsion section disjoint from the identity
  component on both reducible fibres.  By \cite{shimada-table} this is
  the full torsion subgroup, so we have determined the lattice.

  In turn, this quotient surface admits a fibration with reducible fibres of
  type $\tA_3, \tA_3, \tD_8$ and a $2$-torsion section.  Since this section
  has height $0$ (or by direct computation)
  it must pass through the far component of each of the
  reducible fibres, which means that the quotient of the corresponding
  isogeny has a $\tD_6$ fibre and eight of type $\tA_1$ (the other
  six being the images of singular irreducible fibres).  Again by
  \cite{shimada-table} the quotient must have full level-$2$ structure
  and no further torsion;
  this determines the lattice, and it is in the same genus as the
  essential lattice of the standard fibration on a double cover of $\P^2$
  branched along six lines.  (The fibration we have shown to exist is
  of class 2.7 in \cite[Table~2]{kloosterman}.)

  Thus we have shown the existence of a pair of $2$-isogenies---more
  precisely, quotients by van Geemen-Sarti involutions---whose
  composition is a map from a surface obtained by our construction to
  a suitable double cover of $\P^2$.  Considering the dual isogenies, we
  also obtain a map in the other direction.
\end{proof}

To close this section, we remark that a general statement can be made
using Proposition \ref{prop:birat-moduli}.  In terms of Definition
\ref{def:higher-moduli}, we have shown that the elliptic surfaces
parametrized by $\M'_{\E,3,(\{2,1\})}$ are motive-finite.
It is not necessary to introduce $\M'_{\E,n,R}$ to discuss the case $n = 3$,
because the ramification data associated to the cover $E_3/\pm \to E/\pm$
are generic.  However, this is necessary for larger $n$; applying our
construction in general, we conclude:

\begin{thm}\label{thm:motive-finite-ell-surfs}
  Let $R$ be the ramification data consisting of two copies of $(2,\dots,2)$ and two of
  $(2,\dots,2,1,1)$ for $n$ even, or four copies of
  $(2,\dots,2,1)$ for $n$ odd.  Then the elliptic surfaces parametrized by
  $\M'_{\E,n,R}$ are covered by the square of a curve, and in particular
  have Kimura-finite motive.
\end{thm}

\section{Hodge theory}\label{sec:hodge}
Following the discussion in \cite[Sections 1--2]{paranjape}, we
reinterpret the construction of Section \ref{sec:constr-curves} in
terms of Hodge theory.
In the case $n = 3$ that we are
considering, the construction can be summarized as follows:
\begin{enumerate}
\item The cover $C_3 \to D_3$ has a generalized Prym variety $P_4$ which is an
  abelian variety of dimension $7 - 3 = 4$.
\item The automorphism of order $3$ on $C_3$ acts on 
  $H^1(P_4,\Q)$ to make it a $4$-dimensional vector space over $\Q(\sqrt{-3})$.
\item The symplectic form on $H^1$ gives an hermitian form with respect to the
  $\Q(\sqrt{-3})$-structure whose invariants can be computed.
\item The decomposition of the $\Q(\sqrt{-3})$-vector space
  $H^1(P_4,\Q(\sqrt{-3}))$
  into components of dimensions $2, 2$ gives a decomposition of
  $\Lambda^2(H^1(P_4,\Q))$ into components of dimensions $1, 4, 1$, which we
  interpret as the transcendental lattice of a K3 surface.
\end{enumerate}

\begin{remark} The only cases in which we obtain a generalized Prym variety
  of dimension~$4$, either for Paranjape's original construction or in our
  variant, are those with $n = 2, 3, 4, 6$.  For $n = 4$ our construction
  produces a curve of genus~$9$ which is an unramified cover of degree $4$
  of a curve of genus~$3$; the intermediate cover of degree $2$ has genus~$5$.
  Again, the moduli space is a finite cover of that of the base curve of
  genus~$1$ with $4$ points, so it has dimension~$4$.

  In the case $n = 6$, we obtain a curve $C_6$ of genus~$13$ which is an
  unramified cover of degree $6$ of a curve of genus~$3$; the intermediate
  covers have genus~$5, 7$.  Thus $\Jac(C_6)$ has abelian variety factors of
  dimension~$3, 5-3, 7-3$, and the quotient by the sum of these has
  dimension~$4$.  These considerations are related to those of
  Remark \ref{rem:various-n}.
\end{remark}

\begin{defn}\label{def:lattices}
  As in Definition \ref{def:mn},
  the {\em hyperbolic lattice} $\H$ is the lattice of rank $2$ with
  a basis $x, y$ such that $x^2 = y^2 = 0, xy = 1$.  Given a lattice $L$,
  let $L(n)$ be the lattice of the same rank whose Gram matrix is $n$ times
  that of $L$.
\end{defn}

In this section we will prove: 
\begin{prop}\label{prop:hodge-theory}
  The abelian variety $A$
  associated to a K3 surface $K$ with transcendental lattice isometric to
  $\Lambda_3 = \langle -2 \rangle + \langle -2n \rangle + \H + \H$
  by the Kuga-Satake-Deligne correspondence is a power of
  the Prym variety of the double cover $C_3 \to D_3$ up to isogeny.  Further,
  there is a correspondence between the Shimura varieties corresponding to
  $SO(\langle -2 \rangle + \langle -6 \rangle + \H+\H)$
  and the unitary group of $\H+\H$ for $\Q(\sqrt{-3})/\Q$
  that parametrize the appropriate Hodge structures.
\end{prop}

The discussion in \cite[Section 2]{paranjape} up to the end of the proof of
\cite[Lemma 2]{paranjape} applies to our situation almost word for word, given the differences we previously pointed out in Remark \ref{rem:paranjape}.   
Since our automorphism is of order $3$ rather than $4$, the action on the
spaces $V$ and $W$ is by $\zeta_3, \zeta_3^2$ rather than by $\pm i$.
In addition, the discriminant of the Hermitian structure $H$ in our case is
a power of $3$, but again this is enough to ensure that it is a norm from
$\Q(\sqrt{-3})$ to $\Q$.  The result of Landherr \cite{lh} used in
\cite[Section 2]{paranjape} does not depend on
the particular quadratic extension that arises.

One might speculate that for other quadratic fields $\Q(\sqrt{-n})$ there
would be an analogous construction with the automorphism replaced by a
correspondence, perhaps based on Paranjape's idea or our variant.
One would expect that
the discriminant of the Hermitian structure would be a power of $n$ and
therefore a norm in the extension $\Q(\sqrt{-n})/\Q$ and that the argument
would still apply.  (See \cite{lombardo}, \cite{lps} for interesting
discussions of the Hodge structure for the associated abelian Kuga-Satake
variety and for computations similar to the ones below.)

In any case, even in the absence of a geometric construction beyond the
simplest cases $n = 1, 3$, let
$F$ be a $\Q(\sqrt{-n})$-vector space of dimension~$4$ equipped with a basis
$e_1, f_1, e_2, f_2$ and an hermitian form with $H(e_i,f_i) = H(f_i,e_i) = 1$
and all other pairings equal to $0$.  We carry through this calculation in
detail in order to obtain a precise statement at the end.
We consider $\Lambda^2(F)$ with the basis
\begin{align*}a_1 = e_1 \wedge f_1, & \quad a_2 = e_1 \wedge e_2, \quad a_3 = e_1 \wedge f_2\\
  b_1 = e_2 \wedge f_2, & \quad b_2 = f_2 \wedge f_1, \quad b_3 = f_1 \wedge e_2\\
\end{align*}
(chosen such that $a_i \wedge b_i$ is independent of $i$).
We define $H$ on $\Lambda^2 F$ by defining
$H(p \wedge q, r \wedge s) = H(p,s)H(q,r) - H(p,r)H(q,s)$, which is linear
in $p,q$ and antilinear in $r,s$.  It also satisfies
$H(p \wedge q, r \wedge s) = -H(q \wedge p, r \wedge s) = -H(p \wedge q, s \wedge r)$;
thus it gives
a well-defined form on $\Lambda^2 F$, which is easily seen to be
Hermitian.  For example,
\begin{align*}
  H(a_1,a_1) &= H(e_1 \wedge f_1,e_1 \wedge f_1)\\
  &= H(e_1,f_1)H(f_1,e_1) - H(e_1,e_1)H(f_1,f_1)\\
  &= 1 \cdot 1 - 0 \cdot 0 = 1.\\
\end{align*}
Similar calculations show that $H(b_1,b_1) = H(a_2,b_2) = -H(a_3,b_3) = 1$,
while all other products are $0$.  Likewise, we may extend $H$ to a symmetric
form on the $1$-dimensional space
$\Lambda^4 F = \langle g \rangle$, where $g = a_i \wedge b_i$.  Doing so,
we find that $H(g,g) = H(a_1,b_1)H(b_1,a_1)-H(a_1,a_1)H(b_1,b_1) = 1$.

Still following Paranjape, we now define a $\Q(\sqrt{-n})$-antilinear
automorphism $t$ of $\Lambda^2 F$ by the condition that
$H(u_1,u_2) g = -u_1 \wedge tu_2$ for all $u_1, u_2 \in \Lambda^2 F$.
For example, the only nonzero evaluation of $H$ involving $a_1$ is
$H(a_1,a_1) = 1$, while the only nonzero wedge product is $a_1 \wedge b_1 = g$.
Thus we must have $t(a_1) = -b_1$, and similarly $t(b_1) = -a_1$, $t(a_2)=-a_2$, and $t(b_2)=-b_2$ while
$t(a_3) = a_3$ and $t(b_3) = b_3$.
The invariant subspace for this involution is generated by
$$a_1 - b_1,\> (a_1 + b_1)\sqrt{-n},\> a_2\sqrt{-n},\> b_2\sqrt{-n},\> a_3,\> b_3.$$
By an easy calculation, the matrix $M_H$ of $H$ on this basis is
$$M_H=\begin{pmatrix}
  -2&0&0&0&0&0\cr
  0&-2n&0&0&0&0\cr
  0&0&0&-2n&0&0\cr
  0&0&-2n&0&0&0\cr
  0&0&0&0&0&2\cr
  0&0&0&0&2&0\cr
\end{pmatrix}.$$

\begin{defn}\label{def:mn}
  Let $M_n$ be the lattice 
  $\langle -2 \rangle + \langle -2n \rangle + \H + \H$.
\end{defn}
Note that, for all $k \ne 0$, the quadratic space $\H \otimes \Q$
is isomorphic to the quadratic space over $\Q$ with Gram matrix
$\begin{pmatrix}0&k\\k&0\end{pmatrix}$,
by multiplying one of the generators by $k$.
So, by a suitable change of basis, the matrix $M_H$ can be rewritten to
coincide with the Gram matrix of the standard basis of the lattice
$M_n$. 

In this section we have shown the following statement.
\begin{prop}\label{prop:same-qf}
   Over $\Q$, the quadratic form $H$ is isomorphic to
the quadratic form associated to $M_n$.
\end{prop}

Combining Proposition \ref{prop:same-qf}
with the results of the previous section, we obtain a proof of Proposition \ref{prop:hodge-theory}.

The following result, due to Lombardo, gives additional information:
\begin{prop}\cite[Corollary 6.3 and Theorem 6.4]{lombardo}\label{thm:lombardo}
Notation as in Proposition \ref{prop:hodge-theory}.  
\begin{enumerate}
\item There is an isogeny  $A\sim P_4^4$;
\item $\Q(\sqrt{-3})\subseteq \End_{\Q}(P_4)$;
\item $P_4$ admits a polarization~$H$ such that $(P_4, H, \Q(\sqrt{-3}))$ is a polarized abelian variety of Weil type with $\disc(P_4, H, \Q(\sqrt{-3}))=1$.
\end{enumerate}
\end{prop}
\begin{cor}\label{cor:incl-hodge}
	There is an inclusion of Hodge structures
	\begin{align*}
	T(K)\hookrightarrow H^2(P_4^4\times P_4^4, \Q),
	\end{align*}
        where $K$ is as in Proposition \ref{prop:hodge-theory}.
\end{cor}
This corollary follows directly from Proposition \ref{prop:hodge-theory}.

\begin{thm}\label{thm:hodge-conj-l2-surfaces}
	The Hodge conjecture is true for $K\times K$.
\end{thm}
\begin{proof}
  The proof is the same as for  \cite[Theorem 2]{schlickewei},
  mutatis mutandis as in Proposition \ref{thm:lombardo} and Corollary \ref{cor:incl-hodge} above.
\end{proof}

\begin{remark}\label{rem:isog-trans-u-u-m2-m6}
  To close this section,
we remark that a K3 surface $S$ with Picard lattice isometric to $L_1$
is isogenous to a K3 surface $S'$ with transcendental lattice $\Q$-isometric to
$\langle -2 \rangle + \langle -6 \rangle + \H + \H$.  The signature is correct:
the Hodge index theorem shows that the signature of the transcendental lattice
of a K3 surface of Picard rank $\rho$ is $(2,20-\rho)$.  Indeed, we will show
in Corollary \ref{cor:describe-2-isogeny} that $S$ is $2$-isogenous to a K3
surface whose Picard lattice is $\Q$-isometric to $\H + D_8 + A_5 + A_1$,
and hence to $\H + E_8 + A_5 + A_1$.  We embed this lattice primitively into
$\H^3 + E_8^2$ by matching copies of $\H$ and $E_8$ and embedding $A_5 + A_1$
into $E_8$ in one of the obvious ways; it is routine to calculate the Gram
matrix of the orthogonal complement of $A_5 + A_1$ in $E_8$ and show that
it is equivalent to the diagonal matrix with diagonal $(-2,-6)$.
\end{remark}

\section{Scope of the construction}\label{sec:scope}
We have just shown that
$U(\H + \H,\Q(\sqrt{-n}))$ is a double cover of
$SO(\langle -2 \rangle + \langle -2n \rangle + \H + \H)$; in the case $n = 1$
this was already remarked by Paranjape.
More generally, let $P$ be a Hodge structure of type $(4,4)$ with an action
of an order $\O = \O_K$ in an imaginary quadratic field.  We may then view $P$
as an Hermitian module $P'$ of rank $4$, which is determined by its
signature (necessarily $(2,2)$) and its discriminant up to norms from $K$.

If the discriminant is a square, then as above we have the exceptional
isogeny $SU(P',\O/\Z)/\pm 1 \to SO(M_n)$, where
$M_n \cong \langle -2 \rangle + \langle -2n \rangle + \H + \H $.  If it is
not, then as in \cite[2.13]{garrett} we do not have an isogeny defined over
$\Q$ to the special orthogonal group of any lattice; the analogue of
Paranjape's map $t$ still has all eigenvalues equal up to sign and square
equal to the discriminant, but that now means that the eigenspaces are not
defined over $\Q$.

Suppose that there is a curve $C$ for which the Hodge structure on
$H^1(C)$ admits $P$ as a quotient.  We might then expect that $C^2$ admits
a correspondence to a K3 surface whose transcendental lattice is
isometric to $L$ after tensoring with $\Q$.  Thus, in order to determine
the limitations of this construction, we would like to know which lattices
of rank $6$ arise from the exceptional automorphism in this way.

Our starting point is the following observation.  Let $S$ be a
K3 surface.  Define the projective variety $Q(S) = Q(T(S)) \subset \P^n(\Q)$,
where $n = 21-\rho$, by the vanishing of the discriminant quadratic form.
The relation of the transcendental lattices of isogenous K3 surfaces is
well understood.  Before giving the result, 
we introduce an important lattice.

\begin{defn}\label{def:k3-lattice}
  Let $V = \H^3 + E_8^2$, where
  $E_8$ is the exceptional root lattice with negative sign.  Sometimes
  $V$ is known as the {\em K3 lattice}.
\end{defn}

\begin{remark}\label{rem:lat-standard-facts}
  It is well-known
  that $H^2$ of a K3 surface, with the usual intersection form on cohomology
  given by cup product, is isometric to $V$: see, for example,
  \cite[Section 1]{morr}.
\end{remark}
  
\begin{thm}\label{thm:trans-isogeny} (\cite[Section 1.8]{bsv})
  Let $\phi: S \to S'$ be a rational map of degree $n$ of K3 surfaces.
  Then $T(S')$ is isometric to a sublattice of $T(S)(n)$ for which
  the quotient has exponent dividing $n$.
\end{thm}

\begin{cor}\label{cor:corr-isog}
  Let $S, S'$ be isogenous K3 surfaces (as always, in the sense of
  Definition \ref{def:isogeny}).
  Then $Q_S \cong Q_{S'}$ as varieties over $\Q$.
\end{cor}

\begin{proof}
  This follows immediately from Theorem \ref{thm:trans-isogeny}
  by induction on $k$, because rescaling does not
  affect the quadric at all, while passing to a sublattice amounts to a linear
  change of coordinates.
\end{proof}

We now note:

\begin{prop}\label{prop:embed-q}
  Let $\mathcal L$ be the family of K3 surfaces with a given
  marked Picard lattice $L$ and transcendental lattice $L^\perp$, and let
  $S \in {\mathcal L}$.  Then $Q(S)$ is isomorphic to a linear section of
  $Q(L^\perp)$.
\end{prop}

\begin{proof}
  Indeed, the hypothesis implies that $L$ is embedded in $\Pic S$,
  whence $T(S)$ is primitively embedded in $L^\perp$.  Thus $Q(S)$ is obtained
  from $Q(L^\perp)$ by restricting to the linear subspace on which elements of
  $T(S)^\perp$ vanish.
\end{proof}

\begin{cor}\label{cor:embed-isogeny}
  The same conclusion holds if $S$ is isogenous to a member of $\mathcal L$.
\end{cor}

\begin{proof}
  This follows immediately from the theorem and proposition just above.
\end{proof}

\subsection{Examples and counterexamples}
In this section we give some examples to indicate the limits of our
construction.  First we describe some families of K3 surfaces of Picard number
$16$ for which the discriminant of the Picard lattice is $-12$ times a square
but to which our construction does not apply.  Following that, we show that
no finite union of families of K3 surfaces of Picard number $16$
includes all K3 surfaces of Picard number $17$ or $18$ in its closure.
It follows that 
no finite collection of constructions for K3 surfaces of rank $16$ suffices
to prove that every K3 surface of Picard number $17$ or $18$ is covered
by the square of a curve.

\begin{example}\label{ex:no-isog}
  We consider the example $n = 3$, the primary concern of this
  paper.  One readily computes that the Hasse-Minkowski invariant
  \cite[Chapter 4.1]{cassels}
  of $M_3$ is $1$ at all finite primes.  
  Multiplying the form by a prime $p \equiv 0, 1 \bmod 3$ does not change the
  invariants, while multiplying by $p \equiv 2 \bmod 3$ changes the
  invariants at $3$ and $p$.

  Let $L'$ be a lattice of discriminant $3$ times a square, signature
  $(2,4)$, and nontrivial Hasse-Minkowski invariant at a prime
  $p \equiv 1 \bmod 3$.  It follows from Corollary \ref{cor:corr-isog}
  that a K3 surface with transcendental lattice isometric to
  $L'$ is not isogenous to one
  with transcendental lattice isometric to
  $M_3$.  In particular, such a K3 surface cannot
  be shown by our construction to be covered by a product of curves.
  
  For example, we may take $L'$ to have an orthogonal basis with elements that
  square to $-1, -1, -2, -6, 7, 7$, so that the Hasse-Minkowski invariant
  is nontrivial at $2$ and $7$.  Proving that a K3 surface with
  transcendental
  lattice commensurable with $L'$ is motive-finite would seem to require a
  significant new idea.
\end{example}

We now study K3 surfaces of Picard number $17$ and $18$.
The following result is due to Nikulin:
\begin{thm}\cite[Theorem 14.3.17]{huybrechts}\label{kum-lattice}
There is a lattice $K$ of rank $16$ and discriminant $64$ such that
a K3 surface $S$ is a Kummer surface if and only if there is a primitive
embedding of $K$ into $\Pic S$.
\end{thm}

Nikulin describes $K$ as follows.  Let $G = (\Z/2\Z)^4$, and start
with the lattice $A_1^G$ (that is, the lattice generated by $16$ vectors indexed
by $G$ of norm $-2$ of which any two distinct ones are orthogonal).
Adjoin the vectors of the form $\sum_{i \in C} v_i/2$, where $C$ is a coset of
a subgroup of $G$ of index $2$.

It is easy to see that $K^\perp$ in $V$ is isometric to
$\H(2)^3$.
It follows that $Q(K)$ is $\Q$-isomorphic to the quadric in $\P^5$ defined by
$x_0 x_1 + x_2 x_3 + x_4 x_5 = 0$, since $\H(2) \cong \H$, as pointed out
after Definition \ref{def:mn}.

\begin{cor}\label{cor:not-kummer}
  Let $S$ be a K3 surface of Picard number $17$ (respectively $18$).  If there
  are no rational lines (resp.~rational points) on $Q(S)$, then $S$ is not
  isogenous to a Kummer surface.
\end{cor}

\begin{proof}
  Both statements follow immediately from the existence of rational planes on
  $Q(K)$, in light of Proposition~\ref{prop:embed-q}.
\end{proof}

\begin{cor}\label{cor:not-dc}
  Let $S$ be a K3 surface of Picard number $17$ (respectively $18$), and let
  $p$ be a prime congruent to $1 \bmod 4$.  If there are no $\Q_p$-rational
  lines (resp.~$\Q_p$-rational points) on $Q(S)$, then $S$ is not isogenous to a
  double cover of $\P^2$ branched along six lines.
\end{cor}

\begin{proof}
  Paranjape shows \cite[Lemma 1]{paranjape} that the generic
  transcendental lattice of a double cover of $\P^2$ branched along six lines
  is isometric to $\langle -2 \rangle + \langle -2 \rangle + \H + \H$, which
  is isometric over $\Q(i)$ and hence over $\Q_p$ to
  $\H(2) + \H + \H$.  In turn this is isometric to $\H \oplus \H \oplus \H$,
  as noted below Definition \ref{def:mn}.
\end{proof}

We recall that every K3 surface of Picard number $19$ or $20$ is
isogenous to a Kummer surface \cite[Corollary 6.4 (i)]{morr}.
For completeness, we include an example to illustrate the well-known
fact that this is not true in Picard number $18$.

\begin{example}\label{ex:rank-18}
  We describe an example of a family of K3 surfaces of Picard number $18$
  that are not isogenous to Kummer surfaces.
  Let $T$ be the lattice
  with Gram matrix
  $$\begin{pmatrix}
    -2&-1&0&-1\cr
    -1&2&1&-1\cr
    0&1&-2&1\cr
    -1&-1&1&2\cr
  \end{pmatrix}$$
  of determinant $36$.  Since this is an even lattice of signature $(2,2)$
  it admits a primitive embedding into $V$ (Definition \ref{def:k3-lattice})
  and accordingly it is the transcendental lattice of a very general member of
  a $2$-parameter family of K3 surfaces up to isometry.
  However, $Q(T)$ has no points over $\Q_2$ or $\Q_3$, so it cannot be a
  $3$-plane section of $Q(K)$.

  It appears that the construction of \cite{paranjape} can be used to show that
  K3s of this family are motive-finite.  Indeed, let us consider the lattice
  $A_1 + A_1 + \H + \H$ which is isometric to the transcendental
  lattice of a double cover of $\P^2$ branched along six very general lines
  after tensoring with $\Q$ \cite[Lemma 1]{paranjape}.  Inside this lattice,
  consider the orthogonal complement $T'$ of the subspace spanned by the
  vectors $$(0,0,1,-2,-1,1),(0,0,1,-1,1,-2)$$
  (this was a random choice of two orthogonal
  vectors of norm $6$).  By direct calculation this lattice has
  norm form $\Z$-equivalent to $-2x^2 -2y^2 + 6z^2 + 6w^2 = 0$.  Adjoining
  $(1/2,1/2,1/2,1/2)$ to this lattice produces a lattice isometric to $T$
  (to check this, use LLLGram in Magma and then change basis by a suitable
  combination of signs and a permutation).  Thus a K3 surface with
  transcendental lattice $T'$ can be written as a double cover of $\P^2$
  branched on six lines as above,
  and such a surface would be expected to be $2$-isogenous
  (Definition \ref{def:isogeny}) to a surface with
  transcendental lattice~$T$.

  It is not difficult to construct such K3 surfaces.  To do so, we observe
  that $T^\perp$ (relative to a primitive embedding $T \hookrightarrow V$)
  is isometric to $E_6 + E_6 + D_4$.  For any $c \ne 0$,
  the elliptic surface $y^2 = x^3 + sx + (s+cs^2)$ has an $IV^*$ fibre at
  $s = \infty$ and a type $II$ fibre at $s = 0$, so a quadratic twist
  by $s(s-d)$, where $d \ne 0$, is a surface of the desired form, and this
  gives $2$ moduli of such surfaces as expected.
\end{example}

\begin{example}\label{ex:rank-17-not-kummer}
  In \cite[Section 9.4.4]{bcgp}, Boxer, Calegari, Gee and Pilloni
  present a construction of a certain family of K3 surfaces due to Nori.
  Let $L_0, \dots, L_4$ be five lines in $\P^2$ and let $C$
  be the conic through the five points $L_i \cap L_{i+1}$
  (indices read mod $5$).  Let $L$ be a
  sixth line tangent to $C$, and let $S$ be the double cover of $\P^2$
  branched along the $L_i$ and $L$.  They show that these K3 surfaces are not
  isogenous to Kummer surfaces for very general choices of
  $L_i, L$.
  
  We verify this claim in a different way from that of \cite{bcgp}.
  First we note that the Picard lattice is
  generated by the classes of the hyperplane, the nodes, the lines, and the
  components of the pullback of $C$, and that it is of rank $17$.
  We compute that the discriminant is
  $96$ and that the discriminant group is the same as for $\H + E_8 +
  A_2 + A_1^5$, so that the lattices are in fact isometric by
  \cite[Theorem 2.8]{morr}.  To describe $(\Pic S)^\perp$, we embed this
  into $V$.  We embed $\H + E_8$ into one copy of $\H + E_8$ in $V$.
  Then $A_2 + A_1^3$ can be embedded primitively into $E_8$ by
  embedding the Dynkin diagram: for example, we may take the three roots
  corresponding to the neighbours of the vertex of index $3$ and the two
  that are not adjacent to any of these.  The complement is then
  $A_1 + A_2(2)$.  Finally, the other $A_1$ components are embedded in $\H$,
  each with complement $\langle 2 \rangle$.
  So the transcendental lattice is isometric to
  $A_1 + A_2(2) + \langle 2 \rangle + \langle 2 \rangle$.  The Fano
  variety of lines on the corresponding quadric in its Pl\" ucker embedding
  in $\P^9$ is a smooth variety of dimension~$3$ and
  degree $8$ and has no points over $\Q_2$.  So by Corollary
  \ref{cor:not-kummer}, $S$ is not isogenous to a Kummer surface.

  On the other hand, as pointed out in \cite{bcgp}, the
  transcendental lattice can be rationally embedded in
  $\H + \langle -6 \rangle + \langle -2 \rangle + \langle -2 \rangle$ and
  hence in $\H+\H + \langle -6 \rangle + \langle -2 \rangle$.  Thus
  we have constructed
  the fake abelian surface associated to the
  generalized Prym variety with endomorphisms by an order in
  the division algebra
  $D = (-1,3)_\Q$ whose existence was suggested in \cite[9.4.4]{bcgp}.
\end{example}

  
More generally, the same ideas can be used to prove that no finite collection
of types of K3 surfaces corresponding to a rank-$16$ lattice suffices to
describe all K3 surfaces of rank $17$ or $18$.
\begin{thm}\label{thm:no-finite-set}
  Let $L = \{L_1,\dots,L_n\}$ be a finite set of lattices of rank $16$ and let
  $r \in \{17,18\}$.  Then
  there is a K3 surface $S_r$ of N\'eron-Severi rank $r$
  such that no element of $L$ can be rationally embedded into
  $\Pic S_r$.
\end{thm}

Note that the theorem does not require the $L_i$ to have signature
$(1,15)$, as they would for the Picard lattice of a projective surface;
the signature $(0,16)$ is also permitted.

\begin{proof}
  We may assume that the $L_i$ can be primitively embedded in $V$.
  Let the $Q_i$ be the quadric hypersurfaces in $\P^5(\Q)$ defined by the norm
  forms of the orthogonal complements of the $L_i$ in $V$.  Since the
  $Q_i$ are defined by indefinite quadratic forms, they have real points;
  also, smooth quadrics of dimension~$\ge 4$ over local fields always have
  rational points, so the $Q_i$ are everywhere locally solvable and hence have
  rational points.
  
  Recall that the two families of $k$-planes on a smooth quadric of
  dimension~$2k$ over a field $F$ of characteristic not equal to $2$
  are defined over $F(\sqrt {(-1)^k D})$, where $D$ is the determinant of the
  symmetric matrix associated to the quadric.  Let the fields of definition
  of the families of $2$-planes on the $Q_i$ be the $\Q(\sqrt{D_i})$.  Since
  the $Q_i$ have rational points, they have planes defined over the
  $\Q(\sqrt{D_i})$.

  In particular, every hyperplane section over this field has rational lines
  and every $\P^3$-section has rational points.  Let $p$ be an odd prime that
  splits in every $\Q(\sqrt{D_i})$, let $N_p$ be a lattice of signature $(2,2)$
  with norm form $x^2 - ny^2 + pz^2 - npw^2$, where $n$ is a quadratic
  nonresidue mod $p$, and let $T_p$ be an even
  sublattice of $N_p$ of full rank.  By \cite[Theorem 2.8]{morr}, we can
  embed $T_p$ primitively into $V$.  However, the quadric associated to $T_p$
  has no $\Q_p$-points and therefore none over any of the $\Q(\sqrt{D_i})$,
  because the completion of any of these at a prime over $p$ is isomorphic
  to $\Q_p$.  Thus a K3 surface of Picard number $18$ with transcendental
  lattice $T_p$ is not isogenous to any surface into whose Picard lattice
  any of the $L_i$ embeds.

  Similarly, if we take $T_p$ to be an even lattice of rank $5$
  admitting a rational embedding of $N_p$, then the corresponding quadric
  has no $\Q_p$-rational lines, because it has a hyperplane section with no
  $\Q_p$-points, and hence it is not a hyperplane section of any of the $Q_i$.
\end{proof}

\begin{cor}\label{cor:no-finite-set-isog}
  With notation as in the theorem, no element
  of $L$ can be embedded into $\Pic S'$ for any K3 surface $S'$ isogenous
  to $S_{r}$.
\end{cor}

\begin{proof} This follows by combining Theorem~\ref{thm:no-finite-set}
  with Corollary~\ref{cor:embed-isogeny}.
\end{proof}

\begin{thm}\label{thm:not-kum-not-dc-17}
  There exist K3 surfaces of Picard number $17$ that
  are covered by the square of a curve but that are not isogenous to
  a Kummer surface or a double cover of $\P^2$ branched along $6$ lines.
\end{thm}

\begin{proof}
  Fix a prime $p \equiv 17 \bmod 24$, and let $L_p$ be the lattice
  with Gram matrix
  $\langle -2 \rangle + \langle -6 \rangle + \H + \langle 4p \rangle$.  Again,
  this is the transcendental lattice of a K3 surface, which we will call
  $S_p$.  One easily checks
  that the local invariant at $p$ of the associated quadratic form is $-1$;
  this means that the quadric defined by this form has no lines over
  $\Q_p$.  It follows from Proposition~\ref{prop:embed-q} and
  Corollary~\ref{cor:not-kummer} that $S_p$ is not isogenous to a Kummer
  surface, and from Proposition~\ref{prop:embed-q} and
  Corollary~\ref{cor:not-dc} that $S_p$ is not isogenous to a double cover
  branched along $6$ lines.

  On the other hand, since $\H$ has a primitive vector of norm $4p$, the
  lattice $L_p$ can be embedded primitively in
  $\langle -2 \rangle + \langle -6 \rangle + \H + \H$, and thus $S_p$ belongs
  to the family of K3 surfaces whose general member we have shown to be
  covered by the square of a curve (Theorem~\ref{thm:main-cover}).  Since
  our map of moduli spaces is only a birational equivalence and not an
  isomorphism we cannot conclude that the $S_p$ are covered in this way.
  However, a birational equivalence cannot fail to be defined on infinitely
  many divisors, so the general $S_p$ can be covered for all but finitely
  many $p$.
\end{proof}

\begin{example}\label{ex:not-kum-not-dc-18}
  In order to apply this argument to show that a K3 surface of rank $18$ is
  isogenous neither to a Kummer surface nor to a double cover branched along
  $6$ lines, we need a transcendental lattice whose norm form is not solvable
  at a prime congruent to $1 \bmod 8$.  The smallest examples appear to have
  determinant $2^2 \cdot 17^2$; one possible Gram matrix is
  $$\begin{pmatrix}
    6&5&3&-3\cr
    5&6&-2&4\cr
    3&-2&-6&-2\cr
    -3&4&-2&6\cr
  \end{pmatrix}.$$
  Note that a K3 surface with this transcendental lattice would be expected
  from our results to be covered by the square of a curve of genus~$7$.
  We do not know this for certain because we have only constructed a
  birational equivalence of moduli spaces, not an isomorphism.  (The method
  used to prove the last theorem does not apply here, because an infinite set
  of codimension-$2$ loci could all be contained in a single divisor.)
  Indeed, this transcendental lattice tensored with $\Q$ is isometric to
  $\langle -2 \rangle + \langle -6 \rangle + \langle 17 \rangle + \langle 51 \rangle$,
  as one sees by computing Hasse-Minkowski invariants.  This embeds into
  $(\langle -2 \rangle + \langle -6 \rangle + \H + \H) \otimes \Q$ by a map that
  takes the first two basis vectors to the first two basis vectors and the
  last two to vectors in the two copies of $\H$ of the appropriate norm.
  \end{example}

\section{K3 surfaces}\label{sec:k3s}
We now study the family of K3 surfaces with Picard lattice isometric to
$L_1$ (Definition \ref{def:our-pic-short}) in order to
relate them to other interesting families.  
\begin{defn}\label{def:k-pic-l0}
	Let $K$ be a K3 surface with Picard lattice isometric to $L_1$.
\end{defn}

Our work in previous sections provides a proof of the following proposition:
\begin{prop}\label{prop:explicit-corr}
There is an explicitly computable algebraic correspondence between $C_3^2$ and $K$, where $C_3$ depends on $K$ and is as defined in Section \ref{sec:constr-curves}.  This correspondence induces a morphism of Hodge structures $H^2(C_{3}, \Z)\rightarrow H^2(K, \Z)$ when the base field of $K$ has characteristic zero.
\end{prop}
In this section, among other things, we prove  
the following statements:
\begin{prop}\label{prop:g-s-isogeny}
$K$ is isogenous to a K3 surface $K_{gs}$ in $\P^4$ with $15$
ordinary double points.  
\end{prop}
\begin{thm}\label{thm:hodge-conj-gs-surfaces}
The Hodge conjecture is true for fourfolds $K_{gs}\times K_{gs}$.
\end{thm}
The proofs appear at the end of the paper.
The notation~$K_{gs}$ is in honour of Garbagnati and Sarti, who
first discussed such surfaces in \cite{g-s}.  

Combining Proposition \ref{prop:explicit-corr} and 
Proposition \ref{prop:g-s-isogeny} gives an explicitly computable correspondence between $C_3^2$ and $K_{gs}$.   
\begin{remark}\label{rem:15-nodes-motive-finite}
  Laterveer proved \cite[Theorem 3.1]{laterveer}
  that K3 surfaces of degree $8$ with a faithful symplectic
  action of $(\Z/2\Z)^4$ are motive-finite, by showing that the quotients are
  double covers of $\P^2$ branched along six lines and invoking the results of
  \cite{paranjape}.  He suggested \cite[Remark 3.3]{laterveer} that it would
  be interesting to prove an
  analogous result for other families of K3 surfaces with a faithful symplectic
  action of $(\Z/2\Z)^4$.  In light of the result of Garbagnati and Sarti
  \cite[Theorem 8.6 (2)]{g-s}, there is such a family of K3 surfaces that
  admits a map to the $K_{gs}$, and Laterveer's arguments now show that these 
  are also motive-finite.

  We believe, as suggested by Laterveer, that this
  should be possible for all K3 surfaces with an action of $(\Z/2\Z)^4$,
  and in fact we have a construction for one more family.  We may return to
  this in future work.
\end{remark}
  
\begin{remark}\label{rem:lats-heights}
  Recall that in Definition \ref{def:our-pic-short}
  we defined $L_1$ as a certain lattice of rank $16$ and
  signature $(1,15)$ containing a sublattice $L_0$
  isometric to $\H + D_4 + A_1^9$ with quotient
  $\Z \oplus (\Z/2\Z)^2$.
  We suppose given a fibration as in Definition~\ref{def:mk1}
  (cf.~Figure \ref{fibrationPhi} and Theorem \ref{thm:image-is-k3}), with
  nine singular fibres of type $I_2$, one of type $I_0^*$, and Mordell-Weil group
  $\Z \oplus (\Z/2\Z)^2$.  Recall the notation: the section
  $0_{\phi_1}$ has been chosen as the origin, the $T_i$ are the torsion sections,
  and $G$ is a generator of the Mordell-Weil group modulo torsion that passes
  through the nonzero component of all the $\tA_1$ fibres.
  Let the $a_i, b_i$ be respectively the
  zero and nonzero components of the $I_2$ fibres, and the $d_i$ for
  $0 \le i \le 3$ the
  reduced components of the $I_0^*$, where $d_0$ is the zero component.
  By \cite[Lemma 1.18]{cz},
  the canonical height of the section~$G$ is $3/2$.
  Note that the intersection of $G$ with any of the $2$-torsion sections is
  $0$.  This can be checked from \cite[Lemma 1.18]{cz} or directly by writing
  down the classes of the $2$-torsion sections in $L_0 \otimes \Q$.
\end{remark}

\begin{prop}\label{prop:other-fib} On the surface $K$ there is an
  elliptic fibration $\phi_3$ with one singular fibre of type $I_2^*$, one of type
  $I_3$ 
  and six of type $I_2$.  Its Mordell-Weil group is $\Z/2\Z$, and the
  nonzero section passes through the zero component of the $I_3$
  the nonzero component of every $I_2$, and the reduced component of the $I_2^*$
  that is at distance $2$ from the zero component.
  There is one $I_1$ fibre.
\end{prop}

\begin{proof}
  We consider the sections $G,-G$ together with one of the $b$, say $b_j$.
  We saw in Proposition \ref{prop:mk1-bir-mk2} that $G \cdot -G = 1$.
  So any two of $G, -G, b_j$  have intersection~$1$ and
  they constitute an $I_3$ fibre $F$, and we obtain a genus-$1$ fibration.
  (This is not at all the same as constructing a fibration with
  $G \cup -G \cup 0_{\phi_1}$ as a fibre, as we did in Definition \ref{def:mk2}
  and Figure \ref{fibrationPi}.)
  The $T_j$ are disjoint from $\pm G$,
  as is seen by calculating the height pairing.
  So if $j$ is one of the two indices such that
  $T_j$ meets the chosen $b_i$, then $T_j$ is a section of $\phi_3$.
  Computing the essential lattice of the fibration \cite[Section 11.1]{ss}
  we find the given singular fibre types and torsion.  (Each of the
  $I_2$ fibres contains one of the zero components of an $I_2$ fibre of
  the fibration~$\phi_1$ that meets $T_2$ or $T_3$, shown in blue and red
  in Figure \ref{fibrationPhi}, while the components of the $I_2^*$ include
  the zero components of the two remaining $I_2$ fibres and $T_1$.)
  The rank of the Mordell-Weil
  group is $0$, by the Shioda-Tate formula \cite[Corollary 6.13]{ss}.
  As a consistency check we note that
  the discriminant of the lattice $\H + D_6 + A_2 + A_1^6$
  is $-4 \cdot 3 \cdot 2^6$, and
  dividing this by $4$ for the $2$-torsion point gives $-192$, which we know
  to be the discriminant of $L_1$.
  There cannot be a fibre of type $IV$ because there is $2$-torsion
  (recall that we are assuming that the characteristic of the ground field
  is not $2$), so the
  fibre whose $ADE$ type is $\tA_2$ must be an $I_3$.
  The reducible fibres contribute $23$ to the Euler characteristic.
  This implies the statement about singular irreducible fibres.
\end{proof}
\begin{figure}
  \caption{Elliptic fibration $\phi_3$}\label{fibrationPhi3}
  \centering
\begin{equation*}
\begin{tikzpicture}[xscale=0.7]

\draw[red] (0,0) -- (12,0);
\node[red] at (-0.5,0) {$T_3$};

\draw[blue] (0,6) -- (12,6);
\node[blue] at (-0.5,6) {$T_2$};

\draw[black] plot [smooth] coordinates {(0.5,-0.5) (1,3) (1.6,3.5) (2,3) (1.7,2.5) (1,3) (0.5,6.5)   };
\node[black] at (0.5,-1) {$I_1$};

\draw[blue] plot [smooth] coordinates {(3,6.5) (2.5,3) (3,1.5)};
\draw[black] plot [smooth] coordinates {(2.5,-0.5) (3,3) (2.5,4.5)};
\node[black] at (2.5,-1) {$I_2$};

\draw[blue] plot [smooth] coordinates {(4,6.5) (3.5,3) (4,1.5)};
\draw[black] plot [smooth] coordinates {(3.5,-0.5) (4,3) (3.5,4.5)};
\node[black] at (3.5,-1) {$I_2$};

\draw[blue] plot [smooth] coordinates {(5,6.5) (4.5,3) (5,1.5)};
\draw[black] plot [smooth] coordinates {(4.5,-0.5) (5,3) (4.5,4.5)};
\node[black] at (4.5,-1) {$I_2$};

\draw[black] plot [smooth] coordinates {(6,6.5) (5.5,3) (6,1.5)};
\draw[red] plot [smooth] coordinates {(5.5,-0.5) (6,3) (5.5,4.5)};
\node[black] at (5.5,-1) {$I_2$};

\draw[black] plot [smooth] coordinates {(7,6.5) (6.5,3) (7,1.5)};
\draw[red] plot [smooth] coordinates {(6.5,-0.5) (7,3) (6.5,4.5)};
\node[black] at (6.5,-1) {$I_2$};

\draw[black] plot [smooth] coordinates {(8,6.5) (7.5,3) (8,1.5)};
\draw[red] plot [smooth] coordinates {(7.5,-0.5) (8,3) (7.5,4.5)};
\node[black] at (7.5,-1) {$I_2$};

\draw[black] (9,6.5) -- (9,-0.5);
\draw[orange] (8.5,4) -- (10.25,2.75);
\draw[orange] (8.5,2) -- (10.25,3.25);
\node[black] at (8.75,3) {$b_1$};
\node[black] at (9,-1) {$I_3$};
\node[orange] at (9.75,3.5) {$G$};
\node[orange] at (9.75,2.5) {$-G$};

\draw[blue] (10.5,6.5) -- (11.5,3.5);
\node[blue] at (10.25,5.65) {$O_{I_2^*}$};
\draw[red] (10.5,-0.5) -- (11.5,2.5);
\node[black] at (11,-1) {$I_2^*$};

\draw[black] (11.25,5.75) -- (11.25,0.25); 
\draw[black] (11.35,5.75) -- (11.35,0.25);

\draw[green] (11,3.05) -- (13,3.05);
\draw[green] (11,2.95) -- (13,2.95);

\draw[green] (12.5,5) --  (12.5,1);
\draw[green] (12.4,5) -- (12.4,1);
\node[green] at (12.45,0.75) {$T_1$};

\draw[green] (12,4) -- (13,4);
\draw[green] (12,2) -- (13,2);

\end{tikzpicture}
\end{equation*}
\end{figure}
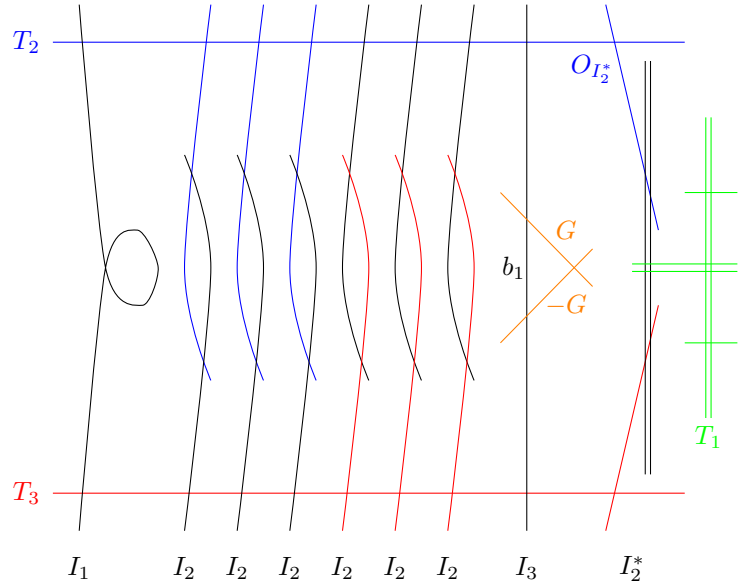

The elliptic fibration $\phi_3$ described in Proposition~\ref{prop:other-fib}
is depicted in Figure~\ref{fibrationPhi3} for the choice of $b_j = b_1$.
The colours red, blue, green and orange match those in Figures~\ref{fibrationPhi}~\ref{fibrationPi}. The horizontal red and blue curves are the sections $T_2,T_3$.  The other curves in the figure are the components of the singular fibres.  The blue curve labelled $O_{I_2^*}$ marks the identity component of the $I_2^*$ singular fibre when we choose $T_2$ to be the zero section of the fibration.

\begin{remark} Using Nishiyama's method \cite{nishiyama}, or by enumerating
  the lattices in the genus, one finds that there are $25$ possible
  essential lattices for elliptic fibrations on $K$.  When
  the root sublattice of an essential lattice is not saturated in it,
  there is an isogeny to another K3 surface $K'$, and we have also shown that
  $K'$ is covered by the square of a curve.  This occurs for $8$ of the
  $25$ essential lattices.  This can be continued indefinitely, finding
  elliptic fibrations with torsion on the quotient surfaces and passing
  to the quotient.  For the state of the art in
  determining the elliptic fibrations on a K3 surface we refer the reader to
  \cite{fv}.
\end{remark}
  
It turns out that the discriminant of the Picard lattice of the quotient of
$K$ by the $2$-torsion translation is much smaller than that for $K$, which
makes the quotient easier to work with in some ways.  In particular, it
has only a few types of genus-$1$ fibration, as we will see in Remark
\ref{rem:quicksand}.

\begin{defn}\label{def:l-l1}
  Let $L$ be the quotient of $K$ by the $2$-torsion translation
  for the fibration of Proposition \ref{prop:other-fib} and let $L'_1$ be its
  Picard lattice.
\end{defn}

\begin{cor}\label{cor:describe-2-isogeny}
  The fibration on $L$ induced by that of Proposition \ref{prop:other-fib} has
  reducible fibres of types $I_4^*, I_6, I_2$, and Mordell-Weil group $\Z/2\Z$.
  In particular $\disc L'_1 = -12$.
\end{cor}

\begin{proof} The quotient of the given $I_2^*$ is of type $I_4^*$, since
  the $2$-torsion section passes through the reduced component that meets
  the same nonreduced component as the zero section.
  The $I_3$ fibre produces an $I_6$ on the quotient, and the 
  Shioda-Tate formula again shows that there must be one more reducible fibre,
  which can only arise from an $I_1$ fibre of $K$ (computing the Euler
  characteristic shows that there is exactly one such fibre).

  Since $K$ has a $2$-torsion point, so does the $2$-isogenous surface $L$.
  On the other hand, it is easily checked that
  $\H + D_8 + A_5 + A_1$ is not a sublattice of index $4$ in
  any even lattice.  Thus $\disc \Pic L = -4 \cdot 6 \cdot 2 / 2^2 = -12$.
\end{proof}

\begin{prop}\label{prop:has-e8-e6}
  The surface $L$ admits a fibration in curves of genus~$1$
  with no section and reducible fibres of type $II^*, IV^*$.
\end{prop}

\begin{proof} We define an $II^*$ fibre whose support consists of the
  zero section and all curves in the $I_4^*$ fibre except for the reduced
  curve adjacent to the zero component, and an $IV^*$ fibre whose support
  consists of the $2$-torsion section and the nonzero components of the
  $I_2$ and $I_6$ fibres.  It is readily checked that these sets of curves
  have the correct topology to support the given fibres and that these fibres
  are linearly equivalent.  There cannot be a section, for if there were the
  discriminant of the Picard lattice would be $-3$ rather than $-12$.
\end{proof}

\begin{lemma}\label{lem:motive-finite-is-invariant}
  The property of having finite motive is invariant under maps of finite
  degree of K3 surfaces.
\end{lemma}

\begin{proof} Let $S \to T$ be such a map.  If $S$ has finite motive, then
  $T$ is covered by the same power of a curve as $S$.  The other direction
  is proved in the course of proving \cite[Theorem 3.1]{laterveer}.
\end{proof}

\begin{cor}\label{cor:finite-e8-e6}
  A general K3 surface with Picard lattice isometric to $L'_1$ or $\H + E_8 + E_6$
  is covered by curves, and hence has finite-dimensional motive.
\end{cor}

\begin{proof} We have seen that a general K3 surface with Picard lattice $L'_1$
  is covered by one with Picard lattice $L_1$ and hence by curves.  In turn,
  there is a primitive class of self-intersection~$0$ in $L_1$ with even
  intersection with all of $L_1$, and this class can be chosen to be
  that of a curve of genus~$1$ by a standard result
  \cite[Corollary 8.2.9]{huybrechts}.  The Picard lattice of the Jacobian
  of the associated fibration is then an even lattice containing $L'_1$
  with index $2$.
  Using the results of Nikulin it is easy to check that this lattice is
  isometric to $\H + E_8 + E_6$.

  The result now follows from
  Clearly the map from the moduli space of K3 surfaces with marked Picard
  lattice $L_1$ to that of surfaces with lattice $\H + E_8 + E_6$
  has finite fibres and is therefore surjective.  We have shown in
  Section \ref{sec:constr-moduli} that a general K3 surface with Picard
  lattice $L_1$ has finite-dimensional motive, so the result now follows
  from Lemma \ref{lem:motive-finite-is-invariant}.
\end{proof}

\begin{remark}\label{rem:quicksand}
  The surfaces with Picard lattice isometric to $\H + E_8 + E_6$ are
  very convenient for Nishiyama's method for enumerating the types
  of elliptic fibration on a K3 surface\cite{nishiyama}.  There is
  an embedding of $E_8 + E_6$ into $E_8^3$ with complement $E_8 + A_2$. 
  The only Niemeier lattices into which $E_8 + A_2$ can be embedded are those
  with an $E_8$ component, and one readily checks that the embedding of $A_2$
  into the complement is unique up to isometry in both cases.  Thus
  there are only two types of elliptic fibration on such a surface;
  their Mordell-Weil ranks are $0$ and $1$.
  It is noteworthy that neither type admits nontrivial torsion, and so
  there is no obvious way to map such a surface to any K3 surface not
  isomorphic to it; we do not know whether such a map exists.

  The surface $L$ admits elliptic fibrations with $8$ different
  essential lattices, in addition to $2$ types of genus-$1$ fibration without
  a section.  The only type with nontrivial torsion is the one arising from
  the description of $L$ as a quotient of $K$, so there are no obvious
  maps from $L$ to other K3 surfaces that do not factor through a surface
  with Picard lattice isometric to $\H + E_8 + E_6$ or that of $K$.
\end{remark}

\begin{cor}\label{thm:hodge-conj-l-surfaces}
The Hodge conjecture is true for $Y\times Y$, where $Y$ is a general K3 surface with Picard lattice isometric to $L'_1$ or $\H + E_8 + E_6$.
\end{cor}
\begin{proof}
  This is an immediate consequence of the proof of Corollary
  \ref{cor:finite-e8-e6} and Theorem \ref{thm:hodge-conj-l2-surfaces}.
\end{proof}
  
These observations allow us to relate our construction to certain families of
K3 surfaces with $15$ nodes studied in \cite{g-s}.  In particular, in the
notation of \cite[Theorem 8.3]{g-s}, we consider the case $d = 3$.  First
we recall the definition of the lattice $\Mz$, which is of fundamental
importance in \cite{g-s}.

\begin{defn}\label{def:groups}
  Let $A = (\Z/2\Z)^4$, let $N$ be the set of nonzero
  elements of $A$, and let $M_0$ be the lattice $A_1^{15}$ with basis $b_i$
indexed by $N$.  For every subset $C \subset N$ of order $8$ which is the
complement of a subgroup, we adjoin the vector $\sum_{i \in N} b_i/2$ to $M_0$.
The result of this is the lattice $\Mz$ of discriminant $-128$.
\end{defn}

\begin{defn}\label{def:n1-n2}
We now consider the lattice $\langle 6 \rangle + \Mz$, whose first
generator will be denoted $h$.  It can be
enlarged by adjoining a vector $(h,v)/2$, where $v = \sum_{i \in G \setminus \{0\}}$
for $G$ a subgroup of order $4$, or by adjoining $(h,w)/2$, where
$w = \sum_{i \in N} b_i$.  Let us denote these two lattices by $N_1, N_2$
(for definiteness, and following \cite{g-s}, we use the subgroup of elements
with first two components $0$ to define $N_1$).
\end{defn}

\begin{remark}\label{rem:why-n1-n2}
  The lattices $N_1, N_2$ are the two lattices on the list of Garbagnati and Sarti
  \cite[Theorem 8.3]{g-s} of possible Picard lattices of K3 surfaces with $15$
  nodes for the case $d = 3$.  Note that these lattices are not isometric
  to $L_1$, although they have the same rank and discriminant.
\end{remark}

We will prove the following theorem.

\begin{thm}\label{thm:s-to-e8-e6}
  Let $S$ be a K3 surface with Picard lattice isometric to
  $N_1$ or $N_2$.  Then $S$ admits a finite map to a K3 surface with Picard
  lattice $\H + E_8 + E_6$.
\end{thm}

Before proving Theorem \ref{thm:s-to-e8-e6}, we first note a simple corollary.

\begin{cor}\label{cor:finite-15-nodes}
  Let $S$ be a general
  K3 surface of degree $6$ with $15$ singularities of type $A_1$.
  Then $S$ is motive-finite.
\end{cor}

\begin{proof} (of corollary)
  First suppose that $S$ has Picard rank $16$.
  Then by \cite[Theorem 8.3]{g-s}, its Picard lattice is
  isometric to $N_1$ or $N_2$.  Combining this with Theorem
  \ref{thm:s-to-e8-e6} and Lemma \ref{lem:motive-finite-is-invariant},
  as in the proof of Corollary \ref{cor:finite-e8-e6}, yields the
  desired result.

  In the general case, we consider a $1$-parameter family of K3 surfaces
  with $15$ singularities of type $A_1$ and Picard rank $16$ whose limit is $S$.
  We invert the map of moduli spaces of Theorem \ref{thm:main-moduli-intro}:
  this is well-defined elsewhere on the curve, so it has a limit at $S$, which
  is generally smooth.  Lifting
  arbitrarily to the cover of $\M_C$ that additionally parametrizes a square
  root of the line bundle $\O(p_1+p_2+p_3+p_4)$ allows us to construct a curve
  of arithmetic genus~$7$ (the stable limit of smooth curves in the family)
  whose square covers $S$; again, it is generically smooth.  
\end{proof}

\begin{remark}\label{rem:expect-covered}
  We expect that these surfaces should be covered by the square of the same
  curves that cover those with Picard lattice $L_1$, and that this can be
  proved by finding chains of maps of finite degree
  leading from surfaces with Picard
  lattice $L_1$ to surfaces with Picard lattice $N_1$ and $N_2$.
\end{remark}

In order to prove Theorem \ref{thm:s-to-e8-e6},
we will identify some rational curves on $N_1, N_2$
and use them to construct genus~$1$ fibrations without a section, whose
Jacobians can thus be used to move toward surfaces with discriminant $-3$.

\begin{defn}\label{def:ci-classes}
  For $S = N_1$ or $N_2$,
  let $H$ be a vector on the boundary of the ample cone of $S$
  that is in the orbit of $\pm h$ under the reflection group (as in, for
  example, \cite[Corollary 8.2.9]{huybrechts}).  Further,
  let the $C_i$ be the images of the $b_i$ under the reflection.
\end{defn}

The map to projective space given by the linear system $|H|$ exhibits
$S$ either as a complete intersection of hypersurfaces of degree $2, 3$
in $\P^4$ or as a double cover of a cubic scroll.  In
either case, every Picard class has intersection a multiple of $3$ with the
hyperplane class, and the curves of class $C_i$ are nodes.

\begin{defn}\label{def:s-n1}
  Let $S_1$ be a K3 surface with Picard lattice $N_1$.
\end{defn}

\begin{prop}\label{prop:cubics-n1}
  There are at least $15$ smooth rational curves
  on the minimal desingularization of $S_1$
  whose intersection with the strict transform of $H$ is~$3$.
\end{prop}

\begin{proof}
  For every subgroup $G \subset A$ of order $8$, consider the Picard class
  $C_G = (H - \sum_{i \in G \setminus \{0\}} C_i)/2 \in N_1$.  It has
  self-intersection~$-2$ and
  positive intersection with $H$, so it is effective.  Repeatedly subtracting
  the classes of the nodes with negative intersection from $C_G$ until there
  are none left, we obtain classes $C_G'$ of self-intersection
  $-2$ and intersection~$3$
  with $H$.  These must represent irreducible curves: by Riemann-Roch they
  are effective, and if they were reducible one of the components would have
  degree $0$ and negative intersection with $C_G'$, which is not possible since
  the only curves of degree $0$ are the $15$ nodes.  One checks that the
  $C_G'$ are all distinct.
\end{proof}

\begin{prop}\label{prop:nosec-fibre}
  Fix a subgroup $G \subset A$ of order $4$ and let
  $F = H - \sum_{i \in G \setminus \{0\}} C_i$.  Then $F$ is not divisible by
  $2$, but $F$ has even intersection with every curve on $S_1$.
\end{prop}

\begin{proof}
  This is a simple calculation.
\end{proof}

\begin{prop}\label{prop:enlarge-pic}
  Let $S$ be a K3 surface and $v \in \Pic S$ a vector not divisible by
  any integer $n > 1$ such that $(v,v) = 0$ and $d \mid (v,w)$ for all
  $w \in \Pic S$.  Then there is a map from $S$ to a K3 surface with Picard
  lattice $\Pic S[v/d]$.
\end{prop}

\begin{proof}
  By \cite[Corollary 8.2.9]{huybrechts}, there is a
  sequence of reflections whose product $\rho$ is such that $\pm \rho(v)$
  is nef.  It thus suffices to assume that $v$ is nef.  In this case it is
  the class of a fibre of a genus~$1$ fibration with minimal multisection
  degree $d$, and the map from this fibration to its Jacobian
  (\cite[Section 11.4]{huybrechts}) given on smooth fibres by
  $P \to dP - F \cap M_d$, is the desired map, where $F$ is the fibre
  containing $P$ and $M_d$ is the multisection.
  (The description
  of the Picard lattice of the target follows from the discussion after
  \cite[Corollary 11.4.7]{huybrechts}.)
\end{proof}

\begin{remark}\label{not-paranjape}
  This proposition cannot be used to prove the existence
  of a map between K3 surfaces whose Picard lattices are not isometric after
  tensoring with $\Q$.  In particular it cannot be used to construct a map
  from a surface with Picard lattice $L_1$ to one with Picard lattice
  $\H + E_8 + E_6$ directly: the Hasse-Minkowski invariants
  are different.  Indeed, for every map between such K3 surfaces of degree $k$,
  we must have $v_2(k) + v_3(k)$ odd, where $v_p$ denotes the $p$-adic
  valuation.  This follows from much the same reasoning as is used to
  prove \cite[Theorem 1.1]{bsv}.  On the other hand, the map
  described above, taking $P$ to $dP - F \cap M_d$, has degree $d^2$, and
  of course $v_2(k^2) + v_3(k^2)$ is even.
\end{remark}

\begin{cor}\label{cor:n1-works}
  A K3 surface with Picard lattice $N_1$ admits a finite map
  to a surface with Picard lattice $\H + E_8 + E_6$.
\end{cor}

\begin{proof}
  We take $G_1, G_2$ to be subgroups whose pairwise intersections
  with each other and with the subgroup $G$ used to define $N_1$ have order $2$,
  but such that $G \cap G_1 \cap G_2 = \{0\}$.  Again one defines
  $F_i = H - \sum_{g \in G_i} C_i$ and verifies that $(F_i,F_i) = 0$ and
  $F_1,F_2 = 4$, so that we obtain a map from the given surface to one with
  Picard lattice $N_1[F_1/2,F_2/2]$.  It is a straightforward matter to find
  a vector in this lattice of norm $0$, not divisible by $2$, and having even
  intersection with all vectors: for example, letting $G_1$ and $G_2$ be
  generated by $C_{0001},C_{0100}$ and $C_{0010},C_{0100}$, we may take
  $2H - 2C_{0111} - 2C_{1011} - \sum_{i,j=0}^1 C_{11ij}$.
  The rest of the proof is the same as for Corollary~\ref{cor:finite-15-nodes}.
\end{proof}

It turns out that $N_2$ is more manageable than $N_1$.
\begin{cor}\label{cor:n2-works}
  A K3 surface with Picard lattice $N_2$ admits a finite map
  to a surface with Picard lattice $\H + E_8 + E_6$.
\end{cor}

\begin{proof}
  Let $G_1, G_2, G_3$ be subgroups of $A$ of order $4$ with pairwise
  intersection of order $2$ whose union generates $A$ (for example,
  let $a_i$ be the generators of $A$ and take $G_i = \langle a_i,a_4 \rangle$).
  Let $F_1, F_2, F_3$ be the corresponding vectors $F$ as in Proposition
  \ref{prop:nosec-fibre}: we then have $(F_i,F_j) = 4$ for $i \ne j$, so that
  $F_2$ is still even in $N_1[F_1/2]$ and $F_3$ in $N_1[F_1/2,F_2/2]$.
  Thus we may apply Proposition \ref{prop:enlarge-pic} successively to the $F_i$,
  obtaining lattices of discriminant $-48, -12, -3$.  To see that the lattice
  of discriminant $-3$ is isometric to $\H + E_8 + E_6$, it suffices
  to compute the discriminant group and apply a result of Nikulin
  \cite[Corollary 2.10 (ii)]{morr}.  Alternatively this can be checked
  directly by embedding $\H$ into $N_2[F_1/2,F_2/2,F_3/2]$ and checking that
  the orthogonal complement is one of the two lattices
  in the same genus as $E_8 + E_6$.
\end{proof}

More generally, we can use Proposition \ref{prop:enlarge-pic} to clarify the
relation among the different families of K3 surfaces with $15$ nodes considered
by Garbagnati and Sarti.  Recall that in \cite[Theorem 8.3]{g-s} they give
a complete description of all rank-$16$ lattices that can be the Picard
lattice of such a surface.  In particular, these lattices contain the lattice
$\langle 2d \rangle + \Mz$ with index $2$; there are $2$ possibilities if
$d \equiv 3 \bmod 4$ and $1$ otherwise.  If $d/d' \notin (\Q^*)^2$, then there
can be no relation between the surfaces corresponding to $d$ and $d'$.
This can be seen from Theorem \ref{thm:trans-isogeny}: since the rank of the Picard
group is even, we have $\disc T(S)(n)/\disc T(S) \in (\Q^*)^2$, and the same conclusion
follows for $\disc T(S')/\disc T(S)$.  On the other hand, we show:

\begin{prop}\label{square-related}
  Let $d, d'$ be positive integers with $d/d' \in (\Q^*)^2$ and let $L_d, L_{d'}$
  be lattices from the list in \cite[Theorem 8.3]{g-s} of possible Picard
  lattices of K3 surfaces with $15$ singularities of type $A_1$.
  Let $S_d$ be a K3
  surface with Picard lattice $L_d$.  Then there is a correspondence between
  $S_d$ and a surface $S_{d'}$ with Picard lattice $L_{d'}$: more precisely,
  there is a sequence of finite maps $(\pi_i)_{i=0}^n$ of K3 surfaces such that
  the domain of $\pi_0$ is $S_d$, the codomain of $\pi_n$ is $S_{d'}$,
  and for all $i$ with $0 \le i < n$ either the domains or the codomains of
  $\pi_i, \pi_{i+1}$ coincide.
\end{prop}

\begin{proof} It suffices to prove this under the assumption that $d/d'$ is
  the square of a prime $p$, since the existence of a correspondence as in
  the statement of the proposition is an equivalence relation.  We begin with
  a bit of notation.  In the lattice $\langle 2d \rangle + \Mz$ that is
  contained in $L_d$ with index $2$, let $\ell_d$ be the first generator,
  of norm $2d$ and orthogonal to $\Mz$; recall that we designate the generators
  of $A_1^{15}$ and their images in $\Mz$ by $b_i$.

  We first dispose of the easy case of odd $p$.  We may assume, if
  $d \equiv 3 \bmod 4$, that we are in the same case (iii) or (iv) for both
  $d$ and $d'$: otherwise, we reduce separately to $d = d'$ squarefree and
  relate both lattices to $L_{4d}$ as shown below.
  In this case $L_{d'} = L_d[\ell_d/p]$.  Since $\Mz$
  contains $4$ pairwise orthogonal elements of norm $2$, it suffices to
  write $d'/2$ as a sum $\sum_{i=1}^4 a_i^2$ of $4$ squares and to note that
  $\ell_d - p\sum_{i=1}^4 a_i b_{g_i}$ is a vector of norm $0$ to which
  Proposition \ref{prop:enlarge-pic} applies.  Thus
  $L_d[(\ell_d-p\sum_{i=1}^4 a_i b_{g,i})/p] = L_d[\ell_d/p]$ is the Picard
  lattice of a K3 surface that admits a finite map from the surface with
  Picard lattice $L_d$.  The same argument works for $p = 2$ if $4 | d'$.

  Now we consider the cases with $p = 2$, and in particular the cases
  (i), (ii), (iii) of \cite[Theorem 8.3]{g-s} (we will treat (iv) separately
  and (v) was taken care of just above).  In each of these cases it turns out
  that $L_d[\ell_d/2]$ is isometric to $L_{d'}$.  This is most easily seen in
  terms of the descriptions of the discriminant forms at the beginning of the
  proof of \cite[Theorem 8.3]{g-s}.  Indeed, the vector $\ell_d$ corresponds to
  the generator $1/2d$ of $q_2 + q_2 \langle 1/2 \rangle + \langle 1/2d\rangle$,
  so the discriminant form of $L_d[\ell_d/2]$ is
  $q_2 + q_2 + \langle 1/2 \rangle + \langle 4/2d \rangle$, the same as that of $L_{d'}$.
  By \cite[Corollary 2.10 (ii)]{morr}, the Picard lattice is uniquely
  determined by its signature and discriminant form in this situation, so it
  must be that $L_d[\ell_d/2] \cong L_{d/2}$.

  In the case (iv), no overlattice
  of $L_d$ is isometric to $L_{d'}$.  Nevertheless, we may proceed
  as follows.  First we divide $\ell_d$ by $2$ as before to obtain the
  lattice $L_{iii,d'}$ of case (iii); as before, this comes from a map of K3
  surfaces. 

  To complete the proof, we show that there is a common overlattice $M_{d'}$,
  containing both of the two lattices $L_{iii,d'}, L_{d'}$ of discriminant $-64d$
  with index $2$, that can be reached by dividing vectors of
  norm $0$ by $2$.  In both cases we will obtain it by dividing a vector
  congruent to $x = \sum_{i \in (\Z/2\Z)^4 \setminus H} b_i$ by $2$, where $H$ is the
  subgroup of order $4$ used to define $L_{iii,d'}$.  This vector has norm
  $-24$ and so the vector $2\ell_{d'} - b_i - x$ has norm
  $8(d'-5)$.  Since $d'-5 \equiv 2 \bmod 4$, it is a sum of three squares,
  and so $-8(d'-5)$ is the norm of some integral combination of the
  $2b_h$ (whose norms are $-8$) for $h \in H \setminus \{0\}$, except in the
  case $d' = 3$ which can be treated directly as in Corollary \ref{cor:n2-works}
  above or by replacing $2\ell_{d'}$ by $6\ell_{d'}$.

  Thus, the K3 surfaces with Picard lattice $L_{d'}$ are in correspondence with
  those of Picard lattice $M_{d'}$, and then with those of Picard lattice
  $L_{iii,d'}$ and $L_d$ as claimed.
\end{proof}

The proof of this proposition provides a proof of Proposition
\ref{prop:g-s-isogeny}.
We close with the proof of Theorem \ref{thm:hodge-conj-gs-surfaces}.
\begin{proof} By Proposition \ref{square-related}, we have a correspondence
  between $K$ and any
surface $S'$ with Picard lattice $L_{d'}$ with $\frac 2{d'}\in {\Q^*}^2$.
This gives an isomorphism of Hodge structures $T(K)\cong T(S')$.
Thus Theorem \ref{thm:hodge-conj-gs-surfaces} follows from
\cite[Theorem 2]{lombardo} in the same way as in the proof of Theorem
\ref{thm:hodge-conj-l2-surfaces}.
\end{proof}

\end{document}